\documentclass[review]{siamltex}
\usepackage{subcaption}
\usepackage{csquotes}

\usepackage{url}
\usepackage{caption}
\usepackage{mathtools}
\usepackage{bbm}
\usepackage{enumitem}
\usepackage{esvect}
\usepackage{centernot}
\usepackage{float} 
\usepackage{fullpage}
\usepackage{amsfonts}
\usepackage{latexsym}
\usepackage{amssymb}
\usepackage{color}
\usepackage[inline]{showlabels}
\usepackage{graphicx}
\graphicspath{{./figures/}}
\usepackage{epstopdf}
\usepackage{algorithm,algpseudocode}

\usepackage{siunitx}
\usepackage{mathrsfs}
\usepackage{threeparttable}
\usepackage{textcomp}
\usepackage{booktabs}
\usepackage{tikz}
\usetikzlibrary{positioning,angles,quotes}

\usepackage{makecell}

\usepackage{algorithm,algpseudocode}

\usepackage{tikz}
\usetikzlibrary{bayesnet}
\usetikzlibrary{arrows}
\usetikzlibrary{decorations.pathreplacing}

\numberwithin{equation}{section}

\allowdisplaybreaks
\def\vgap{\vspace*{.1in}}
\newcommand{\ts}{\textsuperscript}

\newcommand{\E}{\mathbb{E}}

\newcommand{\R}{\mathbb{R}}

\newcommand{\La}{\mathcal{L}}

\def \grad {\nabla}
\DeclareMathOperator*{\argmax}{arg\,max}
\DeclareMathOperator*{\argmin}{arg\,min}

\newcommand{\inner}[2]{\langle {#1,#2} \rangle}
\NewDocumentCommand\norm{ m O{}}{\left\lVert #1 \right\rVert_{#2}}
\NewDocumentCommand\normsq{ m O{}}{\left\lVert #1 \right\rVert^2_{#2}}

\newcommand{\textblue}[1]{{ #1}}

\DeclarePairedDelimiter{\ceil}{\lceil}{\rceil}

\newcommand{\tsum}{\textstyle \sum}

\newcommand{\ep}{\epsilon}
\newcommand{\lam}{\lambda}

\makeatletter
\newcommand{\algmargin}{\the\ALG@thistlm}
\makeatother
\newlength{\whilewidth}
\algdef{SE}[parWHILE]{parWhile}{EndparWhile}[1]
  {\parbox[t]{\dimexpr\linewidth-\algmargin}{%
     \hangindent\whilewidth\strut\algorithmicwhile\ #1\ \algorithmicdo\strut}}{\algorithmicend\ \algorithmicwhile}%
\algnewcommand{\parState}[1]{\State%
  \parbox[t]{\dimexpr\linewidth-\algmargin}{\strut #1\strut}}

\def \bigO {\mathcal{O}}

\usepackage{pifont}

\usepackage{tcolorbox}

\usepackage{empheq}

\usepackage{enumitem}
\setlist{itemsep=0pt}
\usepackage{accents}
\NewDocumentCommand\utilde{m}{\underaccent{\tilde}{#1}}

\newcommand{\addt}[1]{#1^t}
\newcommand{\addtt}[1]{#1^{t-1}}
\newcommand{\addttt}[1]{#1^{t-2}}
\newcommand{\addp}[1]{#1^{t+1}}
\newcommand{\addN}[1]{#1^N}

\newcommand{\addO}[1]{#1^0}

\newcommand{\sadt}[1]{#1_t}

\NewDocumentCommand\bci{O{i}}{#1}
\newcommand{\ave}{{\text{av}}}

\NewDocumentCommand\SubAddi{m m}{#1_{\bci[#2]}}

\NewDocumentCommand\SubAddI{m m}{#1_{#2}}
\NewDocumentCommand\Addit{m O{\bci}}{\addt{\SubAddI{#1}{#2}}}
\NewDocumentCommand\Additt{m O{\bci}}{\addtt{\SubAddI{#1}{#2}}}
\NewDocumentCommand\Addittt{m O{\bci}}{\addttt{\SubAddI{#1}{#2}}}
\NewDocumentCommand\AddiN{m O{\bci}}{\addN{\SubAddI{#1}{#2}}}
\NewDocumentCommand\Addip{m O{\bci}}{\addp{\SubAddI{#1}{#2}}}
\NewDocumentCommand\AddiO{m O{\bci}}{\addO{\SubAddI{#1}{#2}}}
\NewDocumentCommand\Addav{m O{\ave}}{\SubAddI{#1}{#2}}
\NewDocumentCommand\Addavt{m O{\ave}}{\addt{\SubAddI{#1}{#2}}}
\NewDocumentCommand\Addavtt{m O{\ave}}{\addtt{\SubAddI{#1}{#2}}}
\NewDocumentCommand\Addavttt{m O{\ave}}{\addttt{\SubAddI{#1}{#2}}}
\NewDocumentCommand\AddavN{m O{\ave}}{\addN{\SubAddI{#1}{#2}}}
\NewDocumentCommand\Addavp{m O{\ave}}{\addp{\SubAddI{#1}{#2}}}
\NewDocumentCommand\AddavO{m O{\ave}}{\addO{\SubAddI{#1}{#2}}}

\def \targmax {{\textstyle\argmax}}
\def \targmin {{\textstyle\argmin}}
\def \tmax {{\textstyle\max}}
\def \tmin {{\textstyle\min}}

\def \sumi {\tsum_{i=1}^m}
\NewDocumentCommand\nSt{O{t}}{{{S_{#1}}}}
\def \St {{S_t}}
\def \sums {\tsum_{s=1}^{S_t}}

\def \sumt {\tsum_{t=1}^N}

\NewDocumentCommand\sumT{O{N}}{\tsum_{t=1}^{#1}}

\NewDocumentCommand\qs{O{s}}{q_{#1}}

\NewDocumentCommand\etat{O{t}}{\eta_{#1}}
\def \etatt {\etat[t-1]}

\def \etatt {\eta_{t-1}}

\def \gam {\gamma}
\def \gampi {\gam_\pi}
\def \gamA {\gam_A}

\NewDocumentCommand\taut{O{t}}{\tau_{#1}}
\def \tautt {\taut[t-1]}
\def \thetat {\sadt{\theta}}

\NewDocumentCommand\wt{O{t}}{\omega_{#1}}
\NewDocumentCommand\epxt{O{t}O{x}}{\ep^{#1}_{#2}}
\NewDocumentCommand\epuxt{O{t}O{\underline{x}}}{\ep^{#1}_{#2}}
\NewDocumentCommand\eppt{O{t}}{\epxt[#1][p]}
\def \wtt {\wt[t-1]}

\def \Rx {R_{0}}
\def \Nep {N_\ep}
\def \Lf {L_f}

\NewDocumentCommand\Lfi{O{}}{L_{f_{#1}}}

\def \Mtil {\tilde{M}}

\def \MAPi {\Mtil_{A\Pi}}

\def \bMAPiU {\MAPi}

\NewDocumentCommand\MtU{O{t}}{M_{{#1}}}
\NewDocumentCommand\bMtU{O{t}}{\bar{M}_{{#1}}}
\def \MttU {M_{{t-1}}}
\NewDocumentCommand\MPiu{O{U^*}}{M_{\Pi, #1}}
\NewDocumentCommand\MPi{O{i}}{M_{\Pi_{#1}}}

\NewDocumentCommand\DX{O{}}{D_{X^0}}

\def \DP {D_P}
\NewDocumentCommand\Rt{O{t}}{R_{{#1}}}
\NewDocumentCommand\gi{O{i}}{\SubAddi{g}{#1}}

\def \dom {\rm dom}

\def \Lfbar {\bar{L}_f}

\NewDocumentCommand\eptl{O{}}{\tilde{\ep}_{#1}}

\NewDocumentCommand\MXtlt{O{t}}{M_{\tilde{X}}^{#1}}

\NewDocumentCommand\MPitlt{O{t}}{M_{\tilde{\Pi}}^{#1}}
\NewDocumentCommand\fistar{O{i}}{f^*_{#1}}
\NewDocumentCommand\fitilstar{O{i}}{\tilde{f}^*_{#1}}
\NewDocumentCommand\Dfistar{O{i}}{W_{\fistar[#1]}}
\NewDocumentCommand\fistarp{O{i}}{(f^*_{#1})'}
\NewDocumentCommand\ftlit{O{t}O{i}}{\tilde{f}_{#2}^{#1}}
\NewDocumentCommand\futit{O{t}O{i}}{\utilde{f}_{#2}^{#1}}
\def \pmbf {\pmb{f}}
\def \Vi {V_i}

\NewDocumentCommand\myfi{O{\bci}}{f_{#1}}

\NewDocumentCommand\vt{O{t}}{v^{#1}}
\NewDocumentCommand\vit{O{t}O{i}}{v_{#2}^{#1}}
\def \vitt {\vit[t-1]}

\NewDocumentCommand\Ai{O{i}} {A_{#1}}
\NewDocumentCommand\Aitr{O{i}}{A_{#1}^\top}
\def \MA {M_A}
\def \DPi {D_\Pi}
\def \rhostar {\rho^*}
\def \zt {\addt{z}}
\def \hz {\hat{z}}
\def \zbar {\bar{z}}
\def \zstar {z^*}

\NewDocumentCommand\qts{O{s} O{t}}{q_{#1}^{#2}}
\NewDocumentCommand\betat{O{t}}{\beta^{#1}}
\NewDocumentCommand\betats{O{s} O{t}}{\beta_{#1}^{#2}}
\NewDocumentCommand\gamt{O{t}}{\gamma^{#1}}
\NewDocumentCommand\gams{O{s}}{\gamma_{#1}}
\NewDocumentCommand\gamts{O{s} O{t}}{\gamma_{#1}^{#2}}
\NewDocumentCommand\delts{O{s} O{t}}{\delta_{#1}^{#2}}
\NewDocumentCommand\delt{O{t}}{\delta^{#1}}

\def \hp {\hat{p}}
\NewDocumentCommand\ptlt{O{t}}{\tilde{p}^{#1}}
\NewDocumentCommand\pdbtlt{O{t}}{\dbtilde{p}^{#1}}

\NewDocumentCommand\bpt{O{t}}{\bp^{#1}}
\NewDocumentCommand\hpiN{O{i}}{\hat{p}^N_{#1}}
\def \hpN {\hpiN[]}
\NewDocumentCommand\hpi{O{i}}{\hp_{#1}}
\NewDocumentCommand\bpi{O{i}}{\bp_{#1}}
\def \pstar {p^*}
\def \bp {\bar{p}}

\NewDocumentCommand\pt{O{t}}{\myp^{#1}}
\def \ptt {\addtt{\myp}}

\def \ptltt {\ptlt[t-1]}
\def \pdbtltt {\pdbtlt[t-1]}

\NewDocumentCommand\pitl{O{i}}{\tilde{\pmb{p}}_{\bci[#1]}}
\NewDocumentCommand\pit{O{t}O{i}}{{p}^{#1}_{\bci[#2]}}
\NewDocumentCommand\putit{O{t}O{i}}{\utilde{p}^{#1}_{\bci[#2]}}
\NewDocumentCommand\putt{O{t}}{\putit[#1][]}
\NewDocumentCommand\mypi{O{i}}{p_{\bci[#1]}}
\def \myp {p}

\NewDocumentCommand\qit{O{t}O{i}}{q^{#1}_{#2}}

\NewDocumentCommand\pii{O{i}}{\pi_{\bci[#1]}}
\NewDocumentCommand\bpii{O{i}}{\bar{\pi}_{\bci[#1]}}
\NewDocumentCommand\Pii{O{i}}{\Pi_{\bci[#1]}}
\NewDocumentCommand\piit{O{t}O{i}} {\pii[#2]^{#1}}
\NewDocumentCommand\piiutt{O{t}O{i}} {\utilde{\pi}_{#2}^{#1}}
\NewDocumentCommand\utpiit{O{t}O{i}} {\utilde{\pi}_{#2}^{#1}}
\NewDocumentCommand\piitt{O{i}} {\piit[t-1][#1]}
\NewDocumentCommand\piibrt{O{t}}{\{\pii^{#1}\}}
\NewDocumentCommand\hpii{O{i}}{\hat{\pi}_{\bci[#1]}}
\NewDocumentCommand\piistar{O{i}}{\pi^*_{\bci[#1]}}
\NewDocumentCommand\hpiiN{O{i}}{\hat{\pi}^N_{\bci[#1]}}

\NewDocumentCommand\piibar{O{i}}{\bar{\pi}_{#1}}

\def \uy {\underline{y}}

\NewDocumentCommand\yt{O{t}}{y^{#1}}

\NewDocumentCommand\yit{O{t}O{i}}{y^{#1}_{#2}}
\NewDocumentCommand\ytlitj{O{t}O{i}O{j}}{\tilde{y}^{#1}_{#2\setminus\{#3\}}}
\NewDocumentCommand\yutlit{O{t}O{i}}{\utilde{y}^{#1}_{#2}}
\NewDocumentCommand\utyit{O{t}O{i}}{\utilde{y}^{#1}_{#2}}
\NewDocumentCommand\ytlit{O{t}O{i}}{\tilde{y}^{#1}_{#2}}
\NewDocumentCommand\uyit{O{t}O{i}}{\uy_{#2}^{#1}}

\def \blx {\pmb{x}} 
\NewDocumentCommand\blxt{O{t}}{\blx^{#1}}
\def \xstar {x^*}
\NewDocumentCommand\xt{O{t}}{x^{#1}}
\def \xtt {\addtt{x}}

\def \xunder {\underline{x}}
\NewDocumentCommand\xundert{O{t}}{\xunder^{#1}}
\NewDocumentCommand\blxundert{O{t}}{\underline{\blx}^{#1}}
\NewDocumentCommand\xunderit{O{t}O{i}}{\xunder^{#1}_{#2}}
\NewDocumentCommand\xunderavt{O{t}O{i}}{\xunder^{#1}_{\ave}}

\NewDocumentCommand\myxi{O{i}}{x_{#1}}
\NewDocumentCommand\xit{O{t}O{i}}{x^{#1}_{#2}}
\NewDocumentCommand\xitt{O{i}}{\xit[t-1][#1]}
\NewDocumentCommand\xavt{O{t}}{\xit[#1][\ave]}

\def \hx {\hat{x}}
\NewDocumentCommand\hxit{O{t}O{i}}{\hat{x}^{#1}_{#2}}
\NewDocumentCommand\hblxit{O{t}O{i}}{\hat{\pmb{x}}^{#1}_{#2}}

\NewDocumentCommand\xtilt{O{t}}{\tilde{x}^{#1}}
\NewDocumentCommand\blxtlt{O{t}}{\tilde{\blx}^{#1}}
\NewDocumentCommand\xtlit{O{t}O{i}}{\tilde{x}^{#1}_{#2}}
\NewDocumentCommand\xtlavt{O{t}}{\xtlit[#1][\ave]}

\def \barx {\bar{x}}
\def \xbar {\bar{x}}
\def \xbarN {\xbart[N]}
\NewDocumentCommand\bxit{O{t}O{i}}{\bar{x}^{#1}_{#2}}

\NewDocumentCommand\xbart{O{t}}{\bar{x}^{#1}}

\NewDocumentCommand\dit{O{t}O{i}}{d^{#1}_{#2}}
\NewDocumentCommand\Nit{O{t}O{i}}{\mathcal{N}^{#1}(#2)}
\NewDocumentCommand\Nitj{O{t}O{i}O{j}}{\mathcal{N}^{#1}(#2;\{#3\})}
\NewDocumentCommand\Ni{O{i}}{\mathcal{N}(#1)}
\NewDocumentCommand\Hij{O{i,j}}{H_{#1}}
\NewDocumentCommand\Wij{O{i,j}}{W_{#1}}
\NewDocumentCommand\sumNi{O{i}}{\tsum_{j \in \Ni[#1]} \Hij[j, #1]}

\def \blw {\pmb{w}}
\NewDocumentCommand\blwt{O{t}}{\blw^{#1}}
\NewDocumentCommand\ublwt{O{t}}{\underline{\blw}^{#1}}
\NewDocumentCommand\bblwt{O{t}}{\bar{\blw}^{#1}}
\NewDocumentCommand\tilwt{O{t}}{\tilde{\omega}^{#1}}

\NewDocumentCommand\Q{O{}}{Q_{#1}}

\def \bg {\bar{g}}

\NewDocumentCommand\delx{O{x}}{\delta_{#1}}
\NewDocumentCommand\delpit{O{t}}{\delta^{#1}_\pi}
\NewDocumentCommand\delpt{O{t}}{\delta^{#1}_p}
\NewDocumentCommand\delxt{O{t}}{\delta_x^{#1}}

\def \bu {\bar{u}}
\NewDocumentCommand\uit{O{t}}{u^{#1}_i}

\def \calM {\mathcal{M}}
\NewDocumentCommand\Memt{O{t}}{\calM_{#1}}
\NewDocumentCommand\Mit{O{t} O{i}}{\calM_{#2, #1}}
\NewDocumentCommand\Mpiitl{O{t} O{l} O{i}}{\calM^{\pi, #2}_{{#3},#1}}
\NewDocumentCommand\Mpiit{O{t} O{i}}{\calM^\pi_{{#2}, #1}}
\NewDocumentCommand\Mpiitcp{O{t} O{i}}{\calM^{\pi, \text{cp}}_{{#2}, #1}}
\NewDocumentCommand\Mst{O{t} O{s}}{\calM_{#2, #1}}
\NewDocumentCommand\Mitcp{O{t} O{i} O{}}{\Mit[#1][#2]^{\text{cp} #3}}
\NewDocumentCommand\Mitcm{O{t} O{i}}{\Mit[#1][#2]^{\text{comm}}}
\NewDocumentCommand\Mstcp{O{t} O{s} O{}}{\Mit[#1][#2]^{\text{cp} #3}}
\NewDocumentCommand\Mstcm{O{t} O{s} }{\Mit[#1][#2]^{\text{comm}}}
\NewDocumentCommand\Ki{O{i}}{\mathcal{K}_{#1}}
\NewDocumentCommand\Si{O{i}}{\mathcal{S}_{#1}}
\def \Span {\text{span}}
\NewDocumentCommand\prox{O{}}{\text{prox}_{#1}}
\NewDocumentCommand\Kl{O{l}} {K_{#1}}
\def \Mt {\calM_t}

\def \fstar {f_*}

\def \skipdisplay {\setlength{\abovedisplayskip}{7pt}
\setlength{\belowdisplayskip}{7pt}
\setlength{\abovedisplayshortskip}{7pt}
\setlength{\belowdisplayshortskip}{7pt}}

\NewDocumentCommand\Lfp{O{p}}{L_{f, #1}}

\usepackage{calc}

\usepackage{accents}
\newcommand{\dbtilde}[1]{\accentset{\approx}{#1}}

\def\endproof{{\ \hfill\hbox{%
      \vrule width1.0ex height1.0ex
    }\parfillskip 0pt}\par}

\title{ Optimal methods for \textblue{convex} risk averse distributed optimization 
\thanks{This work is partially supported by the ONR grant N00014-20-1-2089 and the NSF AI Institute grant NSF-2112533.
Coauthors of this paper are listed according to the alphabetic order.}
}
\author{
    Guanghui Lan\thanks{H. Milton Stewart School of Industrial \& Systems Engineering, 
                             Georgia Institute of Technology, Atlanta, GA, 30332 .
                            (email: {\tt george.lan@isye.gatech.edu}).}
    \and Zhe Zhang\thanks{H. Milton Stewart School of Industrial \& Systems Engineering, 
                        Georgia Institute of Technology, Atlanta, GA, 30332 .
                        (email: {\tt jimmy\_zhang@gatech.edu}).}
    }
\date{\today}
 
\begin{document}

\maketitle

\begin{abstract}
This paper studies the communication complexity of convex risk-averse optimization over a  network. 
 The problem generalizes the well-studied risk-neutral finite-sum distributed optimization problem and its importance stems from the need to handle risk in an uncertain environment. For algorithms in the literature, a gap exists in communication complexities for solving risk-averse and risk-neutral problems. We propose two distributed algorithms, namely the distributed risk-averse optimization (DRAO) method and the distributed risk-averse optimization with sliding (DRAO-S) method, to close the gap. Specifically, the DRAO method achieves optimal communication complexity by assuming a certain saddle point subproblem can be easily solved in the server node. The DRAO-S method removes the strong assumption by introducing a novel saddle point sliding subroutine which only requires the projection over the ambiguity set $P$. We observe that the number of $P$-projections performed by DRAO-S is optimal. Moreover, we develop matching lower complexity bounds to show the communication complexities of both DRAO and DRAO-S to be unimprovable. Numerical experiments are conducted to demonstrate the encouraging empirical performance of the DRAO-S method.
 \vspace{.1in}

 \noindent {\bf Keywords:} risk-averse optimization, distributed optimization, first-order algorithm, convex optimization, lower complexity.

 \noindent {\bf AMS 2000 subject classification:} 90C25, 90C15, 68W15, 49M27, 49M29\\
\end{abstract}
\section{Introduction}

Consider the following risk-averse optimization problem over a star-shape (worker-server) communication network \textblue{\cite{star_network}}:
\begin{align}\label{eq:orig_prob}
\begin{split}
\min_{x \in X} \{f(x):= &\max_{p \in P} \sumi p_i f_i(x) -  \rho^*(p) + u(x)\},
\end{split}
\end{align}
where $P \subseteq \Delta^m_+ :=\{p \in \R^m| \tsum_{i=1}^m p_i = 1, p_i \ge 0 \} $ and
$X \subseteq \R^n$ and $\Pii \subseteq \R^{m_i}$ are closed and convex, and functions $\myfi(x)$, $u(x)$, and $\rho^*(p)$ are proper closed and convex.
We assume the scenario (or local) cost function $\myfi$ to be only available to worker node $i$ and focus on the situation where  $\myfi$'s are either all smooth or all structured non-smooth. \textblue{We use the following generic representation for both types of $\myfi$'s:
$$ f_i(x) = \max_{\pii \in \Pii} \inner{A_i x}{\pii} -  \fistar(\pii),$$
where $\Pii$ is a closed convex set and $\fistar$ is a proper, closed and convex function.  
Specifically, if $\myfi$ is smooth, $A_i$ is the identity matrix, $I \in \R^{n \times n}$, $\fistar$ is the Fenchel conjugate to $f_i$, and $\Pii = {\dom}(\fistar)$\footnote{$\dom(\fistar):=\{\pii \in \R^n: \fistar(\pii) < \infty\}$.}. If $\myfi$ is structured non-smooth \cite{Nestrov2004Smooth},
 then $A_i \in \R^{m_i \times n}$ is a linear operator, $\Pii$ is bounded, and the $\fistar$-prox mapping can be solved efficiently \cite{LanBook}. This type of structured non-smooth function has found a wide range of applications, including total variation regularization in image processing \cite{total_variation_rudin1992nonlinear}, low-rank tensor \cite{kolda2009tensor,tomioka2011tensor}, overlapped group lasso \cite{mairal2011lasso,tibshirani2005graph_lasso}, and graph regularization \cite{jacob2009group,tibshirani2005graph_lasso}.}   Additionally, we assume the (strongly) convex regularization term $u(x)$ and the risk measure $(\rho^*, P)$ are available to the server node.

If the ambiguity set $P$ consists of only a fixed probability vector $\bar{p}$, say the empirical distribution,  \eqref{eq:orig_prob} is called risk-neutral, and it can be written as a finite-sum problem (see Chapter 5 of \cite{LanBook}):
\begin{equation}\label{eq:finite_sum}
\min_{x \in X} \sumi \bar{p}_i \myfi(x) + u(x).\end{equation}
However, if the costs among workers are imbalanced (different importance, limited availability of data, etc.), taking an average over the costs across workers might be meaningless or operationally wrong.
  In such cases, non-trivial $\rho^*$ and $P$ in \eqref{eq:orig_prob} generalizes risk-neutral optimization to risk-averse optimization and distributionally robust optimization (DRO). Specifically, if $\pmbf :=(\myfi[1],\ldots, \myfi[m])$ denotes the scenario costs and $\rho$ is a convex risk measure, it can be formulated as \eqref{eq:orig_prob} using Fenchel conjugates (see Definition 6.4 and Theorem 6.5 of~\cite{shapiro2014lectures}):
 \begin{equation}\label{eq:rho-def}
 \rho(\pmbf):= \argmax_{p \in P} \inner{p}{\pmbf} - \rho^*(p).\end{equation}
 For example, if we denote the (reference) probability mass function by $\bar p$,  some widely used risk measures and their conjugates are given as follows. \textblue{
\begin{itemize}
\item Mean semideviation of order $r$:
$$\rho(\pmbf) = \sumi {\bar p}_i \myfi +  c(\sumi{\bar p}_i [\myfi - \E\pmbf]_+^r)^{1/r}
=\max_{p \in P} \inner{p}{\pmbf},$$ where the ambiguity set $P:=\{p \in \Delta^m_+: \exists \zeta_i \geq 0 \text{ s.t. } p_i = \bar{p}_i (1 + \zeta_i - \inner{ \zeta}{\bar p}),\norm{\zeta}[s] \leq c \}$,
  $c \in [0, 1]$ and $\norm{\cdot}[s]$ is the conjugate norm to $\norm{\cdot}[r]$, i.e., $1/s + 1/r = 1$.
\item Entropic risk: 
$$\rho(\pmbf) =  \tau^{-1} \log \sumi {\bar p}_i \exp(\tau \myfi) = \max_{p \in \Delta^+_m} \inner{p}{\pmbf} - \tau^{-1}\sumi p_i \log( p_i /  {\bar p}_i).$$
\item Distributionally robust objective: $\rho(\pmbf) := \sup_{p \in P} \inner{\pmbf}{p}$ for some uncertainty set $P$.\\
\end{itemize}
The incorporation of all the above risk measures makes our problem \eqref{eq:orig_prob} more challenging than the finite-sum problem \eqref{eq:finite_sum}.  We note that \eqref{eq:orig_prob} also covers a popular risk measure CV@R with $\rho(\pmbf) = \max_{p \in \Delta^m_+, p_i \in [0,  {\bar p}_i/\alpha]} \inner{p}{\pmb{f}}$, where the parameter $\alpha >0$ captures the degree of risk aversion. The risk measure admits a finite-sum reformulation, $\rho(\pmbf) = \inf_{t}   \sumi  {\bar p}_i \{ [\myfi - t]_+ /\alpha + t\}$, but the function $ \tilde f_i(x,t):=[\myfi(x) - t]_+ /\alpha + t$ is nonsmooth with a very large Lipschitz-continuity constant, even if the original $f_i$ is smooth. In contrast, our conjugate formulation avoids the situation.
}

As alluded to earlier, we assume the communication network to have a star-topology where a computationally powerful central server node is connected directly to many worker nodes.
During a communication round, all the worker nodes send their local information to the server, and the server node broadcasts processed information to all worker nodes. This type of distributed optimization framework is very popular in machine learning, \textblue{such as federated learning \cite{kairouz2021advances},
where the data are held privately in each worker (device) and the central server learns a global model by communicating with the workers.}
Since communication in a network tends to be slower than computation inside a single node by orders of magnitude, and less communication implies better protection of privacy, one of the main goals of this paper is to study the system's communication complexity, i.e., the number of communication rounds required to find a quality solution $\bar x \in X \text{ s.t. } f(\bar x) - f(x^*) \le \epsilon$, where $x^*$ denotes an optimal solution of \eqref{eq:orig_prob}.

Risk-averse optimization problems of form \eqref{eq:orig_prob} have a wide range of applications in portfolio selection \cite{Markowitz}, renewable energy \cite{martinez2015risk}, power security \cite{javanbakht2014risk}, telecommunication \cite{larsson2014MIMO} and climate change planning \cite{tol2009climate}. As a concrete example, consider the massive multiple-input multiple-output (MIMO) system 
in the 5G communication network consisting of multiple active antennas and terminal devices \cite{larsson2014MIMO,parkvall20175G}. The multiple active antennas at the base station should be configured to ensure stable connections for all the terminal devices in its service area, rather than a high connection speed when averaged over all devices. Such an objective can be formulated as \eqref{eq:orig_prob} with $\myfi$ being the negative data speed at the $i$th terminal device and $(P, \rho^*)$ being the conjugate to the mean-semideviation risk measure. To gather information for the downlink and uplink channels, the base station needs to communicate with terminal devices. So in a highly mobile environment, 
finding a quality antenna connection quickly, i.e., with a only few rounds of communication,  is crucial. 
A second example is motivated by climate change. The state government may wish to invest in infrastructure to prepare for it. Each scenario cost function $\myfi$ may denote the long-term economic cost estimated by a certain climate model and a certain impact model \cite{tol2009climate}. To avoid downside risk,  \textblue{$(P, \rho^*)$ could be chosen as the conjugate to some risk measures mentioned above, say the entropic risk measure.} Because these models involve large amounts of data and costly simulation runs, we might need to store $\myfi$'s on separate computing nodes and use a communication network to find the optimal policy. In this case, a small number of communication rounds is crucial for efficiency.

\textblue{
Our formulation is also applicable to the computationally demanding distributionally robust optimization (DRO). DRO provides a powerful framework for learning from limited data \cite{kuhn2019wasserstein} and data-driven decision-making \cite{bertsimas2006robust,ye2016likelihood}. Under the assumption of finite scenario support $\Xi = [\xi_1,\ldots, \xi_m]$, we could use $\myfi(x):= f(x, \xi_i)$ to denote the cost under scenario $\xi_i$ and choose $\rho$ to be the risk measure induced by the corresponding probability uncertainty set \cite{shapiro2014lectures}. When implemented on a distributed communication network with the evaluation of $f(x, \xi_i)$'s performed in parallel on different machines, a small number of communication rounds is essential for fast computation.}

\textblue{
Additionally, the risk-averse formulation in \eqref{eq:orig_prob} could also be useful for federated learning between organizations, i.e., the cross-silo federated learning \cite{kairouz2021advances}. Cross-silo federated learning has found applications in finance risk prediction in reinsurance  \cite{Reinsurance}, drug discovery \cite{Drug-discovery}, electronic health record mining \cite{medical-record} and smart manufacturing \cite{smart-manufacturing}. If the workers represent demographically partitioned organizations or geographically partitioned data centers, we could choose $\rho$ to be the mean-semideviation risk measure to ensure that the trained model offers consistent performance across different populations. Risk measures may also provide incentives for competing organizations to cooperate. For example, consider the operations of competing airlines. When $\myfi(x)$ is the expected relative operation cost of $i$\ts{th} airline, choosing $\rho(\pmb{f}(x)):=\max_{i \in [m]} f_i(x)$ ensures the new policy $x$ benefits every participant. In both cases, a smaller number of communication rounds implies better protection of privacy.

}

Despite the importance of problem~\eqref{eq:orig_prob}, however, the study of its communication complexity and the development of efficient algorithms are rather limited. Since \eqref{eq:orig_prob} can be viewed as a trilinear
saddle point problem, we can potentially apply some recently developed first-order algorithms (e.g. \cite{zhang2019efficient,zhang2020optimal})
for solving it. However, these methods are designed without special consideration for communication burden. 
The most related algorithm is perhaps the sequential dual (SD) method, which was first proposed in \cite{zhang2019efficient} for the structured non-smooth problem and later extended in \cite{zhang2020optimal} to the smooth problem. 
The method is single-loop, so a direct implementation on a communication network requires one communication round in each iteration, leading to communication complexities of $\bigO(\sqrt{\Lf}\DX/\sqrt{\ep} + \DP\DPi\MA\DX/\ep)$ and $\bigO( \DPi\DX/\ep+\DP\DPi\MA\DX/\ep)$ for the smooth and the structured non-smooth problems, respectively.
Here $\Lf$, $\DPi$, $\MA$, $\DP$, and $\DX$ correspond to the overall smoothness constant, the dual radius, the operator norm of $\Ai$, the radius of $P$, and the distance to the optimal solution (see Tables 1.1 and 1.2, and Section 3 for their precise definitions).
On the other hand, for the risk-neutral problem \eqref{eq:finite_sum} with $P:=\{\bar p\}$,  direct distributed implementations of the Nesterov accelerated gradient method \cite{Nes83} and the primal-dual algorithm \cite{chambolle2011first} can achieve communication complexities of $\bigO(\sqrt{\Lf}\DX/\sqrt{\ep})$  and $\bigO(\DPi\MA\DX/\ep)$ for the smooth and the structured non-smooth problems, respectively, which were shown to be tight (see, e.g., \cite{scaman19}).  Clearly, there exists a significant gap in communication complexities, especially for smooth problems where the $\bigO(1/\ep)$ communication complexity for the risk-averse setting is much larger than the $\bigO(1/\sqrt{\ep})$ complexity for the risk-neutral setting. 
Therefore we pose the following research question:
\begin{empheq}[box=\fbox]{align*}
&\text{Can we solve the risk-averse problem over a star-shape network with the}\\
&\text{same communication complexity as the finite-sum problem?}
\end{empheq}

This paper intends to provide a positive answer to this question in three steps.

First, we propose a conceptual distributed risk-averse optimization (DRAO) method. It is inspired by works of Nesterov (Section 2.3.1 of \cite{nesterov2003introductory}) and Lan \cite{Lan15Bundle}  on composite optimization of the form $\min_{x} \rho(\pmbf(x))$ for a smooth vector function $\pmbf$. While Nesterov \cite{nesterov2003introductory} considers the problem with $\rho (\pmbf(x)) = \max_{i=1, \ldots, m} f_i(x)$, Lan~\cite{Lan15Bundle} generalizes $\rho$ to any monotone convex function. They can achieve an $\bigO(1/\sqrt{\ep})$ first-order (FO) oracle  complexity of $\pmbf$ by incorporating the following inner-linearization prox-mapping into the accelerated gradient descent (AGD) method or into the accelerated prox-level (APL) method:\textblue{
\begin{equation}\label{eq:prox-subproblem}
\xt \leftarrow \targmin_{x \in X}\ \rho\left(\myfi[1](\xundert)+ \inner{\grad \myfi[1](\xundert)}{x - \xundert}, \ldots, \myfi[m](\xundert)+ \inner{\grad \myfi[m](\xundert)}{x - \xundert}\right) + \tfrac{\eta}{2}\normsq{x -  x^{t-1}}.\end{equation}
Such an update is a simplified version of \eqref{eq:orig_prob} with $\myfi(\xundert) + \inner{\grad \myfi[i](\xundert)}{x - \xundert}$ denoting some (iterative) linearization of $\myfi$ at $\xundert$ and $\tfrac{\eta}{2}\normsq{x -  x^{t-1}}$ being the proximal term.
Similarly, we modify the SD method by combining the $p$ and $x$-prox updates into a single $(x, p)$-prox update given by 
\begin{align}
\xt \leftarrow& \argmin_{x \in X} \max_{p \in P} \sumi p_i [\inner{x}{\Ai \piit} - \fistar(\piit)] -\rhostar(p)+ u(x)+ \tfrac{\eta}{2} \normsq{x - \xtt}, \label{eq:joint_p_x}
\end{align}
where $\inner{x}{\Ai \piit} - \fistar(\piit)$ also represents some (iterative) linearization of $\myfi$ specified by the dual variable $\piit$. In fact, rewriting $\rho$ in its primal form \eqref{eq:rho-def} shows \eqref{eq:joint_p_x} to be equivalent to 
\begin{align*}
\xt \leftarrow  \argmin_{x \in X} \rho\left(\inner{x}{\Ai[1]\piit[t][1]} - \fistar[1](\piit[t][1]),\ldots,\inner{x}{\Ai[m]\piit[t][m]} - \fistar[m](\piit[t][m])\right) + u(x) + \tfrac{\eta}{2} \normsq{x - \xtt}[],
\end{align*}
which matches \eqref{eq:prox-subproblem} if $\piit$ is selected to be $\grad \myfi(\xundert)$ for smooth $\myfi$'s.}
\textblue{Such a modification of the SD method leads to the DRAO method. As shown in Table \ref{tb:com_complexity}, it achieves the optimal FO oracle complexities for $\pmbf$ (or $\Pii$-projection complexities) for both the smooth and the structured non-smooth problems.} Since $(\rho^*, P)$ is available to the server, \eqref{eq:joint_p_x} can be performed entirely on the server, so the communication complexities are the same (shown in Table \ref{tb:com_complexity}). However,  this approach requires $\rho$ to be simple so that \eqref{eq:joint_p_x} can be efficiently solved. This assumption might be too strong in practice. For example, if $m$ is large, the $(x,p)$-prox update in \eqref{eq:joint_p_x} with either the mean-semideviation risk measure $\rho$ or the Kantorovich ambiguity set $P$ is known to be computationally challenging. 
\begin{table}[h]
\setlength\extrarowheight{3pt}
\begin{minipage}{\textwidth}
\caption[\small Caption for LOF]{Communication Complexity and FO Oracle Complexity of $\pmbf$ for DRAO and DRAO-S\footnote{\scriptsize $\MA=\max_{i \in [m]} \norm{\Ai}$, $\DPi= \max_{i \in [m]} \max_{\pii, \piibar \in \Pii}\norm{\pii - \piibar}$.}}\label{tb:com_complexity}
\vspace{-2mm}
\centering
\scriptsize
\begin{tabular}{|l|l|l|}
\hline
 & Convex ($\alpha = 0$) & strongly convex ($\alpha > 0)$ \\
\hline
Smooth & $\bigO(\sqrt{\Lf}\|\xt[0] - \xstar\|/\sqrt{\ep})$ & $\bigO(\sqrt{\Lf/\alpha}\log(1/\ep))$ \\
\hline
Structured Non-smooth & $\bigO(\MA \DPi \norm{\xt[0] - \xstar}/\ep )$ & $\bigO(\MA \DPi/\sqrt{\ep \alpha} )$ \\
\hline
\end{tabular}
\\
\caption[Caption for LOF]{\small $P$-projection and $X$-projection Complexity of DRAO-S \footnote{\scriptsize $\DP$ denotes $P$'s radius. $\tilde{M}$ denotes the operator norm of $\norm{\grad f_1(x),\ldots, \grad f_m(x)}$ over some bounded ball around $\xstar$ and $\bar{M}_{A\Pi}$ denotes the operator norm of $\norm{\Ai[1]\pii[1],\ldots, \Ai[m]\pii[m]}$ over the whole feasible region $\Pi$.} }\label{tb:p_complexity}
\vspace{-2mm}
\centering
\scriptsize
\begin{tabular}{|l|l|l|}
\hline
 & convex ($\alpha = 0$) & strongly convex ($\alpha > 0$) \\
\hline
Smooth & $\bigO(\DP \tilde{M} \|\xt[0] - \xstar\|/\ep)$ & $\bigO((\Lf/\alpha)^{1/4} \tilde{M} \DP/ \alpha \sqrt{\ep})$\footnote{\scriptsize Number of P-projections required to generate an $\ep$-close solution, i.e., $\norm{\xt[N] - \xstar}^2 \leq \ep$. } \\
\hline
Structured Non-smooth & $\bigO(\DP \bar{M}_{A\Pi} \|\xt[0] - \xstar\|/\ep)$ & $\bigO(\DP \bar{M}_{A\Pi} /\sqrt{\ep \alpha} )$ \\
\hline
\end{tabular}
\end{minipage}
\vspace{-3mm}
\end{table}


Second, we overcome the restrictive assumption 
of $\rho$ being simple by developing a saddle point sliding (SPS) subroutine. It replaces \eqref{eq:joint_p_x}  in the DRAO method by performing only a finite number of $P$-projections and $X$-projections to solve the saddle point subproblem inexactly. 
The new method, called distributed risk-averse optimization with sliding (DRAO-S), maintains the same communication complexities as DRAO while improving on its computation efficiency. Since each inner iteration of the sliding subroutine requires one $P$-projection and one $X$-projection, the total numbers of these projections are optimal in most cases \footnote{Except for the strongly convex smooth problem which is worse off by a factor of $(\Lf/\alpha)^{1/4}$.}. As shown in Table \ref{tb:p_complexity}, they match the lower bounds \cite{ouyang2021lower} for solving a single $(x, p)$ bi-linear saddle point problem, i.e., \eqref{eq:joint_p_x} with a fixed $\pi^t$ and $\etat=0$.
Such a result is similar to that of the gradient sliding (GS) method \cite{lan2016gradient}  for solving an additive composite problem,
\begin{equation}\label{eq:composite-add}
\min_{x \in X}f(x) + g(x).\end{equation}
The GS method can achieve both optimal $f$-oracle and optimal $g$-oracle complexities. 
However, our nested composite problem appears to be more challenging. This is because for a fixed $x$,  the optimal dual variables $p$  and $\pi$ in \eqref{eq:orig_prob} are dependent, while the optimal dual variables $\pi_f$ and $\pi_g$ (associated with the saddle point reformulation of \eqref{eq:composite-add} through bi-conjugation \cite{Beck2017First}) are independent. In fact, \eqref{eq:composite-add} can always be rewritten as a nested composite problem (see the discussion in Example 3 of \cite{Lan15Bundle}). Additionally, the SPS subroutine in the DRAO-S method is initialized differently from the usual sliding subroutines in \cite{lan2016gradient} and \cite{lan2021graph}. Such a modification simplifies both the outer loop algorithm and the convergence analysis. This simplification could motivate the application of the sliding technique to a wider range of problems. Furthermore, an interesting feature of the DRAO-S method is that its inner loop, the SPS subroutine, can adjust dynamically to the varying levels of difficulty, characterized by $\norm{\pi^t}$, of the saddle point subproblem \eqref{eq:joint_p_x}. This allows us to remove the assumption of the smooth $\myfi$'s being Lipschitz continuous, which is required by the SD method in \cite{zhang2020optimal}, but may not hold if the domain $X$ is unbounded.

Third, we show that the communication complexities of both DRAO and DRAO-S are not improvable by constructing lower complexity bounds. Previous developments are restricted to a trivial $P$  and the smooth problem \cite{scaman2017optimal}.  We propose a more general computation model which includes both the $f_i$-gradient oracle and the $\fistar$-prox mapping oracle, and introduce a different set of problem parameters appropriate for the risk-averse problem. They allow us to develop, for a non-trivial $P$ and for both the smooth and the structured non-smooth problems, new lower complexity bounds matching the upper communication complexity bounds possessed by  DRAO and DRAO-S.

The rest of the paper is organized as follows. Preliminary Section 2 reviews a gap function in \cite{zhang2019efficient} which will guide the algorithm design. Section 3 and Section 4 propose and analyze the DRAO and DRAO-S methods, respectively. Section 5 provides lower communication complexity bounds and Section 6 provides some encouraging numerical results. Finally, some concluding remarks are made in Section 7.  



\subsection{Notation \& Assumptions}\label{subsec:notation}
The following assumptions and notations will be used throughout the paper. 
\begin{itemize}
    \item The set of optimal solutions to \eqref{eq:orig_prob}, $X^*$, is nonempty.  $\xstar$ denotes an arbitrary optimal solution, and $f_*$ denotes the optimal objective, $f(\xstar)$. $\Rx$ represents an estimate of the distance from the initial point to $\xstar$, i.e., $\Rx \geq \norm{\xt[0] -\xstar}.$ 
    \item $\DP$ denotes the radius of $P$, i.e.,  $\DP := \max_{p,\bar p \in P}\sqrt{2 U(p, \bar p)}$ where $U$ is the chosen Bregman distance function \cite{LanBook}.
    \item $\pmb{f}: \R^n \rightarrow \R^m$ denotes a vector of scenario cost functions, $[f_1; ...; f_m]$, and $\grad \pmb{f}(x): \R^n \rightarrow \R^{m \times n}$ denotes the Jacobian matrix function.
    \item \textblue{We refer to the following computation as either a prox mapping or a projection:
    \begin{equation}\label{eq:prox-def}
    \hat w \leftarrow \argmin_{w \in W} \inner{g}{w} + h(w) + \tau V(w;\bar w),\end{equation}
    where the vector $g$ represents some ``descent direction'' (the gradient for example), and 
    $h(w)$ denotes a simple convex function \cite{LanBook}.  $V$ denotes the Bregman distance function, $\bar w$ is a prox center, and $\tau$ is a stepsize parameter. Together they ensure the output $\hat w$ is close to $\bar w$. }
    In particular, 
    we call it an $x$, a $\pii$ or a $p$-prox mapping (an $X$, a $\Pii$ or a $P$-projection) if  $W=X$ and $h\equiv 0$, $W=\Pii$ and $h = \fistar$, or $W=P$ and $h= \rho^*$, respectively. Sometimes, the term prox update also is used to emphasize that the prox mapping is performed to update $w^{t}=\hat w$ from $\bar w=w^{t-1}$.
    
\end{itemize}

\section{Preliminary: $Q$-gap function}
We introduce a gap function \cite{zhang2019efficient} which will guide our algorithmic development throughout the paper.
For notation convenience, we  denote $\pi \equiv (\pi_1, \ldots,\pi_m)$ and $\Pi \equiv \Pi_1 \times \Pi_2 \times \ldots \times \Pi_m$ so that \eqref{eq:orig_prob} can be written  
as 
\begin{equation}\label{eq:prob}
\min_{x\in X} \max_{p \in P} \max_{\pi\in \Pi} \{\La(x; p, \pi) := \sumi p_i \left(\inner{A_i x}{\pii} - \fistar(\pii)\right)- \rho^*(p) + u(x)\}.
\end{equation}
The following duality relation between the reformulation and the original problem \eqref{eq:orig_prob} is valid (see Proposition 2.1 of \cite{zhang2019efficient}).
\begin{lemma}\label{lm:duality}
Let $f$ and $\La$ be defined in \eqref{eq:orig_prob} and \eqref{eq:prob}, then the following statements hold for all $ x \in X$.
\begin{itemize}
\item[a)] Weak Duality: $f(x) \geq \La(x, p, \pi)$ for all $p \in P, \pi \in \Pi$.
\item[b)] Strong Duality: $f(x) = \La(x, \hat{p}, \hat \pi)$ for any
 $\hpii \in \argmax_{\pii \in \Pii} \inner{\pii}{\Ai x} - \fistar(\pii)$, $i=1, \ldots,m$, and any $\hat{p} \in \argmax_{p \in P} \sumi p_i \myfi(x)- \rho^*(p)$.
\end{itemize}
\end{lemma}

We measure the quality of a feasible solution $z=(x, p, \pi)$  by a gap function $Q$ 
associated with some feasible reference point $\hz:=(\hx; \hp, \hat \pi)$: 
\begin{equation}\label{eq:Q_func}
Q(z; \hz) := \La(x; \hp, \hat \pi) - \La(\hx; p, \pi).
\end{equation}
The $Q$ function provides a bound on the function optimality gap from above.
\begin{lemma}\label{lm:sad_ex}
Let $Q$ be defined in \eqref{eq:Q_func}, then 
\begin{equation}\label{eq:optimality_gap}
f(x) - f(\xstar) \leq \max_{\hp \in P, \hat \pi \in \Pi} Q((x;p,\pi); (\xstar; \hp, \hat \pi) ).
\end{equation}
Moreover, the optimal solution $\xstar$ of \eqref{eq:orig_prob}, together with some $\piistar\in \argmax_{\pii \in \Pii} \inner{\pii}{\Ai \xstar} - \fistar(\pii)$, $i=1,\ldots,m$, and some $\pstar \in \argmax_{p \in P} \sumi p_i \myfi(\xstar) - \rho^*(p)$   forms a saddle point $\zstar:=(\xstar; \pstar, \pi^*)$ of \eqref{eq:prob}, i.e.,
\begin{equation}\label{eq:saddle_point}
Q(z ; \zstar) \geq 0, \ \forall z\equiv (x; p, \pi) \in X \times P \times \Pi.
\end{equation}
\end{lemma}

\begin{proof}
Let $\hat p$ and $\hat \pi$ by defined in Lemma \ref{lm:duality}.b).
By Lemma \ref{lm:duality}, we have
$f(x) - f(x^*) \le \La(x, \hat{p}, \hat \pi) - \La(x^*, p, \pi) =Q((x; p,\pi); (x^*; \hat p, \hat \pi))$,
from which \eqref{eq:optimality_gap} follows immediately. 
Next,  the first-order optimality condition of \eqref{eq:orig_prob} implies that there exist some $ \piistar \in \argmax_{\pii \in \Pii} \inner{\pii}{\Ai \xstar} - \fistar(\pii)$, $g^* \in \partial u(\xstar)$ and some $\pstar \in \argmax_{p \in P} \sumi p_i \myfi(\xstar) - \rho^*(p)$ such that $\langle \sumi \pstar_i \Aitr \piistar, x - x^*\rangle + u(x) - u(\xstar) \geq \langle \sumi \pstar_i \Aitr \piistar + g^*, x - x^*\rangle \ge 0$ for any $x \in X$. This observation
together with the definition of $\La$ in \eqref{eq:prob}
then imply that
$$\La(x; \pstar, \pi^*) \geq \La(\xstar; \pstar, \pi^*), \forall x \in X.$$
Moreover, due to our choice of $(\pstar, \pi^*)$, Lemma \ref{lm:duality} also implies that 
$$f(x^*) = \La(\xstar; \pstar, \pi^*) \geq \La(\xstar; p, \pi), \forall (p, \pi) \in P \times \Pi.$$
  \eqref{eq:saddle_point}  then follows from combining the preceding two inequalities.
\end{proof}
\vgap
In view of Lemma~\ref{lm:sad_ex},
we can use $Q$ to guide our search for an $\ep$-optimal solution. In particular, we decompose $Q$ into three 
sub-gap functions given by
\[
Q(\zbar;\hz) = Q_x(\zbar;\hz) + Q_p(\zbar;\hz) + Q_\pi(\zbar;\hz)
\]
with
\begin{equation}\label{eq:Q-decomp}
\begin{split}
&Q_{\pi}(\zbar;\hz) := \La(\bar{x}; \hp, \hat \pi) - \La(\bar{x}; \hp, \bar \pi) = \sumi \hp_i \left[\inner{A_i \bar{x}}{\hpii - \bar{\pi}_{\bci}} - \fistar(\hpii) + \fistar(\bpii)\right].\\
&Q_p(\zbar;\hz) := \La(\bar{x}; \hp, \bar \pi) - \La(\bar{x}; \bp, \bar \pi) = \sumi (\hpi - \bpi)[ \inner{A_i\bar{x}}{\bar{\pi}_{\bci}} -  \fistar(\bpii)] - (\rho^*(\hp) - \rho^*(\bp)).\\
&Q_x(\zbar;\hz) := \La(\bar{x}; \bar{p}, \bar{\pi}) - \La(\hx; \bar{p}, \bar{\pi}) = \inner{\sumi \bar{p}_{\bci}\Aitr\bar{\pi}_{\bci}}{\barx - \hx} + u(\bar x) - u(\hat x)  .\\ 
\end{split}
\end{equation}
\section{Upper Bounds for Communication Complexity}

We propose the distributed risk-averse optimization (DRAO) method to provide upper bounds on communication complexity. The algorithm and its convergence properties are presented in Subsection \ref{subsec:DRAO} and the convergence analysis is presented in Subsection \ref{sec:ConvergenceAnalysis}.

\subsection{The DRAO method}\label{subsec:DRAO}

The DRAO method is designed for solving the min-max-max trilinear saddle point problem in \eqref{eq:prob}. It is inspired by two algorithms for optimizing nested composite problems. First, the sequential dual (SD) algorithm,
proposed in  \cite{zhang2019efficient,zhang2020optimal}, performs sequential proximal updates to the dual variables $\pi$ and $p$ before updating the primal variable $x$. The DRAO method is built on similar sequential proximal updates for $\pi$, $p$, and $x$. Second, the accelerated prox-level (APL) algorithm, proposed in \cite{Lan15Bundle}, can reduce the number of outer iterations further by solving a more complicated proximal sub-problem \eqref{eq:prox-subproblem}. The DRAO method exploits this property by combining the separate $p$ and $x$ proximal updates into a single $(x, p)$ prox update step in the server node to save communication. 

Algorithm \ref{alg:cPD} describes a generic DRAO method which will be later specialized for solving the 
smooth and the structured nonsmooth problems.
As shown in Algorithm \ref{alg:cPD}, the server first sends an extrapolated point $\xtilt$ to the workers for them to perform dual proximal updates in Line 3. The only goal is to reduce the sub-gap function $Q_{\pii}$ (c.f. \eqref{eq:Q-decomp}). Here we intentionally leave the prox-function $\Vi$ in an abstract form because its selection and the resulting implementation will depend on the smoothness properties of $\myfi$. Next, the server collects the newly generated $A_i \piit$ in Line 4 to solve the $(x, p)$ prox update problem in Line 5 to reduce both $Q_x$ and $Q_p$. 


\begin{minipage}{.9\textwidth}
\begin{algorithm}[H]
\caption{A Generic \textbf{D}istributed \textbf{R}isk \textbf{A}verse \textbf{O}ptimization (DRAO) Method}
\label{alg:cPD}
\begin{algorithmic}[1]
\Require $\xt[0]=\xt[-1] \in X$, $\piit[0] \in \Pii $ in every node, stepsizes $\{\thetat\}$, $\{\etat\}$, $\{\taut\}$, and  weights $\{\wt\}$.

\For{$t =  1,2,3 ... N$}
\parState{Server computes $\xtilt\leftarrow \xtt + \thetat (\xtt - \xt[t-2])$. Broadcast it to all workers.}
\parState{Every worker computes $\piit \leftarrow \argmax_{\pii \in \Pii} \inner{A_i\xtilt}{\pii} - \fistar(\pii) - \taut V_i(\pii; \piitt)$, and evaluates $\vit \leftarrow \Aitr \piit $ and $\fistar(\piit)$.}
\parState{All workers send their $(\vit, \fistar(\piit))$ to the server.}
\parState{Server updates \\\,\,\,$\xt \leftarrow \argmin_{x \in X} \max_{p \in P} \sumi p_i (\inner{x}{\vit} - \fistar(\piit)) -\rhostar(p)+ u(x)+ \tfrac{\eta_t}{2} \normsq{x - \xtt} .$} 
\EndFor
\State \Return $\xbart[N]:= \tsum_{t=1}^N \wt \xt / (\sumt \wt). $
\end{algorithmic}
\end{algorithm}
\vgap
\end{minipage}

In the generic DRAO algorithm, we assume subproblems in Lines 3 and 5 to be solved
exactly by the workers and server, respectively.
Line 3 reduces to local gradient evaluations in the smooth case, while requiring a prox mapping for the structured nonsmooth case.
Line 5 requires us to solve a structured bilinear saddle 
point problem. We will discuss in detail
how to solve these problems approximately in the next section while focusing on the communication complexity now.


\vgap

First, we consider the smooth problem where all $A_i$'s are identity matrices and all $\myfi$'s are smooth such that 
$
\| \nabla \myfi(x_1) - \nabla \myfi(x_2)\| \le L_i \|x_1 - x_2\|, \forall x_1, x_2 \in \R^n.
$
\textblue{Since the Fenchel conjugate to a smooth convex function   is strongly convex \cite{hiriart1993convex}}, a natural choice of the prox-function $\Vi$ in DRAO would be the Bregman distance function generated by $\fistar$ given by 
\begin{equation} \label{eq:dual_bregman}
\Dfistar(\pii; \bpii) := \fistar(\pii) - \fistar(\bpii) - \inner{\fistarp(\bpii)}{\pii - \bpii}.
\end{equation}
It has been shown in~\cite{LanZhou18RPDG,LanBook,zhang2020optimal} that
the $\pii$ proximal update in Line 3 of Algorithm \ref{alg:cPD} is equivalent to a gradient evaluation. Specifically, with $\xundert[0]=\xt[0]$ and $\piit[0] = \grad \myfi(\xundert[0])$, Line 3 reduces to the following steps:
\begin{align}
    \xundert &\leftarrow (\xtilt + \taut \xundert[t-1] )/(1 + \taut),\label{eq:smooth_alg1}\\
    \piit &\leftarrow \grad \myfi(\xundert),\label{eq:smooth_alg2}\\
    \fistar(\piit)&\leftarrow  \inner{\xundert }{\piit}- \myfi(\xundert).\label{eq:smooth_alg3}
\end{align}
Plugging $\fistar(\piit)$ defined in \eqref{eq:smooth_alg3} into Line 5 of Algorithm~\ref{alg:cPD}, we can completely remove the information about
the conjugate function $\fistar$. Therefore, DRAO is a purely primal algorithm for the smooth problem.

To discuss the convergence properties of DRAO, we need to properly define some Lipschitz smoothness constants. For a given $p \in P$, let us denote $f_p(x) := \sumi p_i f_i(x)$.
Clearly, $f_p$ is a smooth convex function with Lipschitz continuous gradients, i.e.,
$
\|\nabla f_p(x_1) - \nabla f_p(x_2)\| \le L_p \|x_1 - x_2\|, \forall x_1, x_2 \in X. 
$
Moreover, $L_p \le \sumi p_i L_i$. We define an aggregate smoothness constant $\Lf$  to characterize the overall smoothness property of the risk-averse problem \eqref{eq:prob}:
\begin{equation}\label{eq:smo_cst}
\Lf = \max_{p \in P} L_p.
\end{equation}
 Observe that in the risk neutral case with $P=\{(1/m,\ldots,1/m)\},$ $\Lf$ is the global smoothness constant of $f$ \cite{scaman19}, which is upper bounded by $\tfrac{1}{m} \sumi L_i$.
In the robust case when $P= \Delta_m^+$, $\Lf = \max_i L_i$.
 
 Theorem~\ref{thm:central-sm} and \ref{thm:central-sm-str} below show the convergence rates of the DRAO method applied to the aforementioned smooth problems, for a non-strongly convex  $u(x)$ and a strongly convex $u(x)$ respectively.
 Their proofs are given in Section~\ref{sec:ConvergenceAnalysis}.

 \begin{theorem}\label{thm:central-sm}
Let $\Lf$ be defined in \eqref{eq:smo_cst}. If $\{\xt\}_{t=1}^N$ are generated by the DRAO method applied to  a smooth problem with
$$\wt=t,\ \thetat=(t-1)/t, \taut = (t-1)/2, \etat = 2\Lf/t.$$
Then for a reference point $\hz:=(\hx; \hp, \hat \pi)$ in which $\hpii = \grad \myfi(\bar x)$ for some $\bar x\in X$, we have 
\begin{equation}\label{eq:smooth_Q_conv}
\sumt \wt Q(\zt; \hz) + \Lf \normsq{\xt[N] - \hx} \leq \Lf \normsq{\xt[0] - \hx}.\end{equation}
In particular, the ergodic solution $\xbart[N]$ satisfies 
\begin{equation}\label{eq:smooth_f_conv}
f(\xbart[N]) - f(\xstar) \leq 2\Lf \normsq{\xt[0] - \xstar}/N(N+1), \forall N \geq 1.\end{equation}
\end{theorem}

\begin{theorem}\label{thm:central-sm-str}
Let $\Lf$ be defined in \eqref{eq:smo_cst}. Assume, in addition, that $u(x)$ is $\alpha$-strongly convex for some $\alpha > 0$. Let $\kappa:= \Lf/\alpha$ denote the condition number. If $\{\xt\}_{t=1}^N$ are generated by the DRAO method applied to smooth problems with
\begin{equation}\label{stp:central_sm_str}
\thetat=\theta := \tfrac{\sqrt{4\kappa + 1}-1}{\sqrt{4\kappa + 1}+1},\ \wt=(\tfrac{1}{\theta})^{t-1},\ \taut =\tau:= \tfrac{\sqrt{4\kappa + 1} - 1}{2},\ \etat =\eta:= \tfrac{\alpha(\sqrt{4\kappa + 1} - 1)}{2}.\end{equation}
Then for a reference point $\hz:=(\hx; \hp, \hat \pi)$ in which $\hpii = \grad \myfi(\bar x)$ for some $\bar x\in X$, we have 
\begin{equation}\label{eq:smooth_Q_str_conv}
\sumt \wt Q(\zt; \hz) +  \frac{\alpha(\sqrt{4\kappa + 1} - 1)}{4\theta ^N} \normsq{\xt[N] - \hx} \leq \frac{(\sqrt{4\kappa + 1} - 1)}{4} (\alpha\normsq{\xt[0] - \hx} + \Lf \normsq{x^0 - \bar x}).\end{equation}
In particular, the last iterate $\xt[N]$ converges geometrically:
\begin{equation}\label{eq:smooth_x_str_conv}
\normsq{\xt[N] - \xstar} \leq \theta^N (1 + \kappa) \normsq{\xt[0] - \xstar}, \forall N \geq 1.\end{equation}
\end{theorem}

We make two remarks regarding the above convergence results. First, selecting the saddle point $\zstar$ defined in Lemma \ref{lm:sad_ex} as $\hat z$, Theorem \ref{thm:central-sm} (c.f. \eqref{eq:smooth_Q_conv}) and \ref{thm:central-sm-str} (c.f. \eqref{eq:smooth_Q_str_conv}) imply that all generated iterates, $\{\xt\}_{t\geq 1}$, are inside some ball around $\xstar$:
\begin{align*}
 &\norm{\xt - \xstar} \leq \norm{x^0 - \xstar} \quad \quad \quad \quad \quad \text{ if } \alpha = 0,\\
 &\norm{\xt - \xstar} \leq (1 + \Lf/\alpha)\norm{x^0 - \xstar} \text{ if } \alpha > 0.
 \end{align*} 
 This shows that the search space for $\xt$ is essentially bounded. Such a property will become useful when we solve the saddle point subproblem in Line 5 of DRAO approximately in the next section. 
 Second, Theorem \ref{thm:central-sm} and \ref{thm:central-sm-str} imply, respectively, $\bigO(\sqrt{\Lf}\norm{\xt[0] - \xstar}/\sqrt{\ep})$ and $\bigO(\sqrt{\Lf/\alpha}\log(1/\ep))$ communication complexities to find $\ep$-optimal solutions. It is interesting to note that with $\Lf$ defined in \eqref{eq:smo_cst}, these results are valid even if $P$ is larger than the probability simplex, i.e., $\Delta_+^m  \subsetneq P \subset \R^m_+$, which could be useful if the risk measure $\rho$ is not positive homogeneous. 
 We will show later in Section 5 that these communication complexity bounds are not improvable in general.


\vgap

Next, let us consider the structured non-smooth problem. Because $\myfi^*$ may not be strongly convex, the Bregman distance function $\Dfistar$ (c.f. \eqref{eq:dual_bregman}) is no longer suitable for $\pii$ prox update. Instead, we choose
$\Vi(\pii; \bpii):= \tfrac{1}{2} \normsq{\pii - \bpii},$
so that the $\pii$ proximal update is given by:
\begin{equation}\label{eq:drao_ns_pi_dualprox}
\piit \leftarrow \argmax_{\pii \in \Pii} \inner{A_i\xtilt}{\pii} - \fistar(\pii) - \tfrac{\taut}{2} \normsq{\pii-  \piitt}.\end{equation}
Theorem~\ref{thm:cen-non-smooth} below states the convergence properties of Algorithm \ref{alg:cPD} applied to the structured nonsmooth problem 
and its proof is provided in Section~\ref{sec:ConvergenceAnalysis}. We need to define the maximum linear operator norm $\MA$ and the maximum dual radius $\DPi$  as 
\begin{equation}\label{eq:MA_def}
\MA := \max_{i \in [m]} \norm{\Ai}[2,2], \ \DPi:= \max_{i \in [m]} \max_{\pii, \piibar \in \Pii} \norm{\pii - \piibar}.
\end{equation}
Note that $\MA\DPi$ provides an estimate of the Lipschitz continuity constant of $\tsum_{i} p_i f_i(x).$

\begin{theorem}\label{thm:cen-non-smooth}
Let a structured non-smooth risk-averse problem \eqref{eq:orig_prob} be given. Let  $\MA$ and $\DPi$ be defined above in \eqref{eq:MA_def} and let $\Rx\geq \norm{\xt[0] - \xstar}$.  \\
a) If $\alpha=0$ and the stepsizes satisfy
$$\wt=1,\ \thetat=1, \etat = \MA \DPi /  \Rx, \taut = \MA \Rx / \DPi,$$
 the following convergence rate holds for the solution $\xbarN$ returned by the DRAO algorithm 
\begin{equation}\label{eq:nonsmooth_f_conv}
f(\xbart[N]) - f(\xstar) \leq \MA\DPi \Rx/N.\end{equation}
b) If $\alpha>0$ and the stepsizes satisfy
$$\wt=t,\ \thetat=(t-1)/t, \etat = t\alpha/3, \taut = 3\MA^2 / t\alpha,$$
 the following convergence rate holds for the solution $\xbarN$ returned by the DRAO algorithm 
\begin{equation}\label{eq:nonsmooth_f_conv_str}
f(\xbart[N]) - f(\xstar) \leq \left(\alpha \normsq{\xt[0] - \xstar}/3 + 3 \MA^2 \DPi^2/\alpha\right)/N^2.\end{equation}
\end{theorem}


The preceding theorem gives us $\bigO(\Rx\DPi \MA/\ep)$ and $\bigO(\MA\DPi/\sqrt{\ep \alpha})$\footnote{We assume the strong convexity modulus $\alpha$ to be small such that the $\MA^2\DPi^2/\alpha$ term dominates in \eqref{eq:nonsmooth_f_conv_str}.} communication complexities for solving the structured nonsmooth problem under the non-strongly convex and the strongly convex settings, respectively.  These complexity bounds are
worse than those of the smooth problem by an order of magnitude. It is interesting to note that
the smoothness properties of the scenario cost functions have a significant impact on communication complexity,
even under the assumption that the workers are equipped with the capability to solve the $\pii$ proximal update in \eqref{eq:drao_ns_pi_dualprox}.

    \subsection{Convergence analysis} \label{sec:ConvergenceAnalysis}

Our main goal in this subsection is to establish the convergence rates associated with the DRAO method stated in Theorems~\ref{thm:central-sm}, \ref{thm:central-sm-str}, and \ref{thm:cen-non-smooth}.

\vgap

We will first show some general convergence properties about the generic DRAO method in Algorithm~\ref{alg:cPD}. Since this result holds regardless of the strong convexity of $\fistar$ (i.e., $\mu = 0$ is allowed in \eqref{eq:strong_convex_dual}) and the strong convexity of $u$ (i.e., $\alpha=0$ is allowed), it will be applied to both smooth and nonsmooth problems under either convex or strongly convex settings.

\begin{proposition}\label{pr:central_Q_conv}
Let $\{\zt \equiv (\xt; \pt,\pi^t) \}_{t=1}^N$ be generated by Algorithm~\ref{alg:cPD} for some $\pt\in \argmax_{p \in P} \sumi p_i (\inner{\xt}{\piit} - \fistar(\piit)) - \rho^*(p)$. Fix a reference point  $\hz:=(\hx; \hp, \hat \pi) \in X \times P \times \Pi$ (c.f. \eqref{eq:Q_func}). Assume $\mu$ is a non-negative constant satisfying
\begin{equation} \label{eq:strong_convex_dual}
\fistar(\pii) - \fistar(\bpii)- \inner{\gi'(\bpii)}{\pii - \bpii} \geq \mu \Vi(\pii; \bpii),\ \forall \pii, \bpii \in \Pii, \forall i \in [m].
\end{equation}
If there exists a positive constant $q$ satisfying
\begin{equation}\label{eq:aggregate_stx_cvxit_cst}
\begin{split}
&\tsum_{i=1}^{m} p_i V_i(\pii^t; \pii^{t-1}) \geq \tfrac{1}{2q} \normsq{\sumi p_i \Aitr(\piit - \piitt)}, \forall t \geq 2,  \forall p \in P,\\
&\tsum_{i=1}^{m} p_i V_i(\hpii; \pii^{N}) \geq \tfrac{1}{2q} \normsq{\sumi p_i \Aitr (\hpii - \pii^{N})},  \forall p \in P,
\end{split}
\end{equation}
and the stepsizes satisfy the following conditions for all $t \geq 2:$
\begin{align}\label{req:DRAO}
\begin{split}
&\wtt=\wt\thetat,\\
&\etatt \taut \geq \thetat q,\ (\taut[N] + \mu) \etat[N] \geq q,\\
&\wt \etat \leq\wtt (\etatt + \alpha),\  \wt\taut \leq \wtt (\tautt  + \mu),
\end{split}
\end{align}
then the next bound is valid for all $\hz:=(\hx; \hp, \hat \pi) \in X \times P \times \Pi$ and $N \geq 1$:
\begin{equation}\label{eq:central_Q}
\sumt \wt Q(\zt; \hz) + \tfrac{\wt[N] (\etat[N]+ \alpha)}{2} \normsq{\xt[N] - \hx} \leq \tfrac{\wt[1] \etat[1]}{2} \normsq{\xt[0] - \hx} + \wt[1]\taut[1]\sumi \hpi \Vi(\hpii; \piit[0]).
\end{equation}
\end{proposition}

\begin{proof}
Let $Q_\pi$, $Q_p$ and $Q_x$ be defined in \eqref{eq:Q-decomp}.
We begin by analyzing the convergence of $\Q[\pi].$ It follows from the definition of $\xtilt$ that
\begin{align*}
\inner{\hpii -\piit }{A_i(\xtilt-\xt)}=& -\inner{\hpii - \piit}{A_i(\xt -\xtt)} + \thetat \inner{\hpii -\piitt}{A_i(\xtt -\xt[t-2])}\\
&+\thetat \inner{\piitt -\piit}{A_i(\xtt -\xt[t-2])}.
\end{align*}
The optimality condition for the dual update in Line 3
of Algorithm~\ref{alg:cPD} (see Lemma 3.1 of \cite{LanBook}) implies
\begin{align*}
\inner{\hpii - \piit}{A_i \xt}& + \fistar(\piit) - \fistar(\hpii) + \inner{\hpii -\piit }{A_i(\xtilt-\xt)} \\
&\leq \taut \Vi(\hpii; \piitt) - (\taut + \mu) \Vi(\hpii; \piit) - \taut \Vi(\piit; \piitt).
\end{align*}
So, combining the above two relations,
taking the $\wt$ weighted sum of the resulting inequalities and using the conditions that $\wtt=\wt\thetat$ and $\wt\taut \leq \wtt (\tautt  + \mu)$, we obtain 
\begin{align*}
\sumt \wt &(\inner{\hpii - \piit}{A_i \xt} + \fistar(\piit) - \fistar(\hpii)) \\
\leq& - (\wt[N](\taut[N] + \mu) \Vi(\hpii; \piit[N]) -\wt[N]\inner{\hpii - \piit[N]}{A_i (\xt[N] -\xt[N-1])})\\
& -\tsum_{t=2}^{N} [\wt \taut \Vi(\piit; \piitt) + \wtt \inner{\piitt -\piit}{A_i(\xtt -\xt[t-2])} ]\\
&+ \wt[1]\taut[1] \Vi(\hpii; \piit[0]).
\end{align*}
A $\hpi$-weighted sum of the above inequality leads to the desired $Q_\pi$ convergence bound given by
\begin{align}\label{eq:central_Qpii_bound}
\begin{split}
\sumt \wt& Q_\pi(\zt;\hz) \\
\leq& - (\wt[N](\taut[N] + \mu) \sumi \hpi \Vi(\hpii; \piit[N]) -\wt[N]\inner{\sumi \hpi \Aitr(\hpii - \piit[N])}{\xt[N] -\xt[N-1]})\\
& -\tsum_{t=2}^N [\wt \taut\sumi \hpi \Vi(\piit; \piitt) + \wtt \inner{\sumi \hpi \Aitr(\piitt -\piit)}{\xtt -\xt[t-2]} ]\\
&+ \wt[1]\taut[1] \sumi \hp_i \Vi(\hpii; \piit[0]) .
\end{split}
\end{align}

Next, we consider $\xt$ and $\pt$ generated by Line 5 in Algorithm \ref{alg:cPD}. Let
$F(x; \pi^t):= \max_{p \in P}\sumi p_i [\inner{x}{\vit} - \fistar(\piit)] -\rhostar(p) + u(x) + \tfrac{\eta}{2} \normsq{x - \xtt} .$
Since $\xt \in \argmin_{x \in X} F(x; \pi^t)$, the first-order necessity condition implies the existence of some maximizer $\pt$ and some subgradient $u'(\xt) \in  \partial u(\xt) $ such that 
\begin{equation*}
\sumi \pt_i \vit + \eta (\xt - \xtt) + u'(\xt) \in -\partial \delta_X(\xt) \Rightarrow  \inner{\sumi \pt_i \Aitr \piit}{ \xt - \hx} + \inner{\etat (\xt - \xtt) + u'(\xt)}{\xt- \hx} \leq 0. 
\end{equation*}
Since $\alpha$-strong convexity of $u$ implies that $u(\xt)  + \alpha \|\xt - \hx\|^2/2 - u(\hx) \leq \inner{u'(\xt)}{\xt- \hx}$,  and $\tfrac{\etat}{2} \normsq{\xt - \hx} + \tfrac{\etat}{2} \normsq{\xt - \xtt} - \tfrac{\etat}{2} \normsq{\xtt - \hx}  = \inner{\etat (\xt - \xtt)}{\xt- \hx}$ , we get
\begin{equation}\label{pr1:Qx}
\Q[x](\zt;\hz) + \tfrac{\etat + \alpha}{2} \normsq{\xt - \hx} \leq \tfrac{\etat}{2} \normsq{\xtt - \hx} - \tfrac{\etat}{2} \normsq{\xt - \xtt}.
\end{equation}
Additionally, being a maximizer, $\pt$ satisfies 
\begin{align}
\pt \in& \argmax_{p \in P} \sumi p_i (\inner{\xt}{A_i\piit} - \fistar(\piit))  - \rhostar(p) \notag\\
&\Rightarrow \sumi (\hpi - p^t) (\inner{\xt}{A_i \piit} - \fistar(\piit)) + \rhostar(\pt) - \rhostar(\hp)\leq 0 \Rightarrow Q_p(\zt;\hz) \leq 0, \label{pr1:Qp}
\end{align}
So, combining  \eqref{pr1:Qp} and \eqref{pr1:Qx}, taking a $\wt$-weighted sum of the resulting inequality and using $\wt\etat \leq \wtt (\etatt + \alpha)$, we obtain
\begin{equation}\label{eq:central_Qx_Qp}
\sumt \wt (\Q[x](\zt; \hz)+\Q[p](\zt; \hz)) + \tfrac{\wt[N](\etat[N] + \alpha)}{2} \normsq{\xt[N] - \hx} \leq \tfrac{\wt[1]\etat[1]}{2} \normsq{\xt[0] - \hx} - \sumt \tfrac{\wt\etat}{2} \normsq{\xt - \xtt}.
\end{equation}
Then utilizing \eqref{eq:aggregate_stx_cvxit_cst} and the Young's inequality, \eqref{eq:central_Q} follows immediately by adding \eqref{eq:central_Qx_Qp} to \eqref{eq:central_Qpii_bound}.
\end{proof}
\vgap

We now apply the result in Proposition~\ref{pr:central_Q_conv} to the smooth problem.
 Observe that the gradient evaluation point $\xundert$ is common for all workers. This allows us to easily characterize the strong convexity modulus of the aggregate prox-penalty function $\sumi \hpi \Dfistar(\cdot;\cdot)$ (c.f. \eqref{eq:aggregate_stx_cvxit_cst}) in the next lemma. 
\begin{lemma}\label{lm:agg_sm}
Let $\hp\in P$ be given. If $\pii = \grad \myfi(x) $ and $\bpii = \grad \myfi(\xbar)$ for some $x$ and $\xbar$, then 
\begin{equation}
\sumi \hpi \Dfistar(\pii;\bpii) \geq \tfrac{1}{2\Lf}\normsq{\sumi \hpi \Aitr (\pii - \bpii)}.
\end{equation}
\end{lemma}

\begin{proof}
Let $\myfi[\hp](x):=(\sumi \hpi \myfi(x)).$ Then by the definition of $\Lf$ in \eqref{eq:smo_cst}, $\myfi[\hp]$ is $\Lf$-smooth and its conjugate $(\myfi[\hp])^*$ is $1/\Lf$ strongly convex. 
Next, we relate  $W_{(\myfi[\hp])^*}$ to $\sumi \hpi \Dfistar$ to calculate its strong convexity modulus. Since 
$\sumi \hpi \pii = \grad \myfi[\hp](x)$ and $ \sumi \hpi \bpii = \grad \myfi[\hp](\xbar),$
we have by Fenchel duality that
\begin{align*}&\sumi \hpi \inner{\pii -\bpii}{\grad \fistar(\bpii)} =  \inner{\sumi \hpi (\pii -\bpii)}{\xbar} = \inner{\sumi \hpi (\pii -\bpii)}{\grad (\myfi[\hp])^*(\sumi \hpi \bpii) },\\
&\sumi \hpi \fistar(\pii) = \sumi \hpi (\inner{\pii}{x} - \myfi(x)) = \inner{\sumi \hpi\pii}{x} - (\sumi \hpi \myfi)(x) = (\myfi[\hp])^*(\sumi \hpi\pii),
\end{align*}
and, similarly, $\sumi \hpi \fistar(\bpii) =  (\myfi[\hp])^*(\sumi \hpi\bpii).$
Thus 
\begin{align*}
\sumi \hpi &\Dfistar(\pii;\bpii) = (\myfi[\hp])^*(\sumi \hpi\pii) -  (\myfi[\hp])^*(\sumi \hpi\bpii) - \inner{\sumi \hpi (\pii -\bpii)}{\grad (\myfi[\hp])^*(\sumi \hpi \bpii) }\\
&= W_{(\myfi[\hp])^*}(\sumi \hpi\pii;\sumi \hpi\bpii) \geq \tfrac{1}{2\Lf}\normsq{\sumi \hpi(\pii - \bpii)}=\tfrac{1}{2\Lf}\normsq{\sumi \hpi\Aitr(\pii - \bpii)},\end{align*}
where the last inequality follows from $\Aitr = I$ in smooth problems. 
\end{proof}

\vgap

We are now ready to prove Theorems~\ref{thm:central-sm} and \ref{thm:central-sm-str}.

\vgap

\noindent{\bf Proof of Theorem~\ref{thm:central-sm}:}
We apply Proposition \ref{pr:central_Q_conv} to obtain the convergence result in \eqref{eq:smooth_f_conv}.
Since $\fistar$ is $1$-strongly convex with respect to $\Dfistar$, $\mu=1$ satisfies condition \eqref{eq:strong_convex_dual}. Since $\piit=\grad f_i(\xundert)$ and $\hpii = \grad f_i(\bar x)$ for some $\bar x$, $q = L_f$ satisfies condition \eqref{eq:aggregate_stx_cvxit_cst} (c.f. Lemma \ref{lm:agg_sm}). Moreover, since stepsizes proposed in Theorem \ref{thm:central-sm} verifies \eqref{req:DRAO}, all the requirements in Proposition \ref{pr:central_Q_conv} are met. Thus Proposition \ref{pr:central_Q_conv} leads to \eqref{eq:smooth_Q_conv}, i.e.,
$$\sumt \wt \La(\xt; \hp, \hat{\pi}) - \sumt \wt \La(\xstar; \pt, \pi^t)  + \Lf \normsq{\xt[N] - \xstar}\leq \Lf \normsq{\xt[0] - \xstar}. $$
\vgap
In particular, with $\hpiiN = \grad \myfi(\xbart[N])$ and $\hpiN[] \in \argmax_{p \in P} \sumi p_i \myfi(\xbart[N])$  such that  $f(\xbarN)= \La(\xbarN; \hpN, {\hat \pi}^N)$ (see Lemma \ref{lm:duality}), we have
$$\sumt \wt \La(\xt; \hpN, \hat{\pi}^N) - \sumt \wt \La(\xstar; \pt, \pi^t) \leq \Lf \normsq{\xt[N] - \xstar}. $$
Because $\La(\cdot; \hpN, \hat{\pi}^N)$ is convex with respect to $x$, the first term satisfies
$$\sumt \wt \La(\xt; \hpN, \hat{\pi}^N) \geq \tfrac{N(N+1)}{2} \La(\xbarN; \hpN, \hat{\pi}^N) = \tfrac{N(N+1)}{2} f(\xbarN). $$
Due to the weak duality in Lemma \ref{lm:duality}, the second term is upper bounded by 
$$\sumt \wt \La(\xstar; \pt, \pi^t)\leq \sumt \wt f(\xstar)  = \tfrac{N(N+1)}{2} f(\xstar).$$
Then the desired inequality in \eqref{eq:smooth_f_conv} follows immediately.

\endproof

\vgap

\noindent{\bf Proof of Theorem~\ref{thm:central-sm-str}:}
Similar to the preceding proof, the proposed stepsizes (c.f. \eqref{stp:central_sm_str}), together with $\mu=1$ and $q=\Lf$, verify the requirements in Proposition \ref{pr:central_Q_conv}, thus 
 $$\sumt \wt (\La(\xt; \hp, \hat{\pi}) -  \La(\xstar; \pt, \pi^t))  + \wt[N](\eta + \alpha) \normsq{\xt[N] - \xstar}\leq \wt[1](\eta \normsq{\xt[0] - \xstar} + \tau \sumi \hpi \Dfistar(\hpii;\piit[0])). $$
Using the relation of conjugate Bregman distance functions and the identity $\grad (\sumi \hpi f_i)(\bar x) = \sumi \hpi \grad f_i(\bar x)$, the last term can be upper bounded by
\begin{align*}
 \sumi & \hpi \Dfistar(\hpii;\piit[0]) = \sumi \hpi (W_{f_i}( \xt[0]; \bar x)) \\
 &= (\sumi \hpi f_i)(\xt[0]) - (\sumi \hpi f_i)(\bar x) - \inner{\grad (\sumi \hpi f_i)(\bar x)}{\xt[0] - \bar x}\\
 &\leq \tfrac{\Lf}{2} \normsq{x^0 - \bar x}. 
 \end{align*}
 Thus the $Q$ convergence bound in  \eqref{eq:smooth_Q_str_conv} follows immediately from 
 $$\sumt \wt Q(\zt; \hz) +  \frac{\alpha(\sqrt{4\kappa + 1} - 1)}{4\theta ^N} \normsq{\xt[N] - \hx} \leq \frac{(\sqrt{4\kappa + 1} - 1)}{4} (\alpha\normsq{\xt[0] - \hx} + \Lf \normsq{x^0 - \bar x}).$$
Additionally, setting the preceding $\hz$ to the saddle point $\zstar$ defined in Lemma \eqref{lm:sad_ex} (c.f. \eqref{eq:saddle_point}) such that $Q(\zt; \zstar) \geq 0\ \forall t$ and dividing both sides by $\frac{\alpha(\sqrt{4\kappa + 1} - 1)}{4\theta ^N}$, the geometric convergence of $\xt[N]$ to $\xstar$ in \eqref{eq:smooth_x_str_conv} can be deduced.
\endproof
\vgap
Now we move on to present convergence proofs for structured non-smooth problems.
\vgap\\
\noindent{\bf Proof of Theorem~\ref{thm:cen-non-smooth}}
First, we consider part a) with a non-strongly convex  $u(x)$. The result is also a consequence of Proposition \ref{pr:central_Q_conv}. Since $\Vi(\pii; \bpii):=\tfrac{1}{2} \normsq{\pii - \bpii} \geq \tfrac{1}{2\MA^2} \normsq{\Aitr(\pii - \bpii)}$, the Jensen's inequality implies the condition \eqref{eq:aggregate_stx_cvxit_cst} is satisfied with $q=\MA^2$. Since $\fistar$ is convex, the condition \eqref{eq:strong_convex_dual} is satisfied with $\mu=0$. Additionally, since the chosen stepsizes in Theorem \ref{thm:cen-non-smooth} satisfy the condition \eqref{req:DRAO}, all the requirements for Proposition \ref{pr:central_Q_conv} are met. Thus for any feasible $\hz:=(\hx;\hp, \hpii[])$, we have
$$\sumt \wt Q(\zt; \hz) + \tfrac{\wt[N] \etat[N]}{2} \normsq{\xt[N] - \hx} \leq \tfrac{\wt[1] \etat[1]}{2} \normsq{\xt[0] - \hx} + \tfrac{\wt[1]\taut[1]}{2}\sumi \hpi \normsq{\hpii- \piit[0]}.$$
Let $\hpiiN \in  \argmax_{\pii \in \Pii} \inner{\pii}{\Ai\xbart[N]} - \fistar(\pii)$ and $\hpiN[] \in \argmax_{p \in P} \sumi p_i \myfi(\xbart[N])$ such that $f(\xbarN)= \La(\xbarN; \hpiN, \hpiiN)$ (c.f. Lemma \ref{lm:duality}). Setting $\hz$ to $\hz^N:=(\xstar; \hpN, \hat{\pi}^N)$ leads to 
$$\sumt \wt Q(\zt; \hz^N)  \leq \tfrac{\etat[1]}{2} \Rx^2 + \tfrac{\taut[1]}{2} \DPi^2 = \Rx\MA \DPi.$$
Then the resulting convergence bound in \eqref{eq:nonsmooth_f_conv} can be deduced from the fact  $(\sumt \wt) f(\xbarN) - f(\xstar) \leq \sumt \wt Q(\zt; \hz^N)$. 

As for part b), the derivation is the same except for the different stepsize choice to take advantage of the $\alpha$-strong convexity of $u(x)$.
\vgap
\endproof
\vgap

\section{The DRAO-S method}
The practical application of the DRAO method is limited by the exact computation to the following saddle point problem in Line 5 of Algorithm \ref{alg:cPD}:
\begin{equation}\label{eq:saddle_subproblem}
 \xt \leftarrow \argmin_{x \in X} \max_{p \in P} \sumi p_i [\inner{x}{\vit} - \fistar(\piit)] - \rhostar(p)+ u(x)+ \tfrac{\eta_t}{2} \normsq{x - \xtt}.
 \end{equation}
We relax it by assuming only 
the ability to efficiently compute the $p$-prox mapping defined in \eqref{eq:prox-def}. \textblue{In the DRO setting, the efficient implementations for risk measures induced by several probability uncertainty sets are described in \cite{zhang2019efficient}. For the entropic risk measure, if we select prox-function to be $U(p; \bar p):= \sum_{i=1}^m p_i \log(p_i/\bar{p}_i)$, the computation amounts to a softmax evaluation. For the mean semi-deviation risk of order two, the computation can be implemented as a quadratically constrained quadratic program (QCQP).}

In this section, we  design a novel saddle point sliding (SPS) subroutine to solve \eqref{eq:saddle_subproblem} inexactly in the DRAO method and call the 
resulting method {\sl distributed risk-averse optimization with sliding (DRAO-S)}. 
We show  the DRAO-S method maintains the same order of communication complexity as the DRAO method. Moreover, the total number of $P$-projections required by the DRAO-S method will be mostly optimal, in the sense that it is equivalent to the optimal one required for solving problem~\eqref{eq:orig_prob} with linear local cost functions $f_i$'s. 


\subsection{The Algorithm and Convergence Results}
The SPS subroutine for solving \eqref{eq:saddle_subproblem} inexactly is presented in Algorithm \ref{alg:SPS}.
It is closely related to the classic primal dual (PD) algorithm (see \cite{chambolle2011first,LanBook}) for solving a structured bilinear saddle point problem given by 
$$\min_{y \in X} \max_{p \in P} \inner{v^t y}{p} - \rhostar(p),$$
\textblue{where the matrix $v^t$ is obtained from stacking $(v_1^t)^\top, (v_2^t)^\top, \ldots, (v_m^t)^\top$ from Line 3 of Algorithm \ref{alg:cPD} vertically. }
 In iteration $s$, the subroutine computes an extrapolated prediction of $\sumi p_i^s \vit$ in Line 2, a $y$-prox update in Line 3, and then a $p$-prox update in Line 4. The $p$-prox update utilizes a general Bregman distance function $U$ to allow a suitable choice to take advantage of the geometry of $P$. At the end of $\St$ iterations, the SPS subroutine returns a weighted ergodic average of $\{y^s\}$ as an approximate solution to \eqref{eq:saddle_subproblem}.

\begin{minipage}{0.9\textwidth}
\begin{algorithm}[H]
\caption{\textbf{S}addle \textbf{P}oint \textbf{S}liding (SPS) Subroutine}
\label{alg:SPS}
\begin{algorithmic}[1]
\Require Initial points $\xtt, y^0 \in X$, $p^0, p^{-1} \in P$, and gradients $\{\vit\}, \{\vitt\}$. Non-negative stepsizes $\etat$, $\{\delta_{s}\}$, $\{\gam_{s}\}$ and $\{\beta_{s}\}$, averaging weights $\{q_{s}\}$, and iteration number $S_t$.

\For{$s =  1,2,3 ... S_t$}
\parState{$\tilde{v}^s \leftarrow\begin{cases}
\sumi p^0_{i} \vit + \delta_{1} \sumi (p_i^0 - p_i^{-1}) \vitt & \text{if } s = 1,\\
\sumi p^{s-1}_{i} \vit + \delta_{s} \sumi (p_i^{s-1} - p_i^{s-2}) \vit & \text{if } s \geq 2.
\end{cases}$}
\parState{$y^s \leftarrow \argmin_{y \in X} \inner{y}{\tilde{v}^s} + u(y) + \tfrac{\beta_s}{2} \normsq{y - y^{s-1}} + \tfrac{\etat}{2}\normsq{y - \xtt}.$}
\parState{$p^s \leftarrow \argmax_{p \in P} \sumi p_i (\inner{\vit}{y^s} - \fistar(\piit)) -\rhostar(p)- \gamts[s][] U(p; p^{s-1}).$}
\EndFor
\State \Return $x^t:= \sums q_s y^s / (\sums q_s)$, $y^t:=y^{S_t}$, $\bar{p}^t:= \sums q_s p^s / (\sums q_s)$, $\ptlt:=p^{S_t}$ and  $\dbtilde{p}^t := p^{S_t-1}.$
\end{algorithmic}
\end{algorithm}
\vgap
\end{minipage}


The DRAO-S method is obtained by making the following two modifications to DRAO. First, we require additional initial points $(y^0, \ptlt[0], \pdbtlt[0]) \in  X \times P \times P$. For simplicity, we set $y^0= x^0$ and $\ptlt[0] = \pdbtlt[0]$. Second, we replace the definition of $x^t$ in Line 5 of DRAO
with the output solution of the SPS subroutine
according to: 
\begin{equation}\label{subroutine:SPS_sm}
  \begin{split}
  &\textbf{}(x^t, y^t, \bar p^t, \ptlt, \pdbtlt) = SPS(x^{t-1}, y^{t-1}, \ptltt, \pdbtltt,
  \{\vit\}, \{\vitt\}\ | \ 
  \etat, \{\delta^t_s\}, \{\gam_s^t\}, \{\beta_s^t\}, \{q_s^t\}, \St)\ \forall t\geq 1.\end{split}
\end{equation} 
At the beginning, i.e., $t=1$, we set $\vit[0] = \vit[1] \ \forall i \in [m].$
Due to its similarity to the DRAO method, we call the outer loop of the DRAO-S method (Algorithm \ref{alg:cPD} with modification \eqref{subroutine:SPS_sm}) the outer DRAO loop and say a phase of the DRAO-S method happens if $t$ is increased by $1$.
Accordingly, we call the inner loop of the DRAO-S method (Algorithm \ref{alg:SPS}) the inner SPS loop and say an (inner) iteration happens if $s$ is incremented by 1.
Intuitively, if the inner iteration limits $S_t$ were large enough, $x^t$ obtained from the inner SPS loop would
be a good approximate solution of \eqref{eq:saddle_subproblem}.
However, the $P$-projection complexity, i.e., the total number of inner iterations given by $\sum_t S_t$, might be too large. In addition, notice that the computation burdens of the server subproblem and the worker subproblem during each communication round are still different: the server requires several rounds of $P$ and $X$-projections for the SPS subroutine, while the worker needs just one $\pii$-prox mapping. However, since the server is often more powerful, e.g. the cloud-edge system, this asymmetry may have a limited impact on the overall computation performance of the system. 

We note here that the DRAO-S method is related to the Primal Dual Sliding (PDS) method in \cite{lan2021graph}, where the sliding subroutine is also a primal dual type algorithm. However, since we are dealing with a nested trilinear saddle point problem, rather than the sum of two bilinear saddle point problems, the DRAO-S method differs from the PDS method in two important ways. First, the linear operator $\vt$ for the saddle point subproblem (c.f. \eqref{eq:saddle_subproblem}) changes in every phase. To coordinate consecutive inner SPS loops, we construct a special momentum term from both the current $\vt$ and the previous $v^{t-1}$ when transitioning into a new phase, i.e., Line 2 of Algorithm \ref{alg:SPS}. This construction is inspired by the novel momentum term in the SD method \cite{zhang2019efficient}. Second, as opposed to the single initialization in both the PDS method and the gradient sliding method \cite{lan2016gradient}, the $y$ prox update in Line 3 of Algorithm \ref{alg:SPS} utilizes two distinct initialization points, the ergodic average $\xt[t-1]$ and the last iterate  obtained in the last phase, $\yt[0]$. Even though both are approximate solutions to \eqref{eq:saddle_subproblem}, $y^s$ is used only in the inner SPS loop while the ergodic average $\xt$ is used in both the outer loop and the inner loop.  
As will be discussed in the next subsection, the additional initialization point appears to significantly simplify  the convergence analysis and the selection of stepsizes in comparison to \cite{lan2016gradient,lan2021graph}. 

Now, let us consider the smooth problem. The stepsizes associated with the inner SPS loop need to adapt dynamically in two aspects.  First,
the inner iteration limit $S_t$ needs to be an increasing function of $t$ to maintain the same communication complexity as the DRAO method. 
Second, as a primal dual type algorithm, the inner SPS loop stepsizes, $\gamts$, $\delts$ and $\betats$, need to satisfy a certain condition related to the operator norm of $\vt$, i.e., $\gamts[s-1] \betats \geq \delts \norm{v^t}^2$, to ensure convergence. Specifically,  if the Bregman distance function $U$ in Line 4 of Algorithm \ref{alg:SPS} is 1-strongly convex with respect to $\norm{\cdot}[U]$,  the operator norm of interest is given by 
\begin{equation}\label{def:Mt-def}
\MtU:=\norm{v^t}[2, U^*]:=\max_{\norm{p}[U] \leq 1, \norm{y}\leq 1}\tsum_{i=1}^{m} p_i (\vit) ^\top y.
\end{equation}
 Since a uniform bound on $\MtU$ may not exist if $X$ is unbounded under the smooth setting, we choose $\gamts$ and $\betats$ to adjust dynamically to $\MtU$  in each phase. 
In particular, the next theorem presents the stepsize choice and the convergence result under the  non-strongly convex setting. 

\begin{theorem}\label{thm:sps-sm}
Let a smooth risk-averse problem (c.f. \eqref{eq:orig_prob}) be given.  Let $\MtU$ and $\Lf$ be defined in \eqref{def:Mt-def} and \eqref{eq:smo_cst}, and let $\Rx$ and $\DP$ be defined in Subsection \ref{subsec:notation}. 
If $\{\xt\}_{t=1}^N$ are generated by the DRAO-S method \eqref{subroutine:SPS_sm}  with the following stepsizes:
\begin{align}\label{stp:sps_sm}
\begin{split}
&\wt=t,\ \thetat=(t-1)/(t), \taut = (t-1)/2, \etat = 2\Lf/t,\\
&\Delta > 0,\ \St=\ceil{t\Delta  \MtU },\ \bMtU= \tfrac{\St}{t \Delta },\ \beta_s^t = \beta^t = \tfrac{\DP \bMtU}{\Rt[0]},\ \gam_s^t=\gamt=\tfrac{\Rt[0]\bMtU}{\DP},\\ 
& q_s^t = 1,\ \delta_1^t=\small\begin{cases} \bMtU / \bMtU[t-1] &\text{if } t \geq 2 \\ 1 &\text{if } t = 1 
\end{cases},\ \text{and}\  \delta_s^t = 1\ \forall\ s\geq 2 ,
\end{split}
\end{align}
then the solution $\xbart[N]$ returned by the outer DRAO loop satisfies 
\begin{equation}\label{eq:f_SPS_sm}
f(\xbart[N]) - f(\xstar) \leq \tfrac{2\Lf \Rt[0]^2}{N(N+1)} + \tfrac{2 \DP \Rt[0]}{N(N+1)\Delta}, \forall N \geq 1,  \end{equation}
and there exists an uniform upper bound $\Mtil$ for $\MtU$, i.e., $ \Mtil \geq \MtU,\ \forall t \geq 1$. In addition, if $\Delta= \DP /( \Lf \Rx)$, the convergence bound can be simplified to 
\begin{equation}\label{eq:f_SPS_sm1}
f(\xbart[N]) - f(\xstar) \leq \tfrac{4\Lf \Rt[0]^2}{N(N+1)}, \forall N \geq 1.  \end{equation}
\end{theorem}

A few remarks are in place regarding the above result. 
First, the stepsizes in the outer DRAO loop are exactly the same as that of Theorem \ref{thm:central-sm}. Since each phase requires only two rounds of communication,  the DRAO-S method  has a communication complexity of $\bigO(\sqrt{\Lf} \Rt[0]/\sqrt{\ep})$. Second, $\bMtU$ defined in \eqref{stp:sps_sm} is the smallest upper bound of $\MtU$ needed to make the inner iteration limit  $\St$ an integer. The factor $\Delta$ in \eqref{stp:sps_sm} represents the conversion factor between the $P$-projection complexity and the communication complexity ($\Pi$ projection complexity). A communication complexity of the order $\bigO(1/\sqrt{\ep})$ can be maintained for any  $\Delta>0$ and the specific choice in \eqref{stp:sps_sm} is needed only for optimal constant dependence. Third, since the number of phases needed to find an $\ep$-optimal solution is bounded by $N_\ep:=2\sqrt{\Lf}\Rx/\sqrt{\ep} $, the total number of $P$-projections is given by 
{\skipdisplay $$\tsum_{t=1}^{N_\ep}\ceil{t \Delta \MtU}  \leq \Delta \Mtil N_\ep^2 + N_\ep = \bigO(\DP \Mtil \Rx/\ep).$$}
Fourth, both the inner SPS loop stepsizes, $\betats$, $\gamts$ and $\delts$, and the inner iteration limit $\St$ adjust dynamically to the varying operator norm $\MtU$ characterizing the difficulty of the saddle point problem \eqref{eq:saddle_subproblem} of each phase. Specifically, when the saddle point problem is easy, i.e., $\MtU$ is small, $\gamts$, $\betats$, and $\St$ become small so that a small number of inner iterations is performed, and vice versa. 
   Thus, when most $\MtU$'s are significantly smaller than the upper bound $\Mtil$,  the total number of $P$-projections can be much smaller than $\bigO(\DP \Mtil \Rx/\ep)$. Such a saving is possible because $\St$ can compensate for the changing $\betats$ and $\gamts$ such that the the effective proximal penalty parameter, $\wt \gamts/\St$ and $\wt \betats / \St$, remains constant across phases.
In contrast, it is difficult for single loop primal dual type algorithms, such as the SD method \cite{zhang2020optimal}, to adjust dynamically to the varying operator norm of $\vt$ in each iteration.

The following theorem presents the stepsize choice and the convergence result under the  strongly convex setting. 
\vgap


\begin{theorem}\label{thm:sps-sm-str}
Let a smooth problem $f$ (c.f. \eqref{eq:orig_prob}) with $\alpha>0$ be given. Let $\Rx$ and $\DP$ be defined in Subsection \ref{subsec:notation}. Let the smoothness constant $\Lf$ be defined in \eqref{eq:smo_cst} such that $\kappa:=\Lf/\alpha$ denotes the condition number. If $\{\xt\}_{t=1}^N$ are generated by the DRAO-S method \eqref{subroutine:SPS_sm}  with the following stepsize:
\begin{equation}\label{stp:sps_sm_str}
\begin{split}
&\thetat=\theta := \tfrac{\sqrt{8\kappa + 1}-1}{\sqrt{8\kappa + 1}+1},\ \wt=(\tfrac{1}{\theta})^{t-1},\ \taut =\tau:= \tfrac{\sqrt{8\kappa + 1} - 1}{2},\ \etat =\eta:= \tfrac{\alpha(\sqrt{8\kappa + 1} - 1)}{4},\\
& \betats =\tfrac{\alpha (s-1)}{4} ,\ \gamts=\tfrac{2\St(\St + 1) }{\wt \alpha s \Delta },\ \delta_s^t=(s-1)/s,\\
&\qts = s,\ S_t =  \ceil{( 2 \wt \Delta)^{1/2}\MtU},\ \Delta > 0,
\end{split}
\end{equation}
the last iterate $\xt[N]$ converges geometrically:
\begin{equation}\label{eq:sps_x_sm_str_conv}
\normsq{\xt[N] - \xstar} \leq \theta^N \left((1 + 2\kappa)\normsq{\xt[0] - \xstar} + \tfrac{4\DP^2 }{\alpha\eta \Delta }\right), \forall N \geq 1,\end{equation}
and there exist an $\tilde M$ such that $\MtU \leq \tilde M\ \forall t \geq 1.$ If $\Delta= 2  \DP^2 / \eta \Lf \Rt[0]^2$ and $\kappa \geq 1$,  the total number of $P$-projections required to find an $\epsilon$-close solution, i.e., $\norm{x^N - \xstar} \leq \epsilon$, is bounded by $\bigO(\tfrac{\kappa^{1/4}\Mtil \DP}{\alpha\sqrt{\ep}} + \sqrt{\kappa} \log(\tfrac{1}{\ep}))$. 
\vgap\end{theorem}

Note that the strong convexity modulus $\alpha$ is split into two to accelerate both the outer DRAO loop and the inner SPS loop. For the outer DRAO loop, the proposed stepsize is the same as that of  Theorem \ref{thm:central-sm-str} if $\alpha/2$ is viewed as the strong convexity modulus. This allows the DRAO-S method to maintain the same order of communication complexity, i.e., $\bigO(\sqrt{\kappa} \log(1/\ep))$. For the inner SPS loop, the stepsizes are similar to that of the accelerated primal-dual method (see \cite{LanBook}) with a strong convexity modulus of $\alpha/2$. 
Moreover, similar to Theorem~\ref{thm:sps-sm}, the stepsize $\gamts$ and the inner iteration limit $\St$ adjust dynamically to the varying operator norm $\bMtU$ in each phase.
It is also worth noting that the constant dependence of the $P$-projection complexity on $\kappa$ in Theorem \ref{thm:sps-sm-str} is larger than the optimal by a factor of $\kappa^{1/4}$. This complexity can be further improved to $\bigO({\Mtil \DP}/{\alpha\sqrt{\ep}} + \sqrt{\kappa} \log({1}/{\ep}))$ if some stepsize choice similar to Theorem 6 of \cite{zhang2022solving} is utilized.  


Next, let us move on to the structured non-smooth problem. Since $\Pi$ is assumed be bounded, the following uniform upper bound of $\MtU$ (c.f. \eqref{def:Mt-def}) is useful for convergence analysis,
\begin{equation}\label{def:Mpibar}
\MAPi = \max_{\pi \in \Pi} \{\norm{[\Aitr[1] \piit[t][1];\ldots; \Aitr[m] \piit[t][m] ]}[2, U^*] := \max_{\pi \in \Pi}\max_{\norm{y}[2] \leq 1, \norm{p}[U] \leq 1} \tsum_{i=1}^{m} p_i \inner{\Aitr \pii}{y}\}.
\end{equation} 
Specifically, the stepsize choices and the convergence properties of the DRAO-S method, applied to  both  non-strongly and  strongly convex settings, are presented in the next theorem.

\begin{theorem}\label{thm:sps-ns}
Let a structured non-smooth problem $f$ \eqref{eq:orig_prob} be given. Suppose $\DPi$, $\MA$, $\MtU$, and $\MAPi$ are defined in \eqref{eq:MA_def}, \eqref{def:Mt-def} and \eqref{def:Mpibar}, and suppose $\DP$ and $\Rx$ are defined in Subsection \ref{subsec:notation}. \\
a) If $\alpha=0$ and the stepsizes are given by
\begin{equation}\label{stp:drao-s-ns}
    \begin{split}
        &\wt=1,\ \thetat=1, \etat = \MA \DPi / 2\Rt[0], \taut = \MA \Rt[0] / 2\DPi,\\
        &\Delta > 0, \St= \ceil{\MtU \Delta},\ \bMtU:= \tfrac{\St}{\Delta},\ \beta_s^t = \beta = \tfrac{\DP \bMtU}{\Rt[0]},\ \gam_s^t=\gam=\tfrac{\Rt[0]\bMtU}{\DP},\\ 
        & q_s^t = 1,\ \delta_1^t=\small\begin{cases} \bMtU / \bMtU[t-1] &\text{if } t \geq 2 \\ 1 &\text{if } t = 1 \end{cases},\ \text{and}\  \delta_s^t = 1\ \forall\ s\geq 2 ,
    \end{split}
\end{equation}
 the solution $\xbarN$ returned by the DRAO-S method satisfies
\begin{equation}\label{eq:sps-f-ns}
f(\xbarN) - f(\xstar) \leq \tfrac{ \MA \DPi \Rt[0]}{N} + \tfrac{\DP  \Rt[0]}{\Delta N}, \forall N \geq 1.
\end{equation}
In particular, if $\Delta = \tfrac{\DP }{\MA \DPi}$, the convergence bound can be simplified to 
\begin{equation}\label{eq:sps-f-ns1}
f(\xbarN) - f(\xstar)  \leq \tfrac{2 \MA \DPi \Rt[0]}{N}\ \forall N \geq 1.
\end{equation}
b) If $\alpha > 0$ and the stepsizes are given by
\begin{align}\label{stp:drao-sps-ns-str}
\begin{split}
&\wt=t,\ \thetat=(t-1)/t,\ \etat = t\alpha/6,\ \taut = 6 / t\alpha,\\
&\Delta > 0 \ , \St = \ceil{\Delta \bMAPiU^2},\ \gamt_s = \gamt := 4 \bMAPiU^2 /\alpha t,\ \\
&\betats = \small \begin{cases}\tfrac{\alpha}{4}(t-1) & \text{if } s = 1 \\ \tfrac{\alpha}{4} t\  & \forall\ s \geq 2 \end{cases},\ 
\delts = \small \begin{cases}\tfrac{t-1}{t} &\text{if } s=1\\ 1 &\forall\ s \geq 2\end{cases}, \ \qts = 1,   
\end{split}
\end{align}
 the solution $\xbarN$ returned by the DRAO-S method satisfies
\begin{equation}\label{eq:sps-f-ns-str}
\begin{split}
f(\xbarN) - f(\xstar) &\leq \tfrac{1}{N(N+1)} (\tfrac{\alpha}{6} \Rt[0]^2 + \tfrac{6 \MA^2 \DPi^2}{\alpha} + \tfrac{4 \bMAPiU^2 \DP^2}{\alpha \Delta}).
\end{split}
\end{equation}
In particular, if $\Delta = { \DP^2}/{(\MA^2\DPi^2 + \alpha^2 \Rt[0]^2/6)}$, the convergence bound can be simplified to 
\begin{equation}\label{eq:sps-f-ns-str1}
f(\xbarN) - f(\xstar) \leq \tfrac{1}{N(N+1)} (\alpha \Rt[0]^2 + \tfrac{10 \MA^2\DPi^2}{\alpha}).
\end{equation}
\end{theorem}

A few remarks are in place.
First, observe the above inner iteration limit $\St$   adjust dynamically to the operator norm $\MtU$ for the non-strongly convex case \eqref{stp:drao-s-ns}, but not for the strongly convex case \eqref{stp:drao-sps-ns-str}. This shortcoming is an artifact of the order of prox updates in Algorithm \ref{alg:SPS}. If a $p$-prox update, utilizing a $y$-momentum prediction, is performed before the $y$-prox update,  $\St$ can be chosen to be $\ceil{\Delta \MtU^2}$ to achieve the same effect. 
Moreover, since two communication rounds is required for each phase, the preceding result implies an $\bigO(1/\ep)$ ($\bigO(1/\sqrt{\ep})$) communication complexity when $\alpha=0$ (resp. $\alpha>0$). Since the inner iteration limit $\St$ is bounded, it also implies an $\bigO(1/\ep)$ (resp. $\bigO(1/\sqrt{\ep})$) $P$-projection complexity. In particular, with the specific choices of $\Delta$ shown above, the DRAO-S method can achieve
 the optimal constant dependence on problem parameters, that is, $\bigO(\Rt[0]\MA\DPi/\ep)$ communication and $\bigO(\bMAPiU \DP \Rt[0]/\ep)$ $P$-projection complexities when $\alpha=0$, and  $\bigO(\MA\DPi/\sqrt{\ep \alpha})$ communication and $\bigO(\bMAPiU \DP /\sqrt{\ep\alpha})$ $P$-projection complexities when $\alpha>0.$


\subsection{Convergence Analysis}
Our goal in this subsection is to establish the convergence rates of the DRAO-S method stated in Theorem \ref{thm:sps-sm}, \ref{thm:sps-sm-str} and \ref{thm:sps-ns}.\\
First, we present a recursive bound to characterize the convergence property of each inner SPS loop under both  the non-strongly convex ($\alpha = 0$) and the strongly convex ($\alpha>0$) settings. 

\begin{proposition}\label{pr:one-step-sliding}
Fix a $t\geq 1$. Let $\MtU := \norm{\vt}[2, U^*]$
and $\MttU := \|\vt[t-1]\|_{2, U^*}$. 
If the SPS stepsizes in Algorithm \ref{alg:SPS} satisfy: 
\begin{align}\label{req:SPS_phase}
\begin{split}
&\delta_s = q_s / q_{s-1},\ \beta_s \gam_{s-1} \geq \delta_{s} \MtU^2, \forall s \geq 2,\\
&q_{s-1} (\beta_{s-1} + \alpha/2) \geq q_s \beta_s, \ q_{s-1} \gam_{s-1} \geq q_{s} \gam_{s}, \forall s \geq 2,
\end{split}
\end{align}
 the generated iterates, $\{(\yt[s], \pt[s])\}$ and $(\xt, \bar{p}^t)$, satisfy the following relation for all $x \in X$, $p \in P$ and $S\geq 1$:
\begin{align}\label{eq:SPS_phase_bound}
 \begin{split}
 (\tsum_{s=1}^{S} q_s)& [\La(x^t; p, \pi^t) - \La(x; \bar{p}^t, \pi^t)]  + \delta_1 q_1 \inner{y^0 - x}{\tsum_{i=1}^m \vit[t-1] (p_i^0 - p_i^{-1})} -q_S \inner{y^{S} - x}{\tsum_{i=1}^m \vit (p_i^{S} - p_i^{S - 1})}\\
   &  +\tfrac{(\tsum_{s=1}^{S} q_s)}{2} [{\eta_t} \normsq{\xt - \xtt} + (\eta_t + \alpha/2) \normsq{\xt - x} - \etat \normsq{\xtt -x}]\\ 
   \leq& q_1\gam_1 U(p; p^0) - q_S \gam_S [U(p;p^S) + \tfrac{1}{2}\normsq{p^S- p^{S-1}}[U]]  +\tfrac{q_1 \delta_1^2 \MttU^2}{2\beta_1} \normsq{p^{0} - p^{-1}}[U] \\
   &-\tfrac{1}{2}[q_S (\beta_S + \alpha/2) \normsq{y^S - x}  - q_1 \beta_1 \norm{y^0 - x}^2].
   \end{split}
\end{align}
\vspace{-5mm}
\end{proposition}

\begin{proof}
Fix points $x \in X$ and $p \in P$. First, consider the convergence of $y^s$. Since $u(y) + \etat \normsq{y-x^{t-1}}/2$ has a strong convexity modulus of $\alpha + \etat$, the $y$-proximal update in Line 3 of Algorithm \ref{alg:SPS} leads a three-point inequality (see Lemma 3.1 of \cite{LanBook}): 
{\skipdisplay
\begin{align*}
\inner{y^s - x}{\tilde{v}^s}& + u(y^s) - u(x) + 
\tfrac{1}{2}[(\beta_s+\alpha+\etat) \normsq{x - y^{s}} + \beta_s \normsq{y^s - y^{s-1}} - \beta_s \normsq{y^{s-1} -x }]\notag\\
&+ \tfrac{\etat}{2}(\normsq{y^s - \xtt} - \normsq{x - \xtt}) \leq 0. \notag
\end{align*}}
Equivalently, we have
\begin{align}
 \inner{y^s - x}{\tilde{v}^s}& + u(y^s) - u(x) + 
 \tfrac{1}{2}[(\beta_s + \alpha/2) \normsq{y - y^{s}} + \beta_s \normsq{y^s - y^{s-1}} - \beta_s \normsq{y^{s-1} -x }]\notag\\
 & +\tfrac{1}{2}[\etat \normsq{y^s - \xtt} +  (\etat + \alpha/2)\normsq{y^s - x} - \etat\normsq{x - \xtt} ] \leq 0.\label{p2:Q_X}
\end{align}
In particular, the definition of $\tilde{v}^s$ in Line 2 of Algorithm \ref{alg:SPS} implies
$$ 
\begin{aligned}
\inner{y^s - x}{\tilde{v}^s} = &\inner{y^s - x}{\sumi p^{s} \vit} - \inner{y^s - x}{\sumi (p_i^{s} - p_i^{s-1}) \vit} \\
&+ \delta_s \inner{y^{s-1} - x}{\sumi (p_i^{s-1} - p_i^{s-2}) \vit} + \delta_s \inner{y^{s} - y^{s-1}}{\sumi (p_i^{s-1} - p_i^{s-2}) \vit}, \forall s \geq 2,\\
\inner{y^1 - x}{\tilde{v}^1} = &\inner{y^1 - x}{\sumi p^{1} \vit} - \inner{y^1 - x}{\sumi (p_i^{1} - p_i^{0}) \vit} \\
&+ \delta_1 \inner{y^{0} - x}{\sumi (p_i^{0} - p_i^{-1}) \vitt} + \delta_1 \inner{y^{1} - y^{0}}{\sumi (p_i^{0} - p_i^{-1}) \vitt}.
\end{aligned}
$$
So, substituting them into \eqref{p2:Q_X}, summing up the resulting inequality with weight $q_s$, noting the step-sizes conditions in \eqref{req:SPS_phase}, and utilizing Young's inequality,  we get 
\begin{align}\label{p2:y_conv}
\begin{split}
 \sums &q_s \left(\La(y^s; p^s, \pi^t) - \La(x; p^s, \pi^t) + \tfrac{1}{2}[\etat \normsq{y^s - \xtt} +  (\etat + \alpha/2)\normsq{y^s - x} - \etat\normsq{x - \xtt}]\right)\\
 &+ q_1 \delta_1 \inner{y^{0} - x}{\sumi (p_i^{0} - p_i^{-1}) \vitt}-q_S  \inner{y^S - x}{\sumi (p_i^{S} - p_i^{S-1}) \vit}\\
 \leq& \tsum_{s=2}^{S} \tfrac{q_{s-1}\gams[s-1]}{2} \normsq{\pt[s-1] - \pt[s-2]}[U] + \tfrac{q_{1} \delta_1^2 \MttU^2 }{2\beta_1} \normsq{p^0 - p^{-1}}[U] \\
 &-\tfrac{1}{2}[q_S (\beta_S + \alpha/2) \normsq{y^S - x}  - q_1 \beta_1 \norm{y^0 - x}^2].
 \end{split} \end{align} 

\vgap
Next, consider the convergence of $p^s$. The $p$-proximal update in Line 4 of Algorithm \ref{alg:SPS} implies
\begin{equation*}
\La(y^s; p, \pi^t) - \La(y^s; p^s, \pi^t) + \gamma_s [U(p; p^{s}) + U(p^s; p^{s-1}) - U(p; \pt[s-1])] \leq 0.
\end{equation*}
Observing  the strong convexity of $U$ with respect to $\norm{\cdot}[U]$ and the stepsize conditions in \eqref{req:SPS_phase}, the $q_s$ weighted sum satisfies 
\begin{equation*}
\sums q_s[ \La(y^s; p, \pi^t) - \La(y^s; p^s, \pi^t)] + \gams[S] \qs[S] U(p; \pt[S]) + \sums \tfrac{\qs \gams}{2}  \normsq{\pt[s] - \pt[s-1]}[U] \leq q_1 \gams[1]U(p; \pt[0]).
 \end{equation*} 
Then, combining it with the $y$ convergence bound in \eqref{p2:y_conv}, we get 
\begin{align*}
 \begin{split}
 \sums& q_s \left(\La(y^s; p, \pi^t) - \La(x; p^s, \pi^t)  + \tfrac{ 1}{2} [{\eta_t} \normsq{y^s - \xtt} + (\eta_t + \alpha/2) \normsq{y^s - x} - \etat \normsq{\xtt -x}]\right)\\
     &+q_1 \delta_1 \inner{y^0 - x}{\sumi \vitt (p_i^0 - p_i^{-1})} - q_S\inner{y^{S} - x}{\sumi \vit (p_i^{S} - p_i^{S - 1})}\\ 
   \leq& q_1\gam_1 U(p; p^0) - q_S \gam_S [U(p;p^S) + \tfrac{1}{2}\normsq{p^S- p^{S-1}}[U]]  +\tfrac{q_1 \delta_1^2 \MttU^2}{2\beta_1} \normsq{p^{0} - p^{-1}}[U] \\
   &-\tfrac{1}{2}[q_S (\beta_S + \alpha/2) \normsq{y^S - x}  - q_1 \beta_1 \norm{y^0 - x}^2].
   \end{split}
\end{align*}
Moreover, since $\La(y^s; p, \pi^t)$, $\normsq{y^s - \xtt}$ and $\normsq{y^s - x}$ are convex with respect to $y^s$ and 
$\La(x; p^s, \pi^t)$ is linear with respect to $p^s$, the desired convergence bound \eqref{eq:SPS_phase_bound} can be derived using the Jensen's inequality.
\end{proof}

In the above proposition, an $\wt$-weighted sum of the terms related to the outer DRAO loop, $\xt$ and $\bar p^t$, in \eqref{eq:SPS_phase_bound} is the same\footnote{Except that $\alpha/2$, instead of $\alpha$, is regarded as the strong convexity modulus for the outer DRAO loop.} as the $Q_x$ and $Q_p$ convergence bound in the proof of Proposition \ref{pr:central_Q_conv}, (c.f. \eqref{eq:central_Qx_Qp}). So a convergence bound of $Q$ in the DRAO-S method can be deduced by plugging it into the proof of Proposition \ref{pr:central_Q_conv}, i.e, the analysis for the $Q$ convergence in the DRAO method.
\vgap
\begin{proposition}\label{pr:central-sliding-Q}
Let $\zt := \{\xt, \bar p^t, \pi^t\}$ be generated by the DRAO-S method with the outer DRAO loop stepsize satisfying  \eqref{eq:strong_convex_dual}, \eqref{eq:aggregate_stx_cvxit_cst} and \eqref{req:DRAO}, and the inner SPS loop stepsize satisfying \eqref{req:SPS_phase}. Let $\tilwt:= \wt / \tsum_{s=1}^{\St} q_s^t$ denote the effective summation weight for the inner SPS loops and let $\MtU$ be defined in \eqref{def:Mt-def}. Suppose the following inter-phase stepsize requirements for inner SPS loop hold for $t\geq 2$:
\begin{equation}\label{req:Q-sliding-req}
\begin{split}
&\tilwt \qts[1] (\delts[1])^2 \MtU[t-1]^2 \leq \tilwt[t-1] \betats[1] \qts[\nSt[t-1]][t-1] \gamts[\nSt[t-1]][t-1],\ 
\MtU[N]^2 \leq  \gamts[\nSt[N]][N] (\betats[\nSt[N]][N] + \alpha/2), \\
&\tilwt \qts[1] \betats[1]   \leq \tilwt[t-1] \qts[\nSt[t-1]][t-1] (\betats[\nSt[t-1]][t-1] + \alpha/2),\   
\tilwt \qts[1] \gamts[1] \leq \tilwt[t-1] \qts[\nSt[t-1]][t-1] \gamts[\nSt[t-1]][t-1], \\
& \tilwt \delts[1] \qts[1] = \tilwt[t-1] \qts[\nSt[t-1]][t-1].
\end{split}
\end{equation}
Then the following $Q$-convergence bound holds for any reference point $z:=(\xstar, p, \pi)$ and for all $N\geq 1$
\begin{equation}\label{eq:central-sliding-Q}
\begin{split}
\sumt \wt Q\left(\zt,z\right) & + \wt[N](\etat[N] + \alpha/2) \normsq{\xt[N] - \xstar}/2 \\
&\leq \wt[1]\taut[1] \sumi p_i \Dfistar(\pii; \piit[0]) + (\wt[1]\etat[1] +\tilwt[1]\qts[1][1]\betats[1][1]) \normsq{\xt[0] - \xstar}/2 + \tilwt[1]\qts[1][1]\gamts[1][1] \DP^2/2.
\end{split}
\end{equation}
The above bound also holds if the last condition in \eqref{req:Q-sliding-req} is replaced by 
\begin{equation}\label{req:Q-sliding-alt}
\gamts[\St] (\betats[\St] + \alpha/2) \geq \MtU^2,\ \betats[1] =0, \text{ and } \delts[1] = 0\ \forall t \geq 1 .\end{equation}
\end{proposition}

\begin{proof}
As pointed out above the proposition, dividing both sides of \eqref{eq:SPS_phase_bound} by $\tsum_{s=1}^{\St}\qts$, taking its $\wt$-weighted sum, and noting the choice of initialization points in \eqref{subroutine:SPS_sm} and the telescope cancellation resulting from the stepsize requirements \eqref{req:Q-sliding-req}, we get a  convergence bound of $Q_x$ and $Q_p$ given by
{\skipdisplay
\begin{equation}\label{eq:sliding-QxQp}
\begin{split}
\sumt &{\wt} [Q_x(\zt; z)  + Q_p(\zt; z)] + \sumt \tfrac{\wt}{2} \etat\normsq{\xt - \xtt} + \wt[N](\etat[N] + \alpha/2) \normsq{\xt[N] - \xstar}  \\
&\leq \tilwt[1]\qts[1][1](\gamts[1][1]U(p; \ptlt[0]) + \betats[1][1] \normsq{y^0 - \xstar}/2)+ \tfrac{\wt[1] \etat[1]}{2} \normsq{\xt[0] - \xstar}.
\end{split}
\end{equation}
}
The preceding bound is almost the same as its counterpart in Proposition \ref{pr:central_Q_conv}, i.e., \eqref{eq:central_Qx_Qp}. Moreover, since the generation of $\pi^t$ in the outer DRAO loop of the DRAO-S method is also the same as that of the DRAO method, the $Q_\pi$ convergence bound in Proposition \ref{pr:central_Q_conv} (c.f. \eqref{eq:central_Qpii_bound}) is also valid. The desired $Q$ convergence bound in  \eqref{eq:central-sliding-Q} then follows from  combining \eqref{eq:central_Qpii_bound} with \eqref{eq:sliding-QxQp}, and noting $U(p; \ptlt[0]) \leq \DP^2/2$ and $y^0 = x^0$. 

In addition, if the alternative stepsize requirement \eqref{req:Q-sliding-alt} is satisfied, \eqref{eq:SPS_phase_bound} can be simplified further to 
\begin{align*}
 \begin{split}
 (\tsum_{s=1}^{\St} q_s)& (\La(x^t; p, \pi^t) - \La(x; \bar{p}^t, \pi^t))
     +\tfrac{(\sums q_s)}{2} ({\eta_t} \normsq{\xt - \xtt} + (\eta_t + \alpha/2) \normsq{\xt - x} - \etat \normsq{\xtt -x})\\ 
   \leq& q_1\gam_1 U(p; p^0) - q_\St \gam_\St U(p;p^\St). \\
   \end{split}
\end{align*}
Then a similar argument would lead to the $Q$ convergence bound in \eqref{eq:central-sliding-Q} as well.
\end{proof}
\vgap

The next convergence proofs of the DRAO-S method for the smooth problem, i.e., Theorem \ref{thm:sps-sm} and Theorem \ref{thm:sps-sm-str}, are direct applications of Proposition \ref{pr:central-sliding-Q}.\\

\noindent{\bf Proof of Theorem~\ref{thm:sps-sm}}
It is easy to verify that the stepsize choice in \eqref{stp:sps_sm} satisfies the requirements in Proposition \ref{pr:central-sliding-Q}, thus the following convergence bound is valid for any reference point $z:=(\xstar; p, \pi)$ and for all $N \geq 1$,
\begin{equation}\label{thm-pf:sm-Q}
\sumt \wt Q\left(\zt,z\right) + \Lf \normsq{\xt[N] - \xstar} \leq  \Lf\Rx^2 + \DP\Rx / \Delta.\end{equation}
Let $\hpiiN = \grad \myfi(\xbart[N])$ and $\hpiN[] \in \argmax_{p \in P} \sumi p_i \myfi(\xbart[N])$  such that  $f(\xbarN)= \La(\xbarN; \hpN, {\hat \pi}^N)$ (see Lemma \ref{lm:duality}), then the desired convergence result in \eqref{eq:f_SPS_sm} can be deduced by choosing the reference point to be $(\xstar; \hp^N, \hat{\pi}^N)$.  Furthermore, the result in \eqref{eq:f_SPS_sm1} can be deduced by substituting in the specific choice of $\Delta$. 

Next, we show the boundedness of $\MtU$. Since $\pi^t = \grad \pmbf(\xundert)$ (c.f. \eqref{eq:smooth_alg2}), $\pmbf$ is smooth and $\xundert$ is a convex combination of $\xt[0]$ and $\{\xtilt\}$, the boundedness of $\norm{\pi^t}[2, U^*]$ follows from  the boundedness of $\xtilt$. 
 Setting the reference point to the saddle point $(\xstar, p^*, \pi^*)$ (c.f. Lemma \ref{lm:sad_ex}), we get from \eqref{thm-pf:sm-Q}
\begin{equation*}
 \Lf \normsq{\xt[N] - \xstar}  \leq 2 \Lf\Rx^2\ \forall N \geq 2.\end{equation*}
This shows that  $\xt$'s are restricted to be a bounded ball around $\xstar$. Since $\thetat \leq 1,$ the extrapolated sequence $\xtilt$'s are also restricted to a bounded ball around $\xstar$, implying the boundedness of $\{\MtU\}_{t=1}^\infty$.  
\endproof

\vgap
\noindent{\bf Proof of Theorem~\ref{thm:sps-sm-str}}
We can verify that the stepsize choice in \eqref{stp:sps_sm_str} satisfies the alternative requirements in Proposition \ref{pr:central-sliding-Q} (c.f. \eqref{req:Q-sliding-alt}). Setting the reference point $z$ to the saddle point $(\xstar, p^*, \pi^*)$ (c.f. Lemma \ref{lm:sad_ex}), we get from \eqref{eq:central-sliding-Q} the desired geometric convergence of $\xt[N]$ \eqref{eq:sps_x_sm_str_conv}, i.e., 
\begin{equation}\label{thm8:x_last_iterat_N}
\normsq{\xt[N] - \xstar} \leq \theta^N [(1 + 2\kappa)\normsq{\xt[0] - \xstar} + \tfrac{4 \DP^2 }{\eta \Delta\alpha }], \forall N \geq 1.\end{equation}
Observe that the above convergence bound also implies the boundedness of $\{\xt\}$. Thus the existence of an uniform bound for $\MtU$ follows from an argument similar to that of the proof of Theorem \ref{thm:sps-sm}. 

Next, we establish an upper bound on the total of inner iterations when $\Delta:= 2 \DP^2 / \eta \Lf \normsq{\xt[0]-\xstar}$ and  $\kappa$ is large. The specific choice of $\Delta$ allows us to simplify \eqref{thm8:x_last_iterat_N} further to 
$$\normsq{\xt[N] - \xstar} \leq \theta^N \left(5\kappa\Rx^2\right).$$
Let $\Nep$ denotes the least number of phases required to satisfy $\normsq{\xt[\Nep] -\xstar} \leq \ep$. Clearly, $\Nep=\bigO(\sqrt{\kappa} \log(1/\ep))$. Specifically, since  $\kappa\geq 1$ implies $1/\theta \leq 2$, we have $(1/\theta)^{\Nep} \leq 10\kappa\Rx/\ep.$

For the total inner iteration number, a bound for $\St$ are provided by the stepsizes requirement \eqref{stp:sps_sm_str}.  Since $\MtU \leq \Mtil$ and
$ \St \leq 1+ \sqrt{2 \wt \Delta} \Mtil,$
 the total number is upper bounded by 
\begin{align*}
\tsum_{t=1}^{\Nep} \St &\leq \Nep + \tsum_{t=1}^{\Nep} (\sqrt{\tfrac{1}{\theta}})^{t-1}\sqrt{\Delta} \Mtil = \Nep + \tfrac{(1/\theta)^{\Nep/2}-1}{\sqrt{1/\theta} - 1}\sqrt{\Delta}\Mtil\\
& \leq \Nep + \tfrac{1}{\sqrt{1/\theta} - 1} \tfrac{1}{\sqrt{\sqrt{8\kappa + 1} - 1}}\tfrac{\Mtil\DP}{\sqrt{\Lf} \Rt[0]} \tfrac{16\sqrt{\Lf} \Rx}{\alpha \sqrt{\ep}}
 \leq \Nep + \tfrac{64 \kappa^{1/4} \Mtil\DP}{\alpha \sqrt{\ep}}.
\end{align*}
The second last inequality follows from the algebraic fact $\sqrt{1 + l} - 1 \geq l/4 $ for $l\leq 1$, and 
\begin{align*}
\tfrac{1}{\sqrt{1/\theta} - 1}&\tfrac{1}{\sqrt{\sqrt{8\kappa + 1} - 1}} = \tfrac{1}{\sqrt{1 + 2/(\sqrt{8\kappa+1}-1)}-1}\tfrac{1}{\sqrt{\sqrt{8\kappa + 1} -1}} \leq 2(\sqrt{8\kappa +1} -1)\tfrac{1}{\sqrt{\sqrt{8\kappa + 1} - 1}}\\
&\leq 2 \sqrt{\sqrt{8\kappa + 1} - 1} \leq 4 \kappa^{1/4}.
\end{align*}
Thus the number of inner iterations, and hence the $P$ projection complexity, are upper bounded by $\bigO(\tfrac{\kappa^{1/4}\Mtil \DP}{\alpha\sqrt{\ep}} + \sqrt{\kappa} \log(\tfrac{1}{\ep}))$.
\endproof

\vgap

\noindent{\bf Proof of Theorem~\ref{thm:sps-ns}}
The proof is similar to that of Theorem \ref{thm:cen-non-smooth}. 
Let us first consider the non-strongly convex case.  Since $\bMtU \geq \MtU, \forall t$, the stepsize choice in \eqref{stp:drao-s-ns} satisfies all the requirements in Proposition \ref{pr:one-step-sliding} (c.f. \eqref{req:Q-sliding-req}). So substituting the stepsize choice into \eqref{eq:central-sliding-Q}, we obtain the following convergence bound of the $Q$ gap function for any reference point $z:=(\xstar, p, \pi)$:
\begin{equation}
  \sumt Q(\zt; z)   + \tfrac{\eta}{2} \normsq{\xt[N] - x} + \tfrac{\gam}{\Delta} U(p;\pt[N]) \leq (\tfrac{\eta}{2} + \tfrac{\beta}{2\Delta})\normsq{\xt[0] - x}+ \tfrac{\gam}{\Delta} U(p;\pt[0]) + \tau \sumi p_i V_i(\pii; \piit[0]).
  \end{equation} 
  The desired function value convergence bound \eqref{eq:sps-f-ns} follows immediately by selecting the reference point to be $(\xstar, \hat p^N, \hat \pi^N)$, where $\hpiiN = \grad \myfi(\xbart[N])$ and $\hpiN[] \in \argmax_{p \in P} \sumi p_i \myfi(\xbart[N])$.

The convergence bound \eqref{eq:sps-f-ns-str} for the strongly case also follows from substituting the stepsize choice in \eqref{stp:drao-sps-ns-str} into \eqref{eq:central-sliding-Q}.
\vgap
\endproof

\section{Lower Communication Complexities}

In this section, we establish theoretical lower bounds for distributed risk-averse optimization to show the communication complexities of both  DRAO and DRAO-S are not improvable.
Towards that end, we propose a distributed prox mapping (DPM) computing environment consisting of the following requirements and propose uniform lower bounds for all algorithms satisfying the requirement.
\begin{itemize}  
\item \textbf{Local memory}: the server node has a finite local memory $\calM_s$ and each worker node has a finite primal memory and a finite dual memory, $\calM_i$ and $\calM^\pi_{i}$ , respectively. In the beginning, the local memories contain only the trivial vector $0$, i.e., 
 $$\Mit[0] =\Mst[0]:= \{0\},\ \Mpiit[0] := \{0\}\ \forall i \in [m].$$
In one communication round, these local memories can be updated by both local computation and server-worker communication: 
$$\Mst[t+1] :=  \Mstcp\cup\Mstcm ,\ \Mit[t+1] := \Mitcp \cup \Mitcm,\ \Mpiit[t+1] := \Mpiitcp,\ \forall i \in [m],$$ 
\textblue{where $\Mstcp$ and $\Mpiitcp$  represent results from the local computation, and $\Mstcm$ and $\Mitcm$ denote the vector(s) communicated to the server and the $i$\ts{th} worker node, respectively.}
\item \textbf{Server-worker communication:} in one communication round, each worker can send  one vector from its local primal memory to the server: 
$$ \Mstcm :=\{y_i \in \Span(\Mit[t-1]), i \in [m]\},$$
and the server can share one vector from its memory with the worker: 
$$\Mitcm \in \Span(\Mst[t-1]).$$
\item \textbf{Local computations}: between communication rounds, each worker can query its dual prox mapping oracle and the $\Ai$-multiplication oracle\footnote{The oracle returns matrix vector multiplication result of the form $\Ai x$ and $\Aitr \pi$.} for $L\geq 0$ times.
{
\small
\begin{align}\label{lower:dual_prox}
&\Mitcp[t][i]:= \Mitcp[t][i][,L],\ \Mpiitcp := \Mpiitl[t][L] \text{\ where }\Mitcp[t][i][,0]:= \Mit[t-1],  \Mpiitl[t][0]:= \Mpiit[t-1]. \nonumber\\
&\textit{For } l = 1, 2, 3, \ldots, L: \nonumber \\
&\quad \ \Mitcp[t][i][,l] = \Mitcp[t][i][,l-1] \cup \{\Aitr \piit[t,l],\ \Aitr\piibar\},\ \Mpiitl[t][l] := \Mpiitl[t][l-1] \cup \{\piit[t,l], \Ai \bar x\} , \text{where } \ \bar{x} \in \Span(\Mitcp[t][i][,l-1]), \nonumber \\
&\quad  \ \ \  \piibar \in \Span(\Mpiitl[t][l-1]),\ \piit[t,l] \in \argmax_{\pii \in \Pii} \inner{\Ai \bar{x}}{\pii} - \fistar(\pii) - \tfrac{\tau}{2}\normsq{\pii - \bar{\pi}_i},\text{ for some } \tau\geq 0. 
\end{align}
}{
The server node can query its $u(x)$ prox mapping oracle for $L\geq 0$ times. 
\setlength{\abovedisplayskip}{2pt}
\setlength{\belowdisplayskip}{2pt}
\setlength{\abovedisplayshortskip}{2pt}
\setlength{\belowdisplayshortskip}{2pt}
\small
\begin{align*}
&  \Mstcp[t][s]:=\Mstcp[t][s][,L] \text{ where } \Mstcp[t][s][,0]:= \Mst[t-1]. \\
&\textit{For } l = 1, 2, 3, \ldots, L: \\
&\quad \Mstcp[t][s][,l] :=  \Mstcp[t][s][,l-1] \cup \{x_s^l\}, \text{ where } x_s^l := \argmin_{x \in X} u(x) + \tfrac{\eta}{2} \normsq{x - \bar x}, \bar{x} \in \Span(\Mstcp[t][s][,l-1])\ 
.
\end{align*}
}
\item \textbf{Output solution:} the output solution $\xt$ comes from local primal memories, 
{\skipdisplay$$\xt \in \Span((\cup_{i\in[m]} \Mit) \cup \Mst), t \geq 1.$$}
\end{itemize}

The only hard requirement for the DPM environment is that only one vector can be sent and received by each worker during one communication round.
Indeed, the computations supported by the DPM environment are quite strong in several aspects. First, it allows gradient evaluation of $\myfi$ since it is equivalent to the $\pii$-prox mapping (c.f. \eqref{lower:dual_prox}) with $\tau=0$, i.e.,
{ \skipdisplay
$$\piibar = \grad \myfi(\bar x) \Leftrightarrow \piibar \in \targmax_{\pii \in \Pii} \inner{\pii}{\bar x} - \fistar(\pii).$$}
Second, it allows a possibly large number of local computation steps to be performed between communications. This assumption of generous computing resource at each node helps us to focus on the communication bottleneck. 
Third, 
it allows the freedom to make an arbitrary selection from the span of the local memory for communication, computation, and outputting solutions. 
For example, it might appear that the DRAO method violates the requirement because of the  $(x, p)$-prox mapping in Line 5 of Algorithm \ref{alg:cPD}. However, if we let $(\xt,\hat{p}^t)$ be an optimal pair of saddle point solutions in the $(x,p)$-prox mapping step (c.f. \eqref{eq:saddle_subproblem}), the output $\xt$ can be written alternatively as
{\skipdisplay $$\xt \leftarrow \targmax_{x \in X} \etat\normsq{x - \xunder}/2 +u(x),$$}
where $\xunder := \xtt - \sumi \hat p^t_i \vit /\etat$ and $\xunder\in \Span(\Mst)$. So the $(x, p)$-prox mapping actually satisfies the above local computation requirement.
Moreover, the computation and communication of $f_i^*(\piit)$'s are unnecessary  because they are only used for generating $\hat p^t$. Since all other steps are directly supported, the DRAO method can be implemented on the DPM environment. Indeed, our setup implies 
that the desired $p$ can be obtained from any oracle when selecting $x$ (from the span of local memory of the server). This renders all communication and computation related to $p$ unnecessary. So the DRAO-S method, and, more generally, any distributed algorithm consisting of the $x$-prox mapping, the $\pi$-prox mapping, and some $p$ update can be implemented on the DPM environment. For simplicity, we will call an algorithm satisfying the DPM requirement a DPM algorithm for the rest of this section.


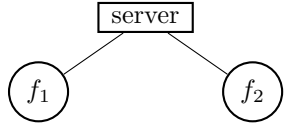
\begin{figure}[!htb]
\centering
          \begin{tikzpicture}[auto,node distance=8mm,>=latex]
            \tikzstyle{round}=[thick,draw=black,circle]
            \tikzstyle{rect}=[thick,draw=black,rectangle]
            \node[rect] (rt) at (0,0) {$\, \text{server} \, $};
            \node[round, below left= 0.7 of rt] (m1) {$f_{1}$};
            \node[round, below right=0.7 of rt] (m2) {$f_2$};
            \draw[-] (rt) -- (m1);
            \draw[-] (rt) -- (m2);
        \end{tikzpicture}
        \vspace{-2mm}
        \caption{Network topology of hard instances.}\label{fig:hard}
    \vspace{-7mm}
\end{figure}

Now we present some hard instances, inspired by \cite{nemirovsky1983problem,nesterov2003introductory,scaman19}, for all DPM algorithms. We first describe a network topology and a general result which will be used in all our constructions. As shown in Figure~\ref{fig:hard}, the problem has only two workers, node $1$ and node $2$. Let $\Ki$ denote the subspace with non-zero entries only in the first $i$ coordinates, $\Ki:=\{ x \in \R^n: x_j = 0\ \forall j > i\}$. We will construct $\myfi[1]$ and $\myfi[2]$ such that the iterate $\xt$ generated in $t$ communication rounds will be restricted to a certain $\Ki$. Towards that end, we call a hard problem 
 \textit{odd-even preserving} if the memories generated by any DPM algorithm satisfies
\begin{align}\label{ls:odd-even}
\begin{split}
& \Mit[0][1] \cup \Mit[0][2] \cup \Mst[0] \subset \Ki[2], \\
&\Mit[t-1][1] \subset  \Ki \Rightarrow \small\begin{cases} \Mitcp[t][1] \subset \Ki &\text{ $i\geq2$ even}\\ \Mitcp[t][1] \subset \Ki[i+1] &\text{ $i\geq2$ odd} \end{cases},\ 
\Mit[t-1][2] \subset \Ki \Rightarrow \small\begin{cases} \Mitcp[t][2] \subset \Ki[i+1] &\text{ $i\geq2$ even}\\ \Mitcp[t][2] \subset \Ki &\text{ $i\geq2$ odd} \end{cases},\\
&\Mit[t-1][s] \subset \Ki \Rightarrow \Mstcp \subset \Ki.
\end{split}
\end{align}
\textblue{This property stipulates that the progresses on the reachable subspace $\Ki$ are possible only on node $1$ or $2$ depending on if $i$ is odd or even, so that a large number of communication rounds between node $1$ and $2$ are necessary for a non-trivial solution.}
The next lemma formalizes such limited progress by a DPM algorithm.
\begin{lemma}\label{lm:even-odd}
If the odd-even preserving property \eqref{ls:odd-even} holds, the output solution $\xt$ generated by a DPM algorithm after $t$ communication rounds satisfy $\xt \subset \Ki[\ceil{t/2}+2].$
\end{lemma}
\begin{proof}
Let $\Mt:= \Span(\Mit[t][1]\cup\Mit[t][2]\cup\Mst)$, and let $t(i):= \min \{t \geq 0: \exists y \in \Memt, j \geq i\ s.t\ y_j \neq 0  \}$ denote the first time a vector with a non-zero $j \geq i$th index is generated. We develop a lower bound for $t(i).$

Consider an even $i> 2$.  By the definition of $t(i)$, $\Memt[t(i)-1] \subset \Ki[i-1].$  The odd-even preserving property then implies $\Mitcp[t(i)][2] \subset \Ki[i-1]$ and $\Mitcp[t(i)][s] \subset \Ki[i-1]$, so 
$\Mit[t(i)][1] \subset \Ki,\ \Mit[t(i)][s] \subset \Ki[i-1],\text{ and } \Mit[t(i)][2] \subset \Ki[i-1]$
after one communication round. Next, the odd-even preserving property again implies $\Mitcp[t(i)+1][1] \subset \Ki[i]$,  $\Mitcp[t(i)+1][2] \subset \Ki[i-1]$ and $\Mitcp[t(i)+1][s] \subset \Ki[i-1]$, so
$\Mit[t(i)+1][1] \subset \Ki,\ \Mit[t(i)+1][s] \subset \Ki[i],\ \text{and } \Mit[t(i)][2] \subset \Ki[i-1]$ after another communication round.
Therefore, we have $t(i+1) > t(i)+1$, i.e., $t(i+1) \geq t(i) + 2.$ The same recursive bound can also be obtained for an odd $i\geq 2$. In view of $t(2)\geq 0$, the largest non-zero index $i$ in $\Mt$ satisfies 
$t \geq t(i) \geq t(2) + 2i - 4 \geq 2i-4 $, thus $i \leq \ceil{t/2}+2. $
\end{proof}
\vgap

We are now ready to provide lower bounds under different problem settings. The next two results establish tight lower communication bounds for the smooth problem with a non-strongly convex $u(x)$ and a strongly convex $u(x)$, respectively.    
\begin{theorem}\label{thm:sm_lower}
Let $\Lf>0$, $\Rt[0] \geq 1$ and $\ep>0 $ be given. For a sufficiently large problem dimension, i.e., $n > 2 \ceil{\sqrt{\Lf}\Rt[0]/8\sqrt{\ep}}$, there exists a smooth hard problem of form \eqref{eq:orig_prob} with an aggregate smoothness constant $\Lf$ (c.f. \eqref{eq:smo_cst}), $\norm{\xt[0] - \xstar} \leq \Rt[0]$ such that any DPM algorithm takes at least $\Omega(\sqrt{\Lf}\Rt[0]/\sqrt{\ep})$ communication rounds to find an $\epsilon$-optimal solution.
\end{theorem}

\begin{proof}
 Consider the following hard problem parameterized by $\beta\geq 0, \gamma \geq 0$ and $k\geq 4$, 
\begin{align}\label{sm_l:prob}
\begin{split}
&f(x) := \tmax_{p \in \Delta_2^+}p_1 f_1(x) + p_2 f_2(x) + u(x) \text{ with } X = \R^{2k+1},\ u(x) = 0,\\ 
&f_1(x):= \tfrac{\beta}{2} [2 \tsum_{i=1}^{k} (x_{2i-1} - x_{2i})^2 + x_1^2 + x_{2k +1}^2 - 2 \gamma x_1], \\
&f_2(x):= \tfrac{\beta}{2} [2 \tsum_{i=1}^{k} (x_{2i} - x_{2i+1})^2 + x_1^2 + x_{2k +1}^2 - 2 \gamma x_1]. \\
\end{split}
\end{align}
Its aggregate smoothness constant  $\Lfbar$ (c.f. \eqref{eq:smo_cst}) satisfies $\Lfbar \leq 6 \beta$, and its optimal solution $(\xstar, \pstar)$ satisfies
$$p^* = [\tfrac{1}{2}, \tfrac{1}{2}], \ \xstar_i = \gamma (1 - \tfrac{i}{2k+2})\ \forall i \leq 2k+1, \text{ s.t.  } \norm{\xt[0] -\xstar} \leq \gamma \sqrt{k+1} \ \text{ and } \fstar = -\tfrac{\beta\gamma^2}{2}  [1 - \tfrac{1}{2k+2}]. $$
Their optimality can be verified with the first order conditions: 
{\skipdisplay $$0 = \grad (\tfrac{1}{2}f_1 + \tfrac{1}{2}f_2)(\xstar) \text{ and } [1/2, 1/2] \in \targmax_{p \in \Delta_2^+} p_1 f_1(\xstar) + p_2 f_2(\xstar).$$}
The even-odd preserving property holds for \eqref{sm_l:prob}. To see this, consider the worker node $\myfi[1]$. Let an even $i \geq 2$ be given and assume $\Mit[t-1][1] \subset \Ki$, i.e., $\Mitcp[t][1][,0] \subset \Ki$. 
Because $\Ai = I$, the update rule in \eqref{lower:dual_prox} imply that $\Mpiitl[t][0][1] \subset \Ki$. We show $\Mitcp[t][1][,l] \cup \Mpiitl[t][l][1] \subset \Ki$ for all $l \geq 0 $ by induction. Clearly, the statement holds for $l=0$. If $\Mitcp[t][1][,l-1] \cup \Mpiitl[t][l-1][1] \subset \Ki$,  $\bar x$ and $\piibar$ chosen in \eqref{lower:dual_prox} must be in $\Ki$. As for the $\pii$-prox mapping, if $\tau =0$, $\piit[t,l][1] := \grad \myfi[1](\bar x) \subset \Ki.$ If $\tau>0$, Lemma \ref{lm:app-pd} in the appendix allows us to write $\piit[t,l][1]$ as
{\skipdisplay
$$\piit[t,l][1] = \piibar[1] + \tfrac{1}{\tau} (\bar x -\uy), \text{ where } \uy \leftarrow \targmin_{y} \myfi[1](y) + \tfrac{1}{2\tau} \normsq{\bar x + \tau \piibar[1] - y}.$$}
In particular, $\uy \in \Ki$ because
{
\skipdisplay
\begin{align*}
\uy =\targmin_{x\in \R^{2k+1}} &\ \tfrac{\beta}{2} [2 \tsum_{j=1}^{i/2} (x_{2j-1} - x_{2j})^2 + x_1^2  - 2 \gamma x_1] + \tfrac{1}{2\tau} \tsum_{j=1}^{i}\normsq{\bar x_j + \tau \bar{\pi}_{1,j} - x_j} \\
&+  \tfrac{\beta}{2} x_{2k +1}^2 + \tfrac{1}{2\tau} \tsum_{j=i+1}^{2k+1}\normsq{ x_j}.\end{align*}}
So $\piit[t,l][1] \in \Ki$ also holds for $\tau > 0$. Thus the principle of induction implies that $\Mitcp[t][1][,L] \cup \Mpiitl[t][L][1] \subset \Ki, \forall L \geq 0$, i.e., $\Mitcp[t][1] \subset \Ki$. 
In addition, when $i \geq 2$ is odd and  $\Mit[t-1][1] \subset \Ki$, we have $\Mit[t-1][1] \subset \Ki[i] \subset \Ki[i+1]$. Since $i+1$ is even, the preceding result implies that $\Mitcp[t][1] \subset \Ki[i+1]$.
A similar result can also be derived for the worker $\myfi[2]$ for both even and odd $i\geq 2$. Therefore, problem \eqref{sm_l:prob} satisfies the even-odd preserving property.

Applying Lemma \ref{lm:even-odd}, the output solution $\xt[k]$ from any DPM algorithm in $ k$ communication rounds must satisfy $\xt[k] \in \Ki[k].$ In particular, let $\bar{f} := (f_1 + f_2)/2$ denote a lower bound for $f$. Then 
$f(\xt[k]) \geq \min_{x \in \Ki[k]} f(x) \geq \min_{x \in \Ki[k]} \bar f(x) = -\tfrac{\beta \gamma^2}{2} [1 - \tfrac{1}{k+1}].$

Now we set the parameters in \eqref{sm_l:prob} to obtain the desired lower bound. If $\ep \geq \Lf \Rx^2/4096$, $\Omega(\sqrt{\Lf}\Rt[0]/\sqrt{\ep}) = \Omega(1)$, so the lower bound clearly hold. Otherwise, we set $\beta:= \Lf/6$, $\gamma := \Rt[0] /\sqrt{k+1}$ and $k:=\ceil{\sqrt{\Lf}\Rt[0]/8\sqrt{\ep}} $ such that \eqref{sm_l:prob} is $\Lf$-smooth (c.f \eqref{eq:smo_cst}) with $\norm{\xt[0] - \xstar} \leq \Rt[0]$ and $k \geq 4$. A solution $\xt[k]$ generated by any DPM algorithm in $k$ communication rounds satisfy 
$f(\xt[k]) - f_* \geq  \tfrac{\gamma^2 \beta}{4k+4} \geq \ep.$
Thus they imply the desired  $\Omega(\sqrt{\Lf}\Rt[0]/\sqrt{\ep})$ lower communication complexity bound when the problem dimension is $2\ceil{\sqrt{\Lf}\Rt[0]/8\sqrt{\ep}} +1$.
\end{proof}
\vgap
We remark here that the above risk-averse lower  bound is the same as the risk-neutral lower bound of $\Omega(\sqrt{\Lfp[\bar p]}\Rx/\sqrt{\ep})$, developed in \cite{scaman19}, if $P$ is a singleton set of the empirical distribution, $P=\{\bar p:= (1/m,...,1/m)\}$, and $\Lfp[\bar p]$ denotes the aggregate smoothness constant (c.f. \eqref{eq:smo_cst}) associated with $\bar p$. But, other than the intuition that the risk-averse problem should be harder than the risk-neutral problem, the latter bound offers limited insights. Our risk-averse lower bound  can be larger than the risk-neutral lower bound  because 
the aggregate smoothness constant $\Lf$ (c.f. \eqref{eq:smo_cst}) defined over a non-trivial $P$ can be significantly larger than  $\Lfp[\bar p]$.  For example, consider an expanded version of \eqref{sm_l:prob} constructed by adding $(m-2)$ additional workers with constant local cost functions, 
$\myfi(x) \equiv C\  \text{for some  $C < \fstar$,}$
 and by setting $P$ to the m-dimensional simplex $\Delta_m^+$. 
The same argument as above will lead to the same lower bound  of $\Omega(\sqrt{\Lf}\Rx/\sqrt{\ep})$ for the expanded problem. However, because the smoothness constants of $\{\myfi\}_{i=3}^m$ are zero, we have $\Lfp[\bar p] \leq 2 \Lf /m << \Lf.$
\vgap

\begin{theorem}
Let $\Lf>8\alpha>0$ and $\ep>0$ be given. There exists an infinite-dimensional smooth problem of form \eqref{eq:orig_prob} with an aggregate smoothness constant $\Lf$ (c.f. \eqref{eq:smo_cst}) and a strong convexity modulus $\alpha$ such that any DPM algorithm requires at least $\Omega(\sqrt{\Lf/\alpha}\log(1/\ep))$\footnote{ We ignore the problem parameter $\Rx$ inside the $\log$.} communication rounds to find an $\ep$-close solution, i.e., $x$ such that  $\normsq{x - \xstar} \leq \ep$.
\end{theorem}

\begin{proof}
Again we prove the result by construction. Consider the following infinite dimensional problem parameterized by $\beta > 2\alpha$ 
 \begin{align}\label{smst_l:prob}
\begin{split}
&f(x) := \tmax_{p \in \Delta_2^+}p_1 f_1(x) + p_2 f_2(x) + u(x) \text{ where} \  X := \R^{\infty} 
\\
&u(x) := \tfrac{\alpha}{2} \normsq{x}, f_1(x):= \tfrac{\beta-\alpha}{4}[x ^ \top A_1 x - 2 x_1],\  f_2(x):= \tfrac{\beta-\alpha}{4}[x ^ \top A_2 x - 2 x_1],\\
&A_1 := \scriptsize\begin{bmatrix}
\tfrac{3 + 5\gamma}{4 + 4\gamma } & -1 \\
 -1 & 1 \\
& & 1 &-1 \\
& & -1 & 1 \\
& & & & 1 &-1 \\
& & & & -1 & 1 \\
&&&&&&\ddots
\end{bmatrix},\ 
A_2 := \scriptsize\begin{bmatrix}
\tfrac{5 + 3\gamma}{4 + 4\gamma } &  \\
 & 1 &-1 \\
 & -1 & 1 \\
 & & & 1 &-1 \\
 & & & -1 & 1 \\
&&&&&\ddots\\
&&&&&&\ddots
\end{bmatrix},\ \gamma =\tfrac{\alpha}{\beta}.
\end{split}
\end{align}
Clearly, the aggregate smoothness constant (c.f. \eqref{eq:smo_cst}) of the problem is bounded by $\beta - \alpha$ and its strong convexity modulus is $\alpha$.
The optimal solutions are given by $\pstar = (1/2, 1/2)$ and $\xstar$, with $\xstar_i = (\tfrac{1-\sqrt{\gamma}}{1 + \sqrt\gamma})^i\ \forall i \geq 1$, since they satisfy the first order optimality conditions:{\skipdisplay
$$\grad (\tfrac{1}{2} f_1 + \tfrac{1}{2} f_2 + u)(\xstar) = 0 \text{ and } (\tfrac{1}{2}, \tfrac{1}{2}) \in \targmax_{p \in \Delta^+_2} p_1 f_1(\xstar) + p_2 f_2(\xstar).$$}
Moreover, similar to Theorem \ref{thm:sm_lower}, the alternating block diagonal structure of $A_1$ and $A_2$ implies the even-odd preserving property, so $\xt[k]$ generated by any DPM algorithm in $ k \geq 4 $ communications rounds satisfy $\xt[k] \subset \Ki[k],$ i.e., 
\begin{equation}\label{lower-pf}\begin{split}
\normsq{\xt[k] - \xstar} &\geq \tsum_{i=k+1}^{\infty} (\xstar_i)^2 = (\tfrac{1-\sqrt{\gamma}}{1 + \sqrt\gamma})^{2k} (\tfrac{1-\sqrt{\gamma}}{1 + \sqrt\gamma})^2/(1 - (\tfrac{1-\sqrt{\gamma}}{1 + \sqrt\gamma})^2) = (\tfrac{1-\sqrt{\gamma}}{1 + \sqrt\gamma})^{2k} \Rt[0]^2\\
&= (1 - \tfrac{2\sqrt{\gamma}}{1+\sqrt{\gamma}})^{2k} \Rt[0]^2 \geq (1 - {2\sqrt{\gamma}})^{2k} \Rt[0]^2 = (1 - 2\sqrt{\gamma})^{2k} \Rt[0]^2 \geq (1 - {2\sqrt{\alpha/(\beta-\alpha)}})^{2k} \Rt[0]^2,
\end{split}
\end{equation}
where $\Rt[0]^2 := (\tfrac{1-\sqrt{\gamma}}{1 + \sqrt\gamma})^2/(1 - (\tfrac{1-\sqrt{\gamma}}{1 + \sqrt\gamma})^2)$. Thus it takes at least $\Omega(\sqrt{(\beta-\alpha)/\alpha} \log(1/\ep))$ communication rounds to obtain an $x$ with $\normsq{x - \xstar} \leq \ep$. 

Now we select the parameter $\beta$ to derive the lower complexity bound.
If $\ep \geq (1 - {2\sqrt{1/\kappa}})^{8} \Rt[0]^2$ such that $\Omega(\sqrt{\Lf/\alpha}\log(1/\ep)) = \Omega(1)$, the desired lower bound clearly holds.
Otherwise, we can 
set $\beta:=\Lf + \alpha$ such that  the hard problem \eqref{smst_l:prob} is $\Lf$-smooth,  and the desired lower communication bound of  $\Omega(\sqrt{\Lf/\alpha} \log(1/\ep))$ follows from \eqref{lower-pf}.
\end{proof}
\vgap
We remark here that a finite dimensional hard problem can also be obtained by modifying \eqref{smst_l:prob} according to \cite{LanZhou18RPDG}.  Next, we move on to consider the structured non-smooth problem. 
\begin{theorem}
Let $\MA>0$, $\DPi>0$ $\Rt[0] \geq 1$ and  $\ep >0$ be given. When the problem dimension $n$ is sufficiently large (specified below), there exists a structured non-smooth problem $f$ of form \eqref{eq:orig_prob} with  $\MA \geq \max_{i \in [m]} \norm{\Ai}[2, 2]$, $\DPi \geq \max_{i \in [m]}\max_{\pii, \piibar\in \Pii} \norm{\pii- \piibar}$ (c.f. \eqref{eq:MA_def}) and $\Rt[0] \geq\norm{\xt[0] - \xstar} $ such that the following communication lower bounds hold. 
{\skipdisplay
\begin{enumerate}
\item[a)] When $u(x)$ is convex and $n > 2\ceil{\DPi \MA \Rt[0]/96 \ep} $, any DPM algorithm requires at least $\Omega(\MA\DPi \Rt[0]/\ep)$ communication rounds to find an $\ep$-optimal solution.
\item[b)] When $u(x)$ is $\alpha>0$ strongly convex and $n > 2\ceil{\DPi \MA/48\sqrt{\alpha \ep}}$, any DPM algorithm requires  at least $\Omega(\MA\DPi /\sqrt{\ep\alpha})$ communication rounds to find an $\ep$-optimal solution.
\end{enumerate}
}
\end{theorem}
\vgap
\begin{proof}
We consider the following hard problem parameterized by $k\geq 4$, $\alpha$, $\gamA$ and $\gampi$:
\begin{align}\label{lb_ns}
\begin{split}
&f(x) := \tmax_{p \in \Delta_2^+}p_1 f_1(x) + p_2 f_2(x) \text{ with } \ X = \R^{2k+1},\  u(x) = \tfrac{\alpha}{2} \normsq{x}, \\
&f_1(x):=  \gamA \gampi [2 \tsum_{i=1}^{k} |x_{2i-1} - x_{2i}|  - (\tfrac{3}{2} + \tfrac{1}{k}) x_1], \\
&f_2(x):=  \gamA \gampi [2 \tsum_{i=1}^{k} |x_{2i} - x_{2i+1}|  - (\tfrac{1}{2} + \tfrac{1}{k}) x_1]. \\
\end{split}
\end{align}
In particular, the scenario cost functions $\myfi[1]$ and $\myfi[2]$  are specified in the structured maximization form (c.f. \eqref{eq:orig_prob}), $\myfi(x):= \max_{\pii \in \Pii}\inner{\Ai x}{\pii} - \fistar(\pi)$ with
\begin{align}\label{ls:ns_fistar_lower}
\begin{split}
&\Ai[1] := \gamA \scriptsize \begin{bmatrix} - (\tfrac{3}{2} + \tfrac{1}{k}) & \\
1 & -1 \\
&& 1& -1 \\
&&&&\ddots \\
&&&&& 1 &-1 & 0
\end{bmatrix},\ \Pii[1] := \gampi (\{1\} \times [-2, 2]^k) \subset \R^{k+1},\ \fistar[1](\pii[1]) \equiv 0, \\
&\Ai[2] := \gamA \scriptsize \begin{bmatrix} - (\tfrac{1}{2} + \tfrac{1}{k}) & \\
& 1 & -1 \\
&&& 1& -1 \\
&&&&&\ddots \\
&&&&&& 1 &-1
\end{bmatrix}, \Pii[2] := \gampi (\{1\} \times [-2, 2]^k) \subset \R^{k+1},\ \fistar[2](\pii[2]) \equiv 0.
\end{split}
\end{align}
Clearly, $\max_{i \in [2]} \norm{\Ai}[2,2] \leq 2 \gamA$ and $\max_{i \in [m]}\max_{\pii, \piibar\in \Pii} \norm{\pii- \piibar} \leq 5 \sqrt{k} \gampi$. Furthermore, $\pstar:=[1/2, 1/2]$ and $\xstar:=\tfrac{\gampi\gamA}{2k\alpha}[2, 1, 1, \ldots, 1]$ form the optimal solution since they satisfy the first order optimality conditions given by:
{\skipdisplay
$$0 \in \partial (\tfrac{1}{2} f_1 + \tfrac{1}{2} f_2 + u)(\xstar)  \text{ and } (\tfrac{1}{2}, \tfrac{1}{2}) \in \textstyle\argmax_{p \in \Delta^+_2} p_1 f_1(\xstar) + p_2 f_2(\xstar).$$ }
So 
$\fstar=f(\xstar)= - (\tfrac{\gampi^2\gamA^2}{2k^2 \alpha} + \tfrac{\gamA^2\gampi^2}{4 k \alpha}),\ \norm{\xstar - \xt[0]}^2 = \tfrac{\gamA^2\gampi^2}{4k^2 \alpha^2} (2k + 4) \leq \tfrac{\gamA^2\gampi^2}{k \alpha^2}.$ 
\vgap

Now, let us verify \eqref{lb_ns} satisfies the even-odd preserving property. Consider the worker $\myfi[1]$. Let an even $i\geq 2$ be given and assume $\Mit[t-1][1] \subset \Ki$, i.e., $\Mitcp[t][1][,0] \subset \Ki$.  Let $\Si[j]$ denote a (dual) subspace with non-zero entries only in the first $j$ coordinates, $\Si[j]:= \{\pi \in \R^{k+1}: \pi_l = 0\ \forall l > j\}$. Because of the block structure of $\Ai[1]$,
if the $i+1$\ts{th} coordinate of $\pii[1]$ is non-zero for any $i \geq 2$, the $2i-1$\ts{th} and the $2i$\ts{th} coordinates of $\Ai[1]\pii[1]$ must be non-zero. So we have $\Mpiitl[t][0][1] \subset \Si[i/2 + 1]$; otherwise the update rule \eqref{lower:dual_prox} in the DPM environment would lead to $\Mit[t-1][1] \not\subset \Ki$.

Next, we show $\Mitcp[t][1][,l] \subset \Ki$ and $\Mpiitl[t][l][1] \subset \Si[i/2 + 1]$ for all $l \geq 0$ by induction. Clearly the statement holds for $l=0$. Moreover, if  $\Mitcp[t][1][,l-1] \subset \Ki$ and 
$\Mpiitl[t][l-1][1] \subset \Si[i/2 +1]$, $\Ai[1]\bar x$ and $\piibar[1]$ must be in $\Si[i/2+1]$, so the dual proximal in \eqref{lower:dual_prox} can be written as 
{\skipdisplay
$$\piit[t,l][1] = \targmax_{\pii[1] \in \Pii[1]} \tsum_{j=1}^{i/2+1} (\Ai[1]\bar x)_j \pii[1,j] - \tau/2 (\pii[1,j] - \piibar[1,j])^2 - \tau/2 \tsum_{j=i/2 + 2}^{k+1} (\pii[1,j])^2.$$}
This leads us to $\piit[t,l][1] \in \Si[i/2 + 1]$ and $\Aitr[1] \piit[t,l][1] \in \Ki$, i.e., $\Mitcp[t][1][,l] \subset \Ki$ and $\Mpiitl[t][l][1] \subset \Si[i/2 + 1]$. Then the principle of induction implies that the statement holds for all $l \geq 0$, i.e., $\Mitcp[t][1] \subset \Ki$. In addition, when $i \geq 2$ is odd and  $\Mit[t-1][1] \subset \Ki$, we have $\Mit[t-1][1] \subset \Ki[i+1]$. Since $i+1$ is even, the preceding result implies that $\Mitcp[t][1] \subset \Ki[i+1]$. The property \eqref{ls:odd-even} for the worker $\myfi[2]$ for both even and odd $i$'s can also be deduced in a similar way. Therefore we have shown that the even-odd preserving property holds for the hard problem \eqref{lb_ns}.

Next, applying Lemma \ref{lm:even-odd}, the solution $\xt[k]$ returned by any DPO algorithm in  $k \geq 4$ communication rounds must satisfy $\xt[k] \in \Ki[k]$. We provide a lower bound of $f$ on $\Ki[k]$. Let $\bar{f}:= \tfrac{1}{2}(f_1 + f_2) + u(x)$ denote a uniform lower bound for $f$ given by 
{\skipdisplay
$$\bar f(x):= \gampi\gamA [ \tsum_{i=1}^{2k} |x_{i} - x_{i+1}|  - (1 + \tfrac{1}{k}) x_1] + \tfrac{\alpha}{2} \normsq{x}.$$}
 In order to find the minimum of $\bar f$ on $\Ki[k]$, observe that arranging $\{x_i\}_{i=1}^{k}$ in a decreasing order decreases $\bar f$. Moreover, if $x_k <0$, setting all negative coordinates to zero decreases $\bar f$, so we can focus on $x_1 \geq x_2 ... \geq x_{k} \geq x_{k+1} = ...= x_{2k+1}=0$.
{\skipdisplay $$\tmin_{x \in \Ki[k]} \bar f(x)= \tmin_{x \in \Ki[k]} - \tfrac{\gampi\gamA}{k} x_1 + \tfrac{\alpha}{2} x_1^2 + \tfrac{\alpha}{2} \tsum_{i=2}^{2k+1} x_i^2 \geq - \tfrac{\gampi\gamA}{2k^2 \alpha}.$$}
Thus,  $f(\xt[k]) -\fstar \geq \min_{x \in \Ki[k]} \bar f(x)  - \fstar \geq {\gamA^2\gampi^2}/{4 k \alpha}.$

Finally, we choose appropriate problem parameters to establish the lower bounds. If $\ep \geq \DPi \MA \Rt[0]/{400 \ep}$, $\Omega(\MA \DPi \Rt[0]/\ep) =\Omega(1)$, so the lower bound in a) clearly holds. Otherwise, setting $k:= \ceil{{\DPi \MA \Rt[0]}/{96 \ep}}$, $n :=2k +1$, $\gampi:= {\DPi}/{5 \sqrt{k}},\ \gamA:= {\MA}/{2}$, and $\alpha := {\gamA\gampi}/{\Rt[0] \sqrt{k}}$, the parameters of \eqref{lb_ns} satisfy 
$\tmax_{i \in [2]} \norm{\Ai}[2,2] \leq \MA,\ \tmax_{i \in [m]}\tmax_{\pii, \piibar\in \Pii} \norm{\pii- \piibar} \leq \DPi,\  4\leq k\text{, and} \norm{\xt[0] - \xstar} \leq \Rt[0].$
 Since the minimum optimality gap attainable in $k = \Omega(\MA \DPi \Rt[0]/\ep)$ communication rounds is lower bounded by $\ep$, the result in a) follows.

Now consider $u(x)$ being $\alpha$-strongly convex for a fixed $\alpha>0$. If $\ep \geq \DPi^2 \MA^2 /{40000 \alpha}$,  $\Omega(\MA \DPi/\sqrt{\alpha \ep})= \Omega(1)$, so the lower bound in b) clearly holds. Otherwise, 
setting $k:= \ceil{{\DPi \MA}/{48\sqrt{\alpha \ep}}}$, $n:=2k+1$, $\gampi:= {\DPi}/{5 \sqrt{k}},$ and $ \gamA:= {\MA}/{2}$, the parameters of \eqref{lb_ns} satisfy 
{$\tmax_{i \in [2]} \norm{\Ai}[2,2] \leq \MA,\ \tmax_{i \in [m]}\tmax_{\pii, \piibar\in \Pii} \norm{\pii- \piibar} \leq \DPi,\  4 \leq k \text{, and} \norm{\xt[0] - \xstar} \leq \Rt[0].$ }
 Since the minimum optimality gap attainable in $k = \Omega(\MA\DPi/\sqrt{\alpha \ep})$ communication rounds is lower bounded by $\ep$, the result in b) follows. 
\end{proof}

\section{Numerical Experiments}

In this section, we present a few numerical experiments to verify the theoretical convergence properties of the proposed DRAO-S method. 
\textblue{
\subsection{Implementation Details}
The  numerical  experiments  are  implemented in  MATLAB  2021b and  are  tested  on an Alienware Desktop with a 4.20 GHz Intel Core i7 processor and 16 GB of 2400MHz DDR4 memory. The stepsize of the DRAO-S method are chosen according to Theorem \ref{thm:sps-sm}, \ref{thm:sps-sm-str} and \ref{thm:sps-ns}. The implementation details of the proximal mappings are deferred to the Appendix. Parameter tuning is used to achieve better empirical performance. The DRAO-S method is first tested on a few trial stepsizes, each running for only 20 phases. Next, the one achieving the lowest objective value during the trials is selected to run till the desired accuracy, subject to a termination limit of 5000 phases. The trial stepsizes are calculated according to \eqref{stp:sps_sm}, \eqref{stp:sps_sm_str}, \eqref{stp:drao-s-ns} and \eqref{stp:drao-sps-ns-str}.   The trial stepsizes are calculated from conservative estimates of $\DP \geq \norm{\pt[0] - \pstar}$ and $\Rt[0]\geq  \norm{\xt[0] - \xstar}$, and from a few scaled estimates of $\Lf$, $\MtU$, $\MA$ and $\MAPi$. Specifically, for the smooth linear regression problem \eqref{eq:linear_reg}, the parameters $\Lf$ and $\MtU$ used for the calculations in \eqref{stp:sps_sm} and \eqref{stp:sps_sm_str} are given by (refer to Subsection \ref{subsec:np-reg} for the definition of $H_i$)
$$
\begin{tabular}{|l | l | l|}
\hline
Parameter & Choices & Conservative Estimate \\
\hline
$\Lf$ & $\{\widehat{\Lf}, 0.3 \widehat{\Lf}\}$ & $\widehat \Lf:= \tmax_{ i \in [m]}\norm{H_i^\top H_i}$ \\
\hline
$\MtU$ & $\{\widehat{\MtU}, 0.3 \widehat{\MtU}\}$ & $\widehat{\MtU} := \norm{[\grad \myfi[1](\xundert); \ldots; \grad \myfi[m](\xundert)]}$\\
\hline
\end{tabular}.
$$
So there are four sets of trial stepsizes. For the structured non-smooth two-stage stochastic program, the parameters $\MAPi$ and $\MA$ used for the calculations in \eqref{stp:drao-s-ns} and \eqref{stp:drao-sps-ns-str} are given by (refer to Subsection \ref{subsec:np-reg} for the definition of $T_i$ and $e_i$)
$$
\begin{tabular}{|l | l | l|}
\hline
Parameter & Choices & Conservative Estimate \\
\hline
$\MA$ & $\{\widehat{\MA}, 0.3 \widehat{\MA}, 0.1 \widehat{\MA}\}$ & $\widehat \MA:= \tmax_{ i \in [m]}\norm{T_i}$ \\
\hline
$\MAPi$ & $\{\widehat{M_{A\Pi}}, 0.3 \widehat{M_{A\Pi}}, 0.1 \widehat{M_{A\Pi}}\}$ & $\widehat{M_{A\Pi}} := \norm{[T_1^\top e_1; \ldots; T_m^\top e_m]}$\\
\hline
\end{tabular}.$$
So there are nine sets of trial stepsizes.}
\subsection{Risk Averse Linear Regression Problem}
\label{subsec:np-reg}
\begin{figure}[H]
\centering
\begin{subfigure}{0.25\textwidth}
\includegraphics[width=\textwidth]{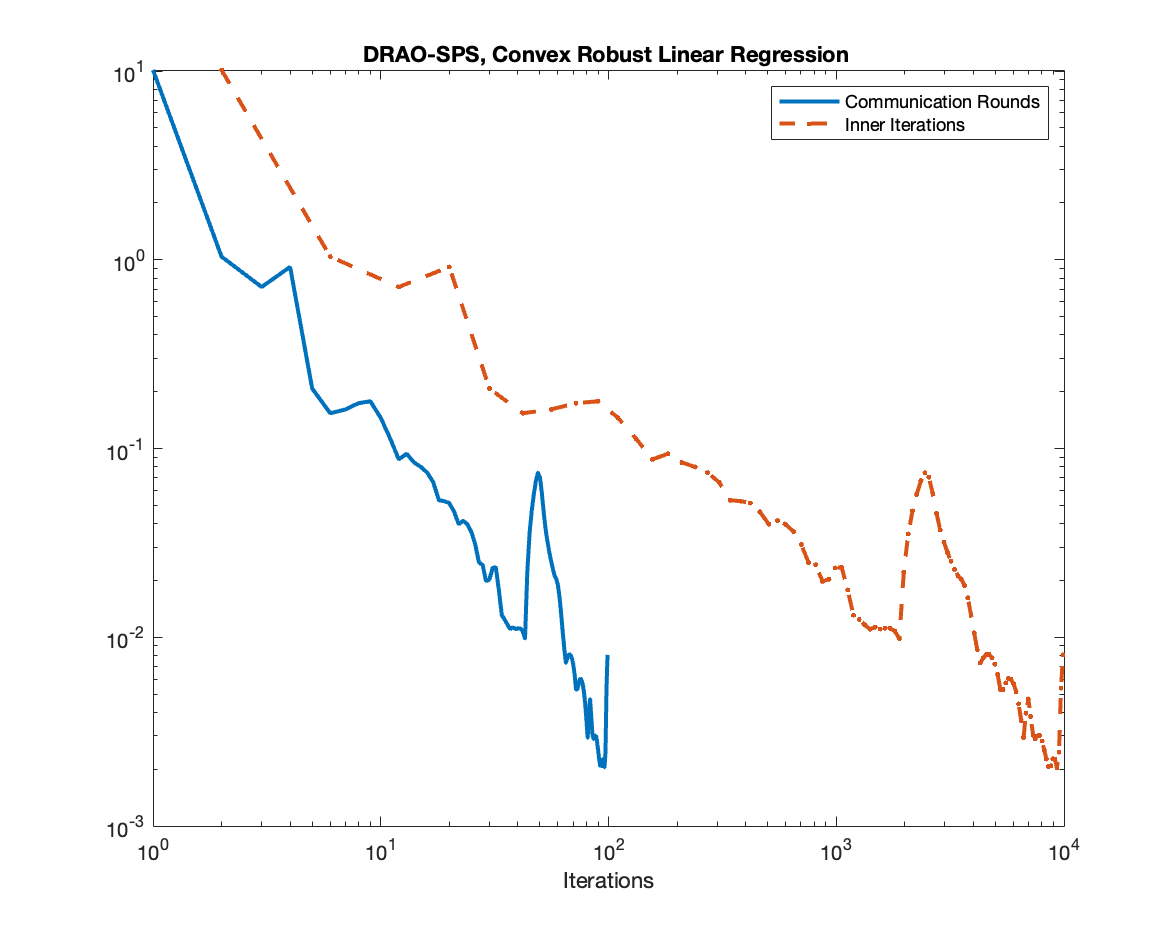}
\caption{figure: $\alpha=0$}\label{fig:sm}
\vspace{-6mm}
\end{subfigure} 
\begin{subfigure}{0.1\textwidth}
\end{subfigure}
\begin{subfigure}{0.45\textwidth}
\includegraphics[width=\textwidth]{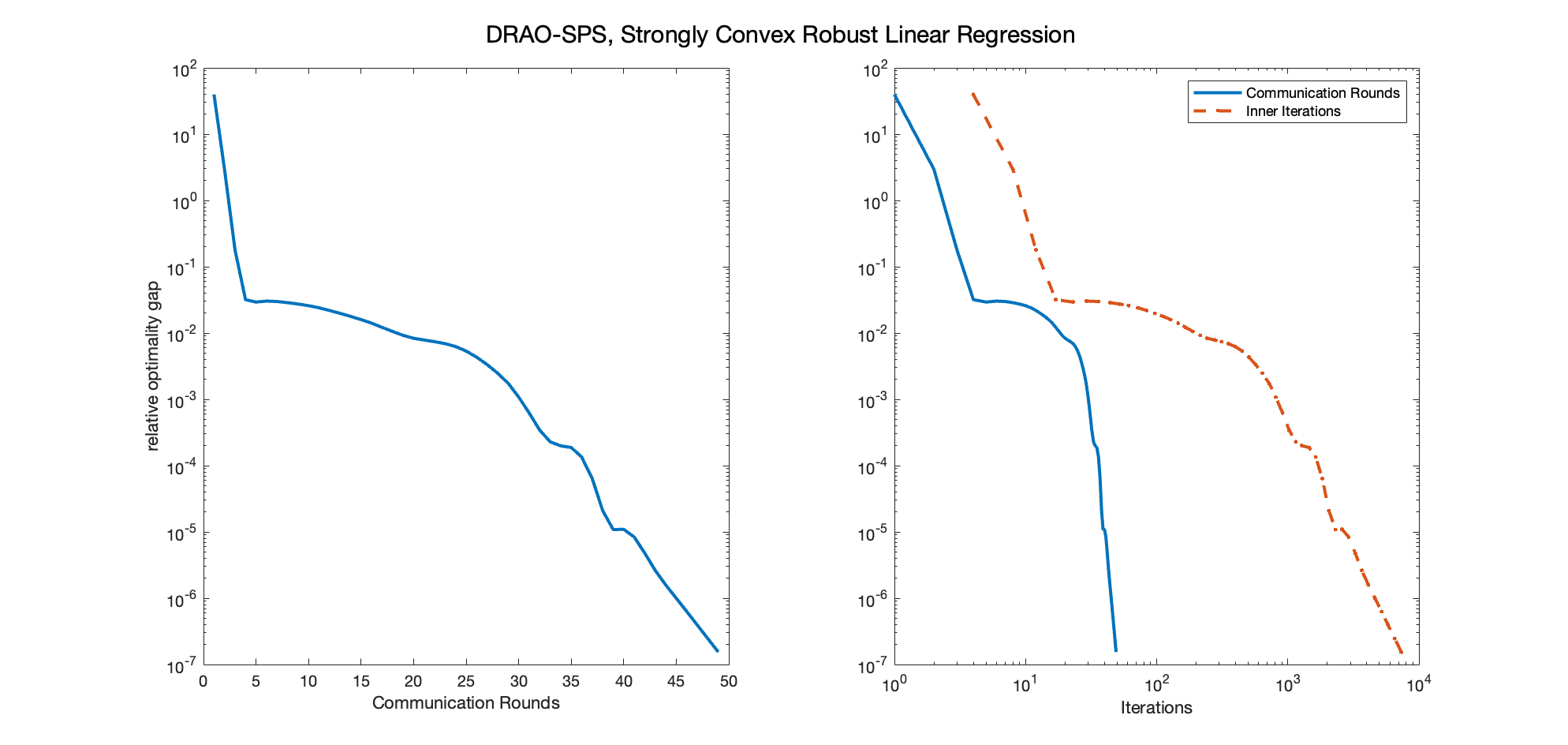}
\caption{figure: $\alpha>0$}\label{fig:sm_str}
\vspace{-6mm}
\end{subfigure}
\caption{\small Convergence of DRAO-S for a Randomly Generated Robust Linear Regression Problem}\label{fig:DRAO-SPS-SM}
\vspace{-3mm}
\end{figure}

\begin{table}
  \centering
  \begin{threeparttable}
    \tiny
\begin{tabular}{|l|l| r r | r r| r r|}
\toprule
\#Scenarios&  Opt. Gap & 
\multicolumn{2}{|c|}{\@ 10\% Risk} & \multicolumn{2}{|c|}{\@ 5\% Risk} & \multicolumn{2}{|c|}{\@ 1.25\% Risk}  \\
\midrule 
\multicolumn{8}{|c|}{Non-strongly Convex $\alpha =0$}\\
\midrule
& & \#Comm & \#$P$-proj& \#Comm & \#$P$-proj & \#Comm & \#$P$-proj\\
\midrule
  & 10\% & 
  6  & 16 &
     8 & 50 & 8  &  50   \\
 20 & 1\% & 
  31   & 529 & 38 & 1236 & 38 & 1236 \\
    & 0.1\% & 
  76  & 2858 & 85 & 6282 & 85 & 6282 \\
\midrule

 & 10\% & 
  3 & 4 & 6& 16 & 7 & 40 \\
 50 & 1\% & 16 & 126 & 20 & 199 & 34 & 1162 \\
  & 0.1\% & 63 & 1935 & 67 & 2206 & 68 & 4245 \\

  \midrule
& 10\% & 3 & 4 & 3 & 4 & 5 & 14 \\
200 & 1\% & 6 & 19 & 12 & 72 & 23 & 260 \\
& 0.1\% & 30 & 516 & 51 & 1336 & 65 & 2113 \\

  \midrule
  \multicolumn{8}{|c|}{Strongly Convex Condition Number $\kappa =10$}\\
  \midrule
  & 1e-3 & 
  32  & 665 &
     39  & 1254 & 39  &  1254   \\
 20 & 1e-4 & 
  43   & 1926 & 44 & 2139 & 44 & 2139 \\
    & 1e-5 & 
  48 & 2980 & 49 & 3603 & 49 & 3603 \\
\midrule

 & 1e-3 & 
  19 & 159 & 29 & 485 & 39 & 1306 \\
 50 & 1e-4 & 36 & 854 & 38 & 1132 & 44 & 2051 \\
  & 1e-5 & 41 & 1538 & 43 & 2061 & 49 & 3754 \\

  \midrule

& 1e-3 & 16 & 54 & 14 & 69 & 29 & 454 \\
200 & 1e-4 & 28 & 205 & 32 & 437 & 42 & 1777 \\
& 1e-5 & 40 & 660 & 44 & 1357 & 46 & 2934 \\

\bottomrule
\end{tabular}
  \end{threeparttable}
    \caption{\small Communications Rounds and $P$-Projections Required by DRAO-S for Linear Regression under a CV@R Risk}
\label{tb:smooth_numerical}
\vspace{-6mm}
\end{table}


For the smooth case, the following risk-averse linear regression problem of the form \eqref{eq:orig_prob} is considered:
 \begin{equation}\label{eq:linear_reg}
 f(x):= \text{CV@R}_{\delta}(\myfi[1](x), \ldots, \myfi[m](x))  + \tfrac{\alpha}{2}\normsq{x} \text{ with }f_i(x) := \tfrac{1}{2} \| H_i x - b_i\|^2, X := \R^n.\end{equation}
Here $\myfi$ denotes the loss function associated with the $i$\ts{th} dataset. Such a problem is motivated by the need for a single robust model under fairness or risk considerations.  For example, the state  education department might wish to build a model to help teachers to identify students who need extra help. $(H_i, b_i)$ could represent the data collected in the $i$\ts{th} county  and the CV@R risk measure could be used to ensure fairness among counties.

In our experiments, we  set $n=40$, and generate matrices $H_i \in \R^{40 \times 200},$ and $\ b_i \in \R^{40}$ randomly. We generate an estimate of $f_*$  by running the bundle level method \cite{lemarechal1995new} to an extremely high degree of accuracy. We record the average number of communication rounds and $P$-projections steps needed, over five randomly generated instances,  to achieve the desired relative optimality gap, i.e., $(f(\xt) - f_*) / f_* \leq \ep$ under different settings.  
In particular, the DRAO-S method is tested on problems with different levels of risk and different numbers of computing nodes to understand how the communication and the $P$-projection complexities vary with $\DP$ and $m$  in practice.
The results are presented in Table \ref{tb:smooth_numerical}. For the number of computing nodes $m$, both the number of $P$-projections and the number of communication rounds scale well with it. In fact, they seem to decrease slightly when $m$ increases. For $\DP$,  recall that a lower risk level corresponds to a larger ambiguity set $P$ and hence a larger radius $\DP$ (c.f. Subsection \ref{subsec:notation}). Both the number of $P$-projections and the number of communication rounds increase with $\DP$, but the number of communication rounds seems to have a weaker dependence on it.
Additionally, typical convergence curves of the DRAO-S method are plotted in Figure \ref{fig:DRAO-SPS-SM} and they seem to verify the theoretical convergence guarantees.
When $\alpha=0$, Table \ref{tb:smooth_numerical} and the convergence curve in Figure \ref{fig:sm} illustrate a communication complexity and a $P$-projection complexity on the order of  $\bigO(1/\sqrt{\ep})$ and $\bigO(1/\ep)$, respectively. When $\alpha>0$, the convergence curves in Figure \ref{fig:sm_str} and Table \ref{tb:smooth_numerical} illustrate a communication complexity and a $P$-projection complexity on the order of  $\bigO(\log(1/\ep))$  and $\bigO(1/\sqrt{\ep})$, respectively.
Thus, the DRAO-S method can find highly accurate solutions within a small number of communication rounds.

\subsection{Risk Averse Two-Stage Stochastic Programming}\label{subsec:num-sp}
\begin{figure}[ht]
\centering
\begin{subfigure}{0.3\textwidth}
\includegraphics[width=\textwidth]{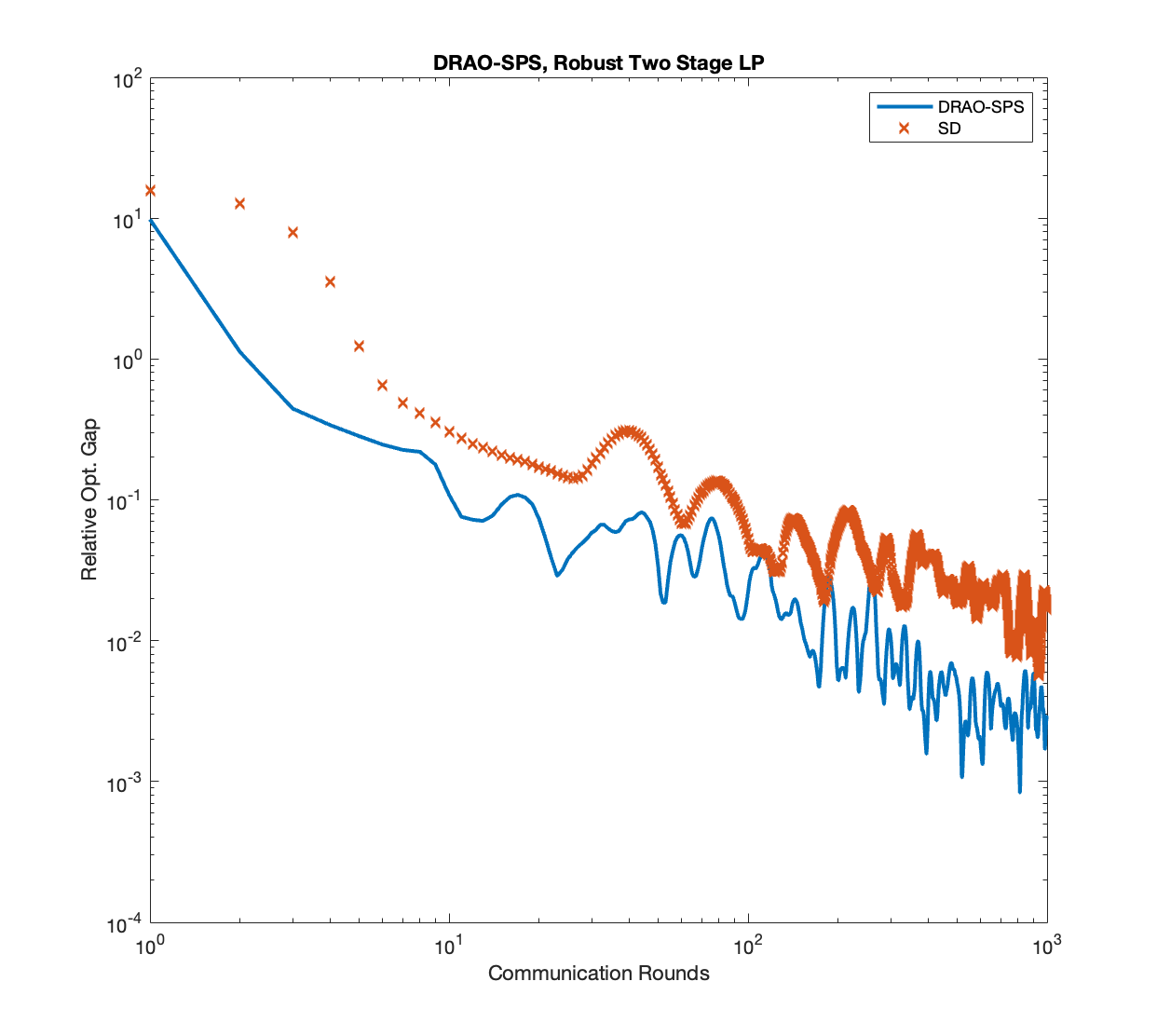}
\caption{figure: $\alpha=0,$ 6 Inner Iterations}\label{fig:ns}
\vspace{-6mm}
\end{subfigure} 
\begin{subfigure}{0.1\textwidth}
\end{subfigure}
\begin{subfigure}{0.36\textwidth}
\includegraphics[width=\textwidth]{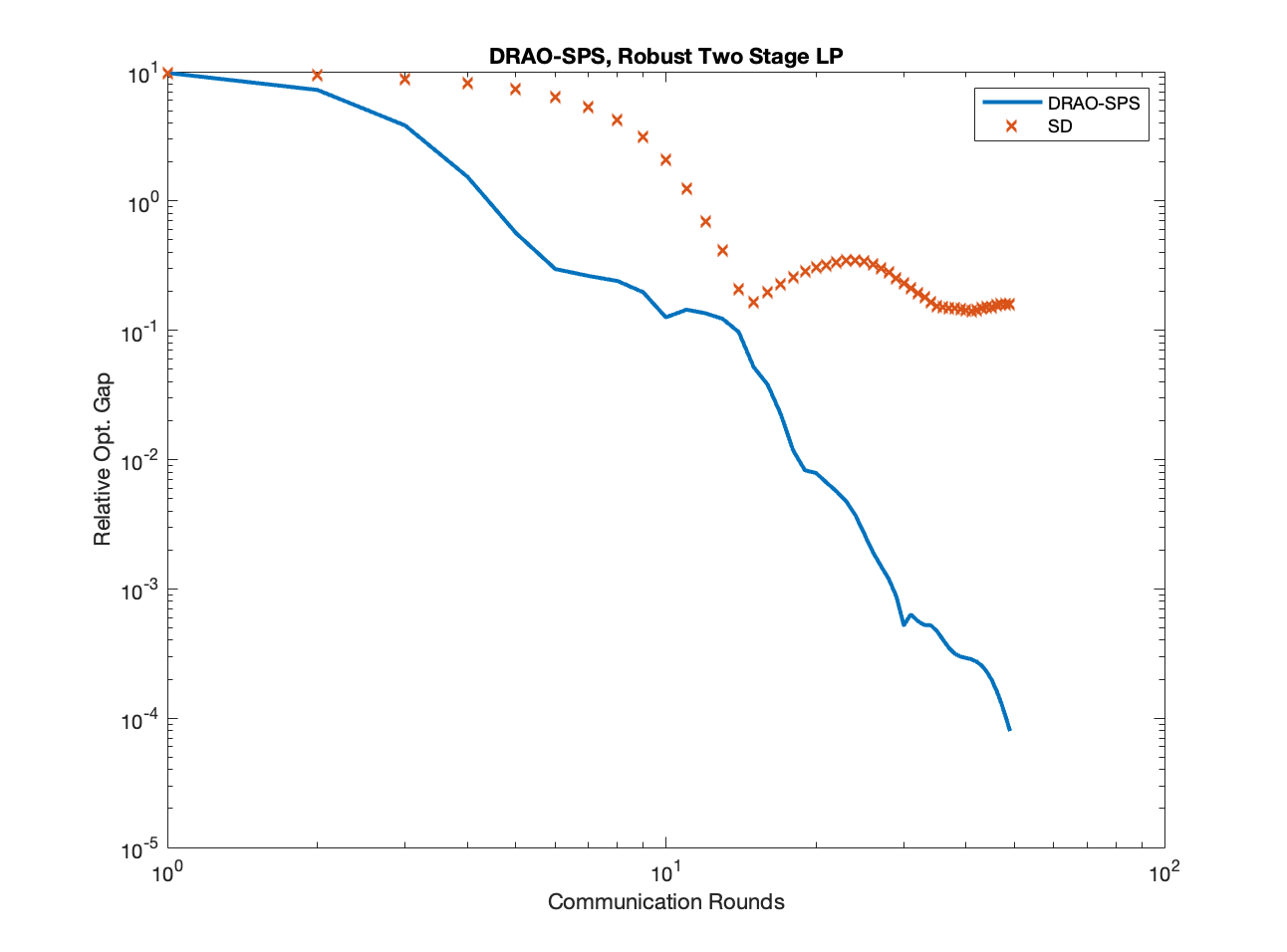}
\caption{figure: $\alpha>0$, 36 Inner Iterations}\label{fig:ns_str}
\vspace{-6mm}
\end{subfigure}
\caption{Convergence of DRAO-S for a Randomly Two-Stage Linear Problem}\label{fig:DRAO-SPS-NS}
\vspace{-3mm}
\end{figure}
\begin{table}[ht!]
\centering
\begin{threeparttable}
\tiny
\begin{tabular}{|l|l| r r | r r| r r|}
\toprule
\#Scenarios&  Opt. Gap & 
\multicolumn{2}{|c|}{\@ 10\% Risk} & \multicolumn{2}{|c|}{\@ 5\% Risk} & \multicolumn{2}{|c|}{\@ 1.25\% Risk}  \\
\midrule 
\multicolumn{8}{|c|}{Non-strongly Convex $\alpha =0$}\\
\midrule
& & SD & DRAO-S & SD & DRAO-S & SD & DRAO-S\\
\midrule
 20 & 10\% & 
  250  & \textbf{59} &
     386 & \textbf{103} & 386 &    \textbf{103}   \\

 & 1\% & 
  2328  & \textbf{420} & 2834 & \textbf{597} & 2834 & \textbf{597} \\
\midrule

 50 & 10\% & 
  357 & \textbf{83} & 520 & \textbf{79} & 733 & \textbf{94} \\

 & 1\% & 2699 & \textbf{511} & 2971 & \textbf{590} & 4292 & \textbf{543} \\
  \midrule

200 & 10\% & 74 & \textbf{12} & 183 & \textbf{16} & 447 & \textbf{34} \\
& 1\% & 1614 & \textbf{187} & 3354 & \textbf{275} & NA & \textbf{293} \\

  \midrule
  \multicolumn{8}{|c|}{Strongly Convex $\alpha =1$}\\
  \midrule
  & 10\% & 
  43 & \textbf{13} & 41 & \textbf{14} & 41 & \textbf{14} \\

 20 & 1\% & 
  181 & \textbf{24} & 129 & \textbf{25} & 129 & \textbf{25} \\

  & 0.1\% & 
  494 & \textbf{41} & 321 & \textbf{47} & 321 & \textbf{47} \\

\midrule
 & 10\% & 
  26 & \textbf{13} & 35 & \textbf{14} & 65 & \textbf{15}   \\
50   & 1\% & 98 & \textbf{21} & 131 & \textbf{23} & 183 & \textbf{24} \\
  & 0.1\% & 226 & \textbf{37} & 258 & \textbf{41} & 335 & \textbf{43} \\
\midrule

 & 10\% & 11& \textbf{10} & 20 & \textbf{12} & 44 & \textbf{14} \\
200 & 1\% & 75 & \textbf{18} & 125 & \textbf{20} & 208 & \textbf{24} \\ 
& 1\% & 320 & \textbf{25} & 508 & \textbf{31} & 438 & \textbf{41} \\ 

\bottomrule
\end{tabular}
\begin{tablenotes}
\item NA : Algorithm has not reached specified accuracy after  5000 communication rounds.
\end{tablenotes}
\end{threeparttable}
\caption{\small Communication Rounds Required by Two-Stage Stochastic Program under a CV@R Risk}\label{tb:ns_numerical}
\vspace{-5mm}
\end{table}

For the structured non-smooth case,  we compare the DRAO-S method with the SD method \cite{zhang2019efficient} using the same risk-averse two-stage stochastic linear programming problem from \cite{zhang2019efficient}: 
\begin{align}\label{eq:n_test}
\begin{split}
 \min_{x \in \R^n}\quad & c^\intercal x + \text{CV@R}_{\delta}( g_1(x), \ldots, g_m(x))  + \tfrac{\alpha}{2}\| x\|^2, \\
  s.t.\ \   & 0 \leq x_j \leq U \ \forall j \in [n],\\
& g_i(x) := \tmin_{y_i \in \R^l_+} \{y_i^\intercal e_i,\  s.t.\ \    R y_i \geq d_i - T_i x\}.\\
\end{split}
\end{align}
The problem models the capacity expansion decision of an electricity company. Being the sole provider of electricity, the company has to meet all demand profiles $\{d_i\}$ using a combination of installed capacity, with an availability factor of $T_i$, and electricity purchased from outside the grid, at a unit cost of $e_i$. Being risk averse, the company intends to find a decision that keeps the total cost low for roughly $(1-\delta)$ of all possible scenarios.  

In our experiments, we set $ n=40$ and $ l=20$,  generate $T_i \in \R^{20 \times 40},\ e_i \in \R^{20}, d_i \in \R^{20}$ and $ c\in \R^{40}$   randomly, and choose $R := I_{20, 20}$ to be the simple complete recourse matrix. We record the average number of communication rounds required to achieve the desired relative optimality gaps for both methods in Table \ref{tb:ns_numerical}. Clearly, DRAO-S enjoys significant savings compared to the SD method. The number of communications rounds required by DRAO-S is also less sensitive to the risk level and  $\DP$.  Moreover, typical convergence curves are plotted in Figure \ref{fig:ns} and \ref{fig:ns_str}. They seem to verify the theoretical communication complexities  of DRAO-S on the orders of  $\bigO(1/\ep)$ and $\bigO(1/\sqrt{\ep})$, respectively, for the non-strongly convex and the strongly convex problems.

\subsection{Risk Measure induced by the $\chi^2$ Ambiguity Set}
Next, we test these algorithms on a more complicated quadratically constrained set $P$. Given a radius parameter $r$, the modified $\chi^2$ probability uncertainty set respect to the empirical probability $[1/m, \ldots, 1/m]$ is given by 
$$P_r = \{p \in \R^m_+: \tsum_{i=1}^{m} p_i = 1, \normsq{p - [1/m, \ldots, 1/m]} \leq r\}.$$
Inspired by the $\chi^2$ test, $P_r$ is useful for distributionally robust optimization (DRO) \cite{chen2021decomposition}. We conduct our experiments with the induced risk-measure $\rho(g) = \max_{p \in P_r} \inner{p}{g}$ on both the linear regression problem \eqref{eq:linear_reg} and the two-stage stochastic program \eqref{eq:n_test}. The average number of communication rounds required to reach the desired sub-optimalities for various levels of $r$ are recorded in Table \ref{tb:X2_smooth_numerical} and \ref{tb:X2_ns_numerical}. Since a larger $r$ implies a larger $P$, the results are consistent with our findings under the CV@R setting. 

\begin{table}
  \centering
 
  \begin{threeparttable}
    \tiny
\begin{tabular}{|l|l| r r | r r| r r|}
\toprule
\#Scenarios&  Opt. Gap & 
\multicolumn{2}{|c|}{\@ $r=0.05$} & \multicolumn{2}{|c|}{\@ $r=0.1$} & \multicolumn{2}{|c|}{\@ $r=0.2$}  \\
\midrule 
\multicolumn{8}{|c|}{Non-strongly Convex $\alpha =0$}\\
\midrule
& & \#Comm & \#$P$-proj& \#Comm & \#$P$-proj & \#Comm & \#$P$-proj\\
\midrule
    & 10\% & 3  & 4 & 3 & 4 & 3  &  4   \\
 20 & 1\% & 14   & 93 & 17 & 134 & 20 & 188 \\
    & 0.1\% & 28  & 274 & 40 & 944 & 73 & 2710 \\
\midrule

    & 10\% & 3 & 4 & 3& 4 & 3 & 4 \\
 50 & 1\% & 8 & 33 & 13 & 83 & 19 & 162 \\
    & 0.1\% & 21 & 216 & 30 & 524 & 69 & 2369 \\

  \midrule
    & 10\% & 3 & 4 & 3 & 4 & 3 & 4 \\
200 & 1\% & 7 & 22 & 14 & 95 & 19 & 164 \\
    & 0.1\% & 21 & 207 & 40 & 872 & 70 & 2426 \\

  \midrule
  \multicolumn{8}{|c|}{Strongly Convex Condition Number $\kappa =10$}\\
  \midrule
    & 1e-3 & 6  & 10 & 9  & 28 & 26  &  235   \\
 20 & 1e-4 & 32   & 216 & 33 & 343 & 35 & 605 \\
    & 1e-5 & 35 & 312 & 36 & 462 & 38 & 754 \\
\midrule

    & 1e-3 & 6 & 10 & 15 & 67 & 26 & 244 \\
 50 & 1e-4 & 29 & 164 & 32 & 323 & 35 & 566 \\
    & 1e-5 & 34 & 266 & 35 & 409 & 37 & 690 \\

  \midrule

    & 1e-3 & 20 & 68 & 18 & 80 & 25 & 220 \\
200 & 1e-4 & 28 & 163 & 34 & 356 & 36 & 676 \\
    & 1e-5 & 34 & 280 & 36 & 450 & 39 & 882 \\

\bottomrule
\end{tabular}
  \end{threeparttable}
    \caption{\small  Communications Rounds and $P$-Projections Required by DRAO-S for Linear Regression under a modified $\chi^2$ Risk Measure}
\label{tb:X2_smooth_numerical}
\vspace{-6mm}
\end{table}

\begin{table}[ht!]
\centering
\begin{threeparttable}
\tiny

\begin{tabular}{|l|l| r r | r r| r r|}
\toprule
\#Scenarios&  Opt. Gap & 
\multicolumn{2}{|c|}{\@ $r=0.05$} & \multicolumn{2}{|c|}{\@ $r=0.1$} & \multicolumn{2}{|c|}{\@ $r=0.2$}  \\
\midrule 
\multicolumn{8}{|c|}{Non-strongly Convex $\alpha =0$}\\
\midrule
& & SD & DRAO-S & SD & DRAO-S & SD & DRAO-S\\
\midrule
 20 & 10\% & 
  88  & \textbf{41} &
     135 & \textbf{49} & 190 &    \textbf{54}   \\

 & 1\% & 
  703  & \textbf{292} & 1006 & \textbf{330} & 2032 & \textbf{343} \\
\midrule

 50 & 10\% & 
  240 & \textbf{46} & 319 & \textbf{57} & 388 & \textbf{75} \\

 & 1\% & 1543 & \textbf{313} & 2146 & \textbf{377} & 2838 & \textbf{409} \\
  \midrule

200 & 10\% & 194 & \textbf{16} & 273 & \textbf{32} & 332 & \textbf{43} \\
& 1\% & 1747 & \textbf{270} & 2818 & \textbf{315} & 3191 & \textbf{335} \\

  \midrule
  \multicolumn{8}{|c|}{Strongly Convex $\alpha =1$}\\
  \midrule
  & 10\% & 
  10 & \textbf{10} & 14 & \textbf{11} & 21 & \textbf{12} \\

 20 & 1\% & 
  41 & \textbf{17} & 68 & \textbf{19} & 83 & \textbf{21} \\

  & 0.1\% & 
  188 & \textbf{28} & 301 & \textbf{37} & 338 & \textbf{39} \\

\midrule
 & 10\% & 
  \textbf{9} & 10 & 15 & \textbf{12} & 46 & \textbf{14}   \\
50   & 1\% & 52 & \textbf{18} & 96 & \textbf{20} & 199 & \textbf{23} \\
  & 0.1\% & 231 & \textbf{29} & 348 & \textbf{36} & 605 & \textbf{42} \\
\midrule

 & 10\% & 14 & \textbf{11} & 21 & \textbf{13} & 33 & \textbf{14} \\
200 & 1\% & 100 & \textbf{18} & 158 & \textbf{21} & 183 & \textbf{24} \\ 
& 1\% & 556 & \textbf{30} & 771 & \textbf{36} & 671 & \textbf{37} \\ 

\bottomrule
\end{tabular}
\end{threeparttable}
\caption{\small  Communication Rounds Required by Two-Stage Stochastic Program under a modified $\chi^2$ Risk Measure}\label{tb:X2_ns_numerical}
\vspace{-5mm}
\end{table}

\section{Conclusion}

This paper introduces the problem of distributed risk-averse optimization. A conceptual DRAO method and a more practical DRAO-S method are proposed. Both of them are able to solve the risk-averse problem with the same communication complexities as those for solving the risk-neutral problem. The optimality of their communication complexities is established with matching lower bounds. And preliminary numerical experiments seem to indicate promising empirical performance for DRAO-S.

In future work, we will attempt to extend our proposed methods to the more general cross-device federated learning setting \cite{kairouz2021advances} where $\myfi$'s are accessible only via a stochastic first-order oracle and the communication network is unreliable. We will also attempt to study the extension to more complicated risk measures for which $p$-prox mappings are prohibitively expensive and only gradient evaluations are possible.  

\bibliographystyle{siam}
\bibliography{Reference}

\begin{thebibliography}{10}

\bibitem{Beck2017First}
{\sc A.~Beck}, {\em First-order methods in optimization}, vol.~25, SIAM, 2017.

\bibitem{bertsimas2006robust}
{\sc D.~Bertsimas and A.~Thiele}, {\em Robust and data-driven optimization:
  modern decision making under uncertainty}, in Models, methods, and
  applications for innovative decision making, INFORMS, 2006, pp.~95--122.

\bibitem{chambolle2011first}
{\sc A.~Chambolle and T.~Pock}, {\em A first-order primal-dual algorithm for
  convex problems with applications to imaging}, Journal of mathematical
  imaging and vision, 40 (2011), pp.~120--145.

\bibitem{chen2021decomposition}
{\sc Y.~Chen, H.~Sun, and H.~Xu}, {\em Decomposition and discrete approximation
  methods for solving two-stage distributionally robust optimization problems},
  Computational Optimization and Applications, 78 (2021), pp.~205--238.

\bibitem{Drug-discovery}
{\sc E.~CORDIS}, {\em Machine learning ledger orchestration for drug
  discovery}.
\newblock \url{https://featurecloud.eu/about/our-vision/}.
\newblock Retrieved: July. 2022.

\bibitem{medical-record}
{\sc FeatureCloud}, {\em Featurecloud: Our vision, 2022}.
\newblock \url{https://cordis.europa.eu/project/id/831472}.
\newblock Retrieved: July. 2022.

\bibitem{star_network}
{\sc A.~Froehlich}, {\em Definition: star network}.
\newblock
  \url{https://www.techtarget.com/searchnetworking/definition/star-network}.
\newblock Retrieved: July. 2022.

\bibitem{hiriart1993convex}
{\sc J.-B. Hiriart-Urruty and C.~Lemar{\'e}chal}, {\em Convex analysis and
  minimization algorithms II}, vol.~306, Springer science \& business media,
  1993.

\bibitem{jacob2009group}
{\sc L.~Jacob, G.~Obozinski, and J.-P. Vert}, {\em Group lasso with overlap and
  graph lasso}, in Proceedings of the 26th annual international conference on
  machine learning, 2009, pp.~433--440.

\bibitem{javanbakht2014risk}
{\sc P.~Javanbakht and S.~Mohagheghi}, {\em A risk-averse security-constrained
  optimal power flow for a power grid subject to hurricanes}, Electric Power
  Systems Research, 116 (2014), pp.~408--418.

\bibitem{kairouz2021advances}
{\sc P.~Kairouz, H.~B. McMahan, B.~Avent, A.~Bellet, M.~Bennis, A.~N. Bhagoji,
  K.~Bonawitz, Z.~Charles, G.~Cormode, R.~Cummings, et~al.}, {\em Advances and
  open problems in federated learning}, Foundations and Trends{\textregistered}
  in Machine Learning, 14 (2021), pp.~1--210.

\bibitem{kolda2009tensor}
{\sc T.~G. Kolda and B.~W. Bader}, {\em Tensor decompositions and
  applications}, SIAM review, 51 (2009), pp.~455--500.

\bibitem{kuhn2019wasserstein}
{\sc D.~Kuhn, P.~M. Esfahani, V.~A. Nguyen, and S.~Shafieezadeh-Abadeh}, {\em
  Wasserstein distributionally robust optimization: Theory and applications in
  machine learning}, in Operations research \& management science in the age of
  analytics, Informs, 2019, pp.~130--166.

\bibitem{Lan15Bundle}
{\sc G.~Lan}, {\em Bundle-level type methods uniformly optimal for smooth and
  nonsmooth convex optimization}, Mathematical Programming, 149 (2015),
  pp.~1--45.

\bibitem{lan2016gradient}
\leavevmode\vrule height 2pt depth -1.6pt width 23pt, {\em Gradient sliding for
  composite optimization}, Mathematical Programming, 159 (2016), pp.~201--235.

\bibitem{LanBook}
\leavevmode\vrule height 2pt depth -1.6pt width 23pt, {\em First-order and
  stochastic Optimization Methods for Machine Learning}, Springer-Nature, 2020.

\bibitem{lan2021graph}
{\sc G.~Lan, Y.~Ouyang, and Y.~Zhou}, {\em Graph topology invariant gradient
  and sampling complexity for decentralized and stochastic optimization}, arXiv
  preprint arXiv:2101.00143,  (2021).

\bibitem{LanZhou18RPDG}
{\sc G.~Lan and Y.~Zhou}, {\em An optimal randomized incremental gradient
  method}, Mathematical Programming, 171 (2018), pp.~167--215.

\bibitem{larsson2014MIMO}
{\sc E.~G. Larsson, O.~Edfors, F.~Tufvesson, and T.~L. Marzetta}, {\em Massive
  {MIMO} for next generation wireless systems}, IEEE communications magazine,
  52 (2014), pp.~186--195.

\bibitem{lemarechal1995new}
{\sc C.~Lemar{\'e}chal, A.~Nemirovskii, and Y.~Nesterov}, {\em New variants of
  bundle methods}, Mathematical programming, 69 (1995), pp.~111--147.

\bibitem{mairal2011lasso}
{\sc J.~Mairal, R.~Jenatton, G.~Obozinski, and F.~Bach}, {\em Convex and
  network flow optimization for structured sparsity.}, Journal of Machine
  Learning Research, 12 (2011).

\bibitem{Markowitz}
{\sc H.~Markowitz}, {\em Portfolio selection}, The Journal of Finance, 7
  (1952), pp.~77--91.

\bibitem{martinez2015risk}
{\sc G.~Mart{\'\i}nez and L.~Anderson}, {\em A risk-averse optimization model
  for unit commitment problems}, in 2015 48th Hawaii International Conference
  on System Sciences, IEEE, 2015, pp.~2577--2585.

\bibitem{smart-manufacturing}
{\sc musketeer}, {\em Musketeer: about, 2022}.
\newblock \url{https://musketeer.eu/project/}.
\newblock Retrieved: July. 2022.

\bibitem{nemirovsky1983problem}
{\sc A.~S. Nemirovsky and D.~B. Yudin}, {\em Problem complexity and method
  efficiency in optimization.}, John Wiley UK/USA, 1983.

\bibitem{Nes83}
{\sc Y.~Nesterov}, {\em A method for unconstrained convex minimization problem
  with the rate of convergence o (1/k\^{} 2)}, in Doklady AN USSR, vol.~269,
  1983, pp.~543--547.

\bibitem{nesterov2003introductory}
\leavevmode\vrule height 2pt depth -1.6pt width 23pt, {\em Introductory
  lectures on convex optimization: A basic course}, vol.~87, Springer Science
  \& Business Media, 2003.

\bibitem{Nestrov2004Smooth}
\leavevmode\vrule height 2pt depth -1.6pt width 23pt, {\em Smooth minimization
  of non-smooth functions}, Mathematical programming, 103 (2005), pp.~127--152.

\bibitem{ouyang2021lower}
{\sc Y.~Ouyang and Y.~Xu}, {\em Lower complexity bounds of first-order methods
  for convex-concave bilinear saddle-point problems}, Mathematical Programming,
  185 (2021), pp.~1--35.

\bibitem{parkvall20175G}
{\sc S.~Parkvall, E.~Dahlman, A.~Furuskar, and M.~Frenne}, {\em Nr: The new
  {5G} radio access technology}, IEEE Communications Standards Magazine, 1
  (2017), pp.~24--30.

\bibitem{total_variation_rudin1992nonlinear}
{\sc L.~I. Rudin, S.~Osher, and E.~Fatemi}, {\em Nonlinear total variation
  based noise removal algorithms}, Physica D: nonlinear phenomena, 60 (1992),
  pp.~259--268.

\bibitem{scaman2017optimal}
{\sc K.~Scaman, F.~Bach, S.~Bubeck, Y.~T. Lee, and L.~Massouli{\'e}}, {\em
  Optimal algorithms for smooth and strongly convex distributed optimization in
  networks}, in international conference on machine learning, PMLR, 2017,
  pp.~3027--3036.

\bibitem{scaman19}
{\sc K.~Scaman, F.~Bach, S.~Bubeck, Y.~T. Lee, and L.~Massouli{{\'e}}}, {\em
  Optimal convergence rates for convex distributed optimization in networks},
  Journal of Machine Learning Research, 20 (2019), pp.~1--31.

\bibitem{shapiro2014lectures}
{\sc A.~Shapiro, D.~Dentcheva, and A.~Ruszczy{\'n}ski}, {\em Lectures on
  stochastic programming: modeling and theory}, SIAM, 2014.

\bibitem{tibshirani2005graph_lasso}
{\sc R.~Tibshirani, M.~Saunders, S.~Rosset, J.~Zhu, and K.~Knight}, {\em
  Sparsity and smoothness via the fused lasso}, Journal of the Royal
  Statistical Society: Series B (Statistical Methodology), 67 (2005),
  pp.~91--108.

\bibitem{tol2009climate}
{\sc R.~S. Tol}, {\em The economic effects of climate change}, Journal of
  economic perspectives, 23 (2009), pp.~29--51.

\bibitem{tomioka2011tensor}
{\sc R.~Tomioka, T.~Suzuki, K.~Hayashi, and H.~Kashima}, {\em Statistical
  performance of convex tensor decomposition}, Advances in neural information
  processing systems, 24 (2011).

\bibitem{ye2016likelihood}
{\sc Z.~Wang, P.~W. Glynn, and Y.~Ye}, {\em Likelihood robust optimization for
  data-driven problems}, Computational Management Science, 13 (2016),
  pp.~241--261.

\bibitem{Reinsurance}
{\sc WeBank}, {\em {WeBank} and {Swiss Re} signed cooperation {MoU}}.
\newblock
  \url{http:https://www.fedai.org/news/webank-and-swiss-re-signed-cooperation-mou/}.
\newblock Retrieved: July. 2022.

\bibitem{zhang2019efficient}
{\sc Z.~Zhang, S.~Ahmed, and G.~Lan}, {\em Efficient algorithms for
  distributionally robust stochastic optimization with discrete scenario
  support}, SIAM Journal on Optimization, 31 (2021), pp.~1690--1721.

\bibitem{zhang2020optimal}
{\sc Z.~Zhang and G.~Lan}, {\em Optimal algorithms for convex nested stochastic
  composite optimization}, arXiv preprint arXiv:2011.10076,  (2020).

\bibitem{zhang2022solving}
\leavevmode\vrule height 2pt depth -1.6pt width 23pt, {\em Solving convex
  smooth function constrained optimization is as almost easy as unconstrained
  optimization}, arXiv preprint arXiv:2210.05807,  (2022).

\end{thebibliography}
\section{Appendix}

\begin{lemma}\label{lm:app-pd}
Let $\myfi: \R^n \rightarrow \R$ be a proper convex closed function and $\fistar$ be its Fenchel conjugate. The following computations are equivalent for all $\bar y \in X,  \piibar \in \R^n, \tau > 0:$
\begin{align}
&\piit \leftarrow \targmax_{\pii} \inner{\bar{y}}{\pii} - \fistar(\pii) - \tfrac{\tau}{2} \normsq{\pii -\piibar}, \label{equiv:dual}\\
&\piit \leftarrow \piibar + \tfrac{1}{\tau} (\bar y -\uy), \text{ where } \uy \leftarrow \targmin_{y} \myfi(y) + \tfrac{1}{2\tau} \normsq{\bar y + \tau \piibar - y}.\label{equiv:primal}
\end{align}

\end{lemma}

\begin{proof}
Let us fix an $i \in [m]$ and let $\bu:=\bar y + \tau\piibar$. Let $\piit$ be generated according to \eqref{equiv:primal}. Consider a Moreau envelop of $\myfi$  given by 
$g(u)= (\myfi \Box \tfrac{1}{2\tau}\|\cdot-y\|^2)(u) :=\inf_{y} \myfi(y) + \tfrac{1}{2\tau} \|u - y\|^2.$
Since $\myfi$ is convex, $g$ is convex and smooth over $\R^n$, thus $\partial g(\bu)$ is non-empty and unique.\\
Next, define $\bg(u):= \myfi(\uy) + \tfrac{1}{2\tau} \|u - \uy\|^2$ such that $g(u) \leq \bg(u)$.
Since $\uy := \inf_{y} \myfi(y) + \tfrac{1}{2\tau} \|\bu - y\|^2$ in \eqref{equiv:primal} implies 
$g(\bu) = \bg(\bu)$, the subgradient of $g$ at $\bu$ must be a subgradient of $\bg$, a dominating function, at $\bu$, i.e.,
$$\partial g(\bu) \subset \partial \bg(\bu) = \{\piit := \tfrac{1}{\tau}(\bu - \uy)\}.$$
Therefore $\piit = \grad g(\bu).$
Using the infimal convolution identity (c.f. Theorem 4.16 in \cite{Beck2017First}) 
$(g)^*(\pii) =(\myfi \Box \tfrac{1}{2\tau}\normsq{\cdot})^*(\pii)=\fistar(\pii) + \tfrac{\tau}{2}\normsq{\pii}\ \forall \pii$, the equivalence between maximization and sub-gradient evaluation, and the fact $\bu:=\bar y + \tau\piibar$, we get
\begin{align*}
\piit \in \partial g(\bar u) = \partial (\myfi \Box \tfrac{1}{2\tau}\normsq{\cdot})(\bu) 
&\Leftrightarrow \piit \in \targmax_{\pii} \inner{\bu}{\pii} - (\myfi \Box \tfrac{1}{2\tau}\normsq{\cdot})^*(\pii)\\
&\Leftrightarrow \piit \in  \targmax_{\pii \in \Pii} \inner{\bar y + \tau\piibar}{\pii} - \fistar(\pii) - \tfrac{\tau}{2} \normsq{\pii}\\
&\Leftrightarrow \piit \in  \targmax_{\pii \in \Pii} \inner{\bar y}{\pii} - \fistar(\pii) - \tfrac{\tau}{2} \normsq{\pii -\piibar}.
\end{align*}
\end{proof}
\vgap
\textblue{
\subsection{Efficient Implementations for Proximal Mappings}
Since $X$ is either a box or $\R^n$, the $x$-prox mappings are implemented with closed-form solutions. The $\pi$-prox mappings also admit closed-form solutions. For the linear regression problem in \eqref{eq:linear_reg}, the equivalent primal gradient computation amounts to a matrix-vector multiplication. For the two-stage stochastic program in \eqref{eq:n_test}, since the simple complete recourse is assumed \cite{zhang2019efficient}, $\Pi_i$ is a box and the projection onto it can implemented by component-wise thresholding.

The $p$-proximal update are implemented with binary searches and some basic matrix operations. When $\rho$ is a $\delta$-CV@R risk measure, $P$ can expressed as the intersection of an equality constraint and a box constraint \cite{shapiro2014lectures}. By dualizing the coupling equality constraint, we arrive at an equivalent two-level optimization formulation for the $\pt$-prox mapping.     
\begin{align}\label{eq:cvar-p-proximla}
\begin{split}
\pt = \targmax_{p}&\ \inner{p}{g} - \tfrac{1}{2} \normsq{p - \ptt} \quad \Leftrightarrow\quad \pt = \tmin_{\lambda \in \R }\targmax_{p}\ \inner{p}{g} - \tfrac{1}{2} \normsq{p - \ptt} + \lambda (\tsum_{i=1}^m p_i - 1)\\
s.t.&\   0 \leq  p_i \leq 1/(m \delta)   \quad \quad \quad \quad \quad \quad \quad \quad \quad \quad \quad \quad  \quad  \quad s.t.\   0 \leq  p_i \leq 1/(m \delta).  \\
&\ \tsum_{i=1}^m p_i = 1
\end{split}
\end{align}
For a fixed $\lambda$, the  inner solution $p(\lambda)$ can be computed via a component-wise vector thresholding and the optimal $\lambda^t$ is characterized by the root condition $\tsum_{i=1}^m p_i(\lam^t) - 1 =0$. Since $p(\lam)$ is a monotonically non-decreasing function of $\lam$, an accurate approximation to $\lam^t$ and hence $p^t$ can be found by a binary search on $\lam$.  Next, when $\rho$ is the risk measure induced by the $\chi^2$ ambiguity set, we can dualize the $\chi^2$ constraint to express the $\pt$-prox mapping equivalently as follows.  
\begin{align*}
\pt = \argmax_{p \geq 0 }&\ \inner{p}{g} - \tfrac{1}{2} \normsq{p - \ptt} \quad \quad \quad \Leftrightarrow\quad \pt = \min_{u \in \R_+ }\argmax_{p \geq 0}\ \inner{p}{g} - \tfrac{1}{2} \normsq{p - \ptt} \\
s.t.&\   \norm{p - [1/m, \ldots, 1/m]}^2 \leq r    \quad \quad \quad \quad \quad \quad \quad \quad \quad \quad \quad \quad     + u (\norm{p - [1/m, \ldots, 1/m]}^2 - r) 
 \\
&\ \tsum_{i=1}^m p_i = 1.  \quad \quad \quad \quad \quad \quad \quad \quad \quad \quad  \quad \quad \quad \quad \quad \quad   s.t.\   \tsum_{i=1}^m p_i = 1.
\end{align*}
For a fixed $u$, the inner solution $p(u)$ above can be solved similarly to \eqref{eq:cvar-p-proximla}. A sufficient optimality condition for $u^t$ is the KKT condition, i.e. either $u^t = 0 $, or $u^t > 0$ and  $\norm{p(u^t) - [1/m, \ldots, 1/m]}^2 - r = 0$. So an accurate $u^t$ and hence $p^t$
can be found by a binary search on $u$.

}

\end{document}